\documentclass[11pt,leqno,fleqn]{amsart}  

\usepackage{amsmath,amstext,amsthm,amssymb,amsxtra}
\usepackage[normalem]{ulem}
\usepackage{txfonts} 
\usepackage[T1]{fontenc}
\usepackage{lmodern}

 \usepackage{euler}   

\usepackage{tikz} \usepackage{paralist}

\usepackage{mathtools}
\mathtoolsset{showonlyrefs,showmanualtags}

\usepackage{hyperref} 
\hypersetup{
    colorlinks=true,       
    linkcolor=blue,          
    citecolor=magenta,        
    filecolor=magenta,      
    urlcolor=cyan           
}

\usepackage[msc-links,backrefs]{amsrefs}  

\allowdisplaybreaks

\theoremstyle{plain} 
\newtheorem{lemma}[equation]{Lemma} 
\newtheorem{proposition}[equation]{Proposition} 
\newtheorem{theorem}[equation]{Theorem} 
\newtheorem{corollary}[equation]{Corollary} 

\newtheorem{question}[equation]{Question}
\newtheorem{priorResults}{Theorem}

\theoremstyle{definition}
\newtheorem{definition}[equation]{Definition} 

\theoremstyle{remark}

\numberwithin{equation}{section}

\title[Two Weight Hilbert] {The Two Weight Inequality for the Hilbert Transform: A Primer}
\author[M. T. Lacey]{Michael T. Lacey}   

\dedicatory{In memory of my father, H.~Elton Lacey} 

\address{ School of Mathematics, Georgia Institute of Technology, Atlanta GA 30332, USA}
\email {lacey@math.gatech.edu}
\thanks{Research supported in part by grant NSF-DMS 0968499,  a grant from the Simons Foundation (\#229596 to Michael Lacey), 
and the Australian Research Council through grant ARC-DP120100399.  The author benefited from two research programs, 
first `Operator Related Function Theory and Time-Frequency Analysis' at the Centre for Advanced Study at the Norwegian Academy of Science and Letters in Oslo during 2012---2013, and second  `Interactions between Analysis and Geometry' program at IPAM, UCLA, 2013.} 

\begin{document}
	\begin{abstract}
	Given a pair of weights $ w, \sigma $, the two weight inequality for the Hilbert transform 
	is of the form $ \lVert H (\sigma f)\rVert_{L ^2 (w)} \lesssim \lVert f\rVert_{L ^2 (\sigma )}$. 
	Recent work of  Lacey-Sawyer-Shen-Uriarte-Tuero and Lacey have established a conjecture of Nazarov-Treil-Volberg, giving a real-variable characterization of which pairs of weights this inequality holds, provided the pair of weights do not share a common point mass. 
	In this paper, the characterization is proved, collecting details from across several papers;
	 counterexamples are detailed; and  areas of application are indicated.      
\end{abstract}
	
	\maketitle  
\setcounter{tocdepth}{1}	
\tableofcontents 	
	
\section{Introduction} 

By a \emph{weight} we mean a non-negative Borel  locally finite measure, typically on $ \mathbb R $. 
We consider the \emph{two weight inequality for the Hilbert transform} for a pair of weights $ w, \sigma $ on $ \mathbb R $: 
\begin{equation} \label{e:N}
  \lVert H (f \cdot \sigma )\rVert_{ L ^2 (w)} 
\le \mathscr N \lVert f\rVert_{L ^2 (\sigma )} \,. 
\end{equation}
Here, $ \mathscr N $ denotes the best constant in the inequality. 
And $H \nu  (x)$ is the Hilbert transform of $ \nu $
\begin{equation}  \label{e:beta} 
H  \nu (x) \coloneqq  \int\frac {\nu (dy)} {y-x} \,. 
\end{equation}
We do not insist on the existence of the principal value, a point addressed in \S~\ref{s:principal_value}. 

The central question is  then a real-variable characterization of the inequality \eqref{e:N}.  In the special case that 
the pair of weights $ \sigma $ and $ w $ do not share a common point mass, this was supplied in three papers, 
one of Lacey-Sawyer-Shen-Uriarte-Tuero \cite{MR3285857} with the refinement of Hyt\"onen \cite{13120843}, and  another of the present author \cite{MR3285858},  
answering a beautiful conjecture of Nazarov-Treil-Volberg \cite{V}. 

\begin{theorem}\label{t:noCommon} 
Define two positive constants $ \mathscr A_2$ and $ \mathscr T$ as the best constants in the inequalities below, 
uniform over intervals $ I$, and with respect to interchanging the roles of $ \sigma $ and $ w$. 
\begin{gather}\label{e:A2}
\frac {\sigma (I)} {\lvert  I\rvert }  \cdot  P (w \mathbf 1_{\mathbb R \setminus I} , I) \le   \mathscr A_2\,, 
\\ \label{e:T}
 \int _{I} H (\sigma \mathbf 1_{I}) ^2 \; d w \le \mathscr T ^2 \sigma (I)\,. 
\end{gather}
There holds $ \mathscr N \simeq  \mathscr H \coloneqq \mathscr A_2 ^{1/2} + \mathscr T$. 
\end{theorem}

The first condition is an extension of the Muckenhoupt $ A_2$ condition to a  `half Poisson condition with a hole.' 
The exact Poisson extension of $ \sigma $ to the upper half-plane is not needed, rather we use the approximation below, 
which is roughly the Poisson extension evaluated at the center of $ I$, and up into the half-plane the length of $ I$, see Figure~\ref{f:sP}. 
\begin{equation}\label{e:P}
P (\sigma ,I) \coloneqq  \int _{\mathbb R } \frac {\lvert  I\rvert } { (\lvert  I\rvert + \textup{dist} (x,I)  ) ^2 }\; \sigma (dx) \,. 
\end{equation}
The remaining conditions are referred to as the Sawyer-type \emph{testing conditions}, 
as Eric Sawyer first introduced these conditions into the two weight setting in his fundamental papers 
on the maximal function \cite{MR676801}, and later the fractional and Poisson integral operators 
\cite{MR930072}.  It is well-known that the $ A_2$ condition  \eqref{e:A2} is necessary for the two weight inequality, and it is obvious that the testing conditions are necessary.  Thus, the substance of the Theorem above concerns the sufficiency of the $ A_2$ and testing inequalities for the 
norm inequality.  

This Theorem is a central result in the non-homogeneous harmonic analysis, as founded in a sequence of influential papers of Nazarov-Treil-Volberg \cites{NTV4,NTV2,MR1998349}.
The proof of the theorem is  involved, encompassing arguments and points of view that were spread across 
several papers \cites{10031596,MR3285857,MR3285858,11082319}.  
Finally, the interest in the two weight inequality is well-motivated by applications to operator theory, model spaces, and spectral theory, 
themselves spread across additional  papers.  

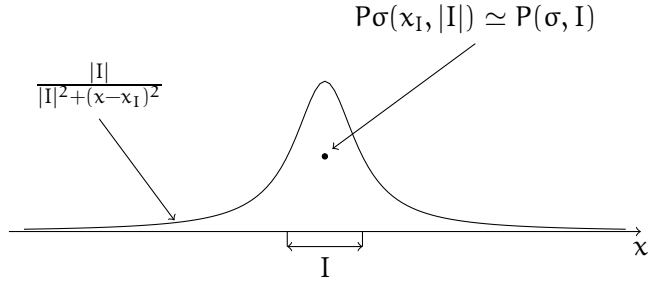
\begin{figure}
\begin{tikzpicture}[domain=0:8, samples=125]
  \draw[->] (-0.2,0) -- (8.2,0) node[right,below] {$x$};
  \draw  plot (\x,{ .5/(.25+(\x-4)^2});      \draw (4.5,0) -- (4.5,-.2);    	    \draw (3.5,0) -- (3.5,-.2); 
    \draw[<->] (3.5, -.2) -- (4.5, -.2) node[midway,below] {$ I$}; 
    \filldraw  (4,1) circle (1pt); 
    \draw[<-] (4.1,1.1) -- (6,2.5) node [above] {$ P \sigma (x_I, \lvert  I\rvert ) \simeq P (\sigma , I)$}; 
    \draw[<-] (2,.15) -- (1,1.5) node[above] {$ \frac {\lvert  I\rvert } { \lvert  I\rvert ^2 + (x-x_I) ^2  }$}; 
\end{tikzpicture}
\caption{The value of $ P (\sigma , I)$ is approximately the Poisson extension of $ \sigma $ evaluated at point in the upper half-plane given by  the center of $ I$, and the length of $ I$.}
\label{f:sP}
\end{figure}

The point of this paper is to  
\begin{enumerate}[(a)]
\item  state and prove  the Theorem, in all detail. 
\item give the proof under the influential \emph{pivotal condition}, which serves to highlight where the difficulties arise in the general case; 
\item  collect  relevant, explicit,  counterexamples; 
\item  give complements and extensions of the theorem, and the proof techniques;  
\item and point to areas of  applications.
\end{enumerate}
Sections  proceed directly towards proofs, but many conclude with some context and discussion.  
The proof is entirely elementary, assuming only the well known facts about martingale differences.

\subsection{An Overview of the Proof}

The result is an \emph{individual two weight inequality.}  It characterizes the boundedness of the Hilbert transform, 
and no other operator.  Therefore, particular properties of this transform must guide the proof. 
The elementary examples of these are the \emph{monotonicity principle,}  Lemma~\ref{mono}, valid for all pairs of weights, and then the 
\emph{energy inequality}, Lemma~\ref{l:energy},  valid under the assumption of interval testing and the $ A_2 $ condition.  
These properties  are a last vestige of positivity:  The kernel $ \frac 1y$ is monotone increasing on $ \mathbb R \setminus \{0\}$. 
This feature will deliver to us the \emph{energy inequality}; finding it, and  unlocking its secrets is the key to the proof. 

The main line of the argument begins with   the bilinear form $ \langle H _{\sigma } f,g \rangle _{w}$.
It's decomposition is made to `regularize' all four quantities in the expression, the two functions $ f$ and $ g$, 
as well as the `irregularities' of the pair of weights, as expressed by the energy inequality.  
Only half of the decomposition needs to be specified, due to the self-dual nature of the question, and some of these 
considerations are familiar to experts in both the $ T1$ and the $ Tb$ theorems. 
But the underlying difficulties do not have any classical analog.

The proof strategy is outlined in Figure~\ref{f:proof}.  
The passage to the `triangular forms'  in Lemma~\ref{l:above} is a rather standard step in many $ T1$-type theorems.   
The \emph{Calder\'on-Zygmund stopping data} defined in \S\ref{s:global} is   the foundational tool. 
It (a) controls the values of certain telescoping sums  of martingale differences; 
(b) regularizes the weights, from the point of view of the energy inequality;
and (c) allows the use of the  \emph{quasi-orthogonality} argument, an important simplification.  
The  triangular forms are of a `local' and a `global' form, and have dual forms as well. 
There are two steps in the analysis, a  `global to local' reduction in \S\ref{s:global}, and an analysis of the `stopping form'  in the \S\ref{s:stop}.   

The stopping data is essential to the  `global to local reduction' in Theorem~\ref{t:aboveCorona}.  
A simple appeal to the testing condition, allows an application of the monotonicity principle to rephrase the inequality  in this Theorem as a certain two-weight inequality for the Poisson integral.
In this inequality, the Poisson integral maps functions on $ \mathbb R $ to those on $ \mathbb R ^2 _+$. 
The weight on $ \mathbb R $ is, say, $ \sigma $.  The weight on $ \mathbb R ^2 _+$ is then derived from $ w$ in a specific fashion from the stopping data, and hence depend upon $ f$ and the pair of weights.   
But the Poisson operator is a positive operator, and one has a quite adequate understanding 
of their two weight inequalities. 
We directly implement this understanding, without proving any more general result.

The local term is then dominated by the analysis of the \emph{stopping form} \eqref{e:stop}.  
This is again a familiar object, to experts in $ T1$ theorem, addressed by \emph{ad hoc} off-diagonal estimates, 
which absolutely do not apply in  the current context. 
Control of the irregularities of the weights is now the main point, complicated by the fact that the stopping form 
is not  intrinsically defined.
A notion of `size' is introduced---it serves as an approximate of the operator norm of the  stopping form, 
and again is most naturally defined in terms of  a measure on $ \mathbb R ^2 _+$, derived from the two given weights. 
The \emph{size lemma}, Lemma~\ref{l:Decompose},  decomposes a stopping form into 
constituent parts.  Those of large size have a simpler form, which allows one to estimate their operator norm by 
size. What is left has smaller size, and so one can recurse.   This argument relies heavily on the Hilbertian structure of the question.

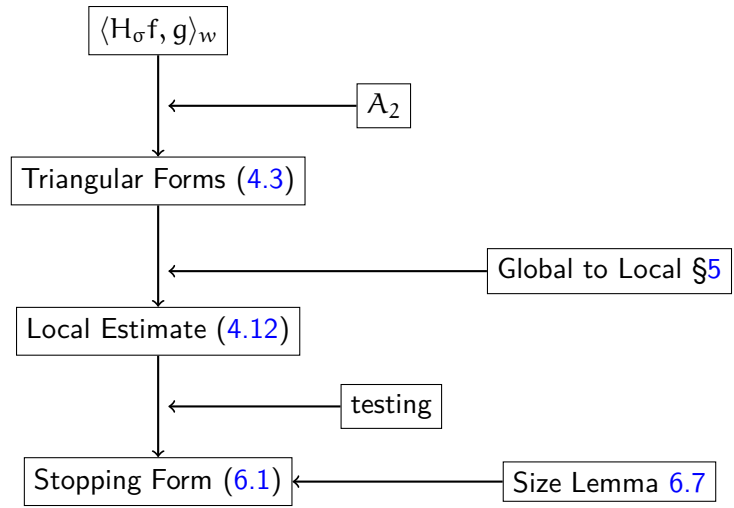
\begin{figure}[bottom] 
\begin{tikzpicture}
\node[rectangle,draw] (h) at (0,0) { $ \langle H _{\sigma }f,g \rangle_{w}$};
\node (a2) at (0,-1) {};   
\node[rectangle,draw] (T) at  (0,-2)  {Triangular Forms \eqref{e:ABOVE}}; 
\node (global) at (0,-3.2)   {} ;
\node[rectangle,draw] (L) at (0,-4)  {Local Estimate \eqref{e:BF}}; 
\node (t) at (0,-5)   {} ;
\node[rectangle,draw] (S) at   (0,-6)  {Stopping Form \eqref{e:stop}}; 

\node[rectangle,draw]  (A2) at (3,-1) {$ A_2$}; 
\node[rectangle,draw]  (F) at (6,-3.2) {Global to Local \S\ref{s:fe}}; 
\node[rectangle,draw]  (s) at (6,-6) {Size Lemma \ref{l:Decompose}}; 
\node[rectangle,draw]  (test) at (3.1,-5) {testing}; 

\draw[->,thick] (A2) -- (a2); 
\draw[->,thick] (h) -- (T);
\draw[->,thick] (T) -- (L);
\draw[->,thick] (L) -- (S);
\draw[->,thick] (F) -- (global) ;
\draw[->,thick] (s) -- (S);
\draw[->,thick] (test) -- (t);
\end{tikzpicture}
\caption{A schematic tree of the proof of the main theorem.}
\label{f:proof}
\end{figure}

Some readers will have noticed that a very common set of objects, Carleson measures, are not mentioned, and indeed, they do not appear in the 
proof at all.
The wide spread prevalence of Carleson measures in $ T1 $ theorems can be traced to two  facts, first  that associated paraproducts operators are the  principle 
obstacle to a simple proof, and second, the paraproduct operators have an essentially canonical form.  
In this theorem, neither of these facts hold, and so we have abandoned the notions  of Carleson measures and paraproducts.  

Carleson measures are also used to, indirectly, control the sums of martingale differences. 
Rather than this, we use the simpler method of  stopping data, as described in \S\ref{s:global}.

\subsection{The $ A_2$ Theory}

\begin{figure}
\begin{tikzpicture}
\draw[->] (-5,0) -- (5,0);  
\draw[->] (0,-.5) -- (0,3.5);  
\draw plot[domain=0:5,samples=80,smooth] (\x,{sqrt{\x+.05}});
\draw plot[domain=0:5,samples=80,smooth] (-\x,{sqrt{\x+.05}}); 
\draw plot[domain=0.3:2.2,samples=40,smooth] (\x^2,1/\x); 
\draw plot[domain=0.3:2.2,samples=40,smooth] (-\x^2,1/\x); 
\draw (4.5,.75) node {$ \sigma $}; 
\draw (4.5,2.5) node {$ w $}; 
\end{tikzpicture}
\caption{For $ 0< \epsilon < 1$, the function $ w (x) = \lvert  x\rvert ^{1- \epsilon } $ is an $ A_2$ weight. It and the dual weight $ \sigma (x) = \lvert  x\rvert ^{\epsilon -1} $ are graphed above.  One can check that $ [w] _{A_2} \simeq \epsilon ^{-1} $.} 
\label{f:A2}
\end{figure}
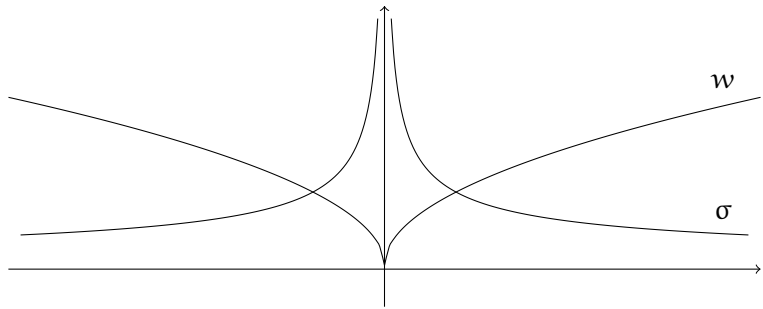

The classical case of an $ A_2$ weight corresponds to the case of $ w (dx) = w (x) dx $, and $ w (x) >0$ a.e. Moreover, the weight $ \sigma $ also has density given by  $ \sigma (x) \coloneqq  w (x) ^{-1} $.  It is assumed that both $ w$ and $ \sigma $ are locally integrable, 
so that they are both weights.  See Figure~\ref{f:A2}.  Note that $ w (x) \cdot \sigma (x) \equiv 1$.  The Muckenhoupt $ A_2$ condition asserts that this 
same equality approximately holds, uniformly over location and scale. 
\begin{equation*}
[w] _{A_2} \coloneqq  \sup _{I} \frac {w (I)} {\lvert  I\rvert } \cdot \frac {\sigma (I)} {\lvert  I\rvert } < \infty \,. 
\end{equation*}
These are `simple' averages.  
This condition  is equivalent to the uniform norm bound  on $ L ^2 (w)$ for  the class of simple averaging operators 
\begin{equation*}
f \mapsto  \frac 1 {\lvert  I\rvert } \int _{I} f \; dx \cdot \mathbf 1_{I} \,, \qquad \textup{$ I$ is an interval.}
\end{equation*}
From this condition flows a rich theory, including the boundedness of all Calder\'on-Zygmund operators.  
The classical result of Hunt-Muckenhoupt-Wheeden \cite{MR0312139} states that $ w $ in in $ A_2$ if and only if the Hilbert 
transform maps $ L ^2 (w) $ to $ L ^2 (w)$.  By a basic change of variables argument, first noted by Sawyer \cite{MR676801},  this is equivalent to $ H _{\sigma }$ mapping  $ L ^2 (\sigma )$ to  $ L ^2 (w)$.  
Stefanie Petermichl \cite{MR2354322} quantified the Hunt-Muckenhoupt-Wheeden theorem as follows.   

\begin{priorResults}\label{t:stef}  A weight  $ w \in A_2$ if and only if $ H $ is bounded from $ L ^2 (w)$ to $ L ^2 (w)$, and moreover the constant $ \mathscr N$ in \eqref{e:N} satisfies  $ \mathscr N \simeq [w] _{A_2}$.  
\end{priorResults}

To place this result in the context of our main result, it is classical and easy to see that the Poisson $ A_2$ characteristic 
satisfies $ \mathscr A_2 \lesssim [w] _{A_2} ^2 $.  And, using the remarkable Haar shift representation of the Hilbert transform 
due to Petermichl \cite{MR1756958}, one can check that the testing condition satisfies $ \mathscr T \lesssim [w] _{A_2}$.  
This is what Petermichl's original proof did.  
All existing proofs of Petermichl's Theorem (see \cites{MR2657437,MR3085756,2912709,MR3085756}) depend ultimately on known Lebesgue measure estimates for the Hilbert transform, or closely related operators. For instance, \cites{150105818,MR3085756} use the weak-$L^1 (dx)$ bound for sparse shift operators.  
Estimates of these type are irrelevant for the two weight theorem.  

It is perhaps worth emphasizing that the powerful Haar shift technique of Petermichl, even with its impressive extension by 
Hyt\"onen \cite{2912709}, seems to be of little use in the general two weight problem.  There are two obstacles: 
Firstly, in order to use it, one must essentially have control on a Haar shift operator, independently of how the grid defining the shift is defined. 
The resulting  condition on the pair of weights is  more subtle than the two weight inequality for the 
Hilbert transform.  Secondly, one should recover the energy inequality of Lemma~\ref{l:energy}.  But, the energy of any fixed Haar shift is zero, 
and indeed, the two weight inequality for Haar shift operators  \cite{MR2407233} has just a few difficulties in its proof.

By the \emph{$ A_2$ Theorem}, it is meant the linear in $ A_2$ bound for all Calder\'on-Zygmund operators.  
This result, pursued by many, and established by Hyt\"onen \cite{2912709}, has many points of contact with the subject of this note. 
But, we refer the reader to  \cite{MR3204859} and references there in for more information, and 
see \cite{150105818} for what is arguably the most elementary proof. ?

\smallskip 

In the $ A_2$ theory, it is essential that $ w (x) >0 $ a.e.  Suppose one relaxes this condition to $ w (x)$ is positive on a measurable set $ E \subsetneq \mathbb R $, and define $ \sigma (x)$ to be supported on $ E$, and equal to $ w (x) ^{-1}$.  One can then ask if the 
Hilbert transform is bounded for this pair of weights, and Theorem~\ref{t:noCommon} applies here.  This question is an instance of the 
non-homogeneous $ A_2$ theory advocated by A.~Volberg.    One can hope that  specificity in the way the weights are prescribed 
could introduce some additional simplifications in the characterization of the two weight inequality in this setting. 
But, none has yet been found.

\subsection{The Individual Two Weight Problem}

Given an operator $ T$, the \emph{individual $ L ^{p}$ two weight inequality  for $ T$} is the inequality 
\begin{equation} \label{e:TT}
\lVert T _{\sigma } f \rVert_{L ^{p} (w)} \le \mathscr N_T \lVert f\rVert_{L ^{p} (\sigma ) } \,. 
\end{equation}
Here and throughout we use the notation $ T _{\sigma } f \coloneqq  T (\sigma f)$. 
We understand that $ T$ applied to a signed measure  $ \sigma \cdot f$ should make sense. 
And, the inequality above is the preferred form of the inequality as duality is expressed in the natural way: 
The inequality \eqref{e:TT} is equivalent to 
\begin{equation*}
\lVert T _{w }  ^{\ast} g \rVert_{L ^{p'} (\sigma )} \le \mathscr N_T \lVert g\rVert_{L ^{p'} (w ) } \,. 
\end{equation*}
The question is then to characterize the pairs of weights for which \eqref{e:TT} holds.  

This specificity of the question is of interest for a few canonical operators, ones for which the corresponding two 
weight inequality will naturally present itself.  
The leading examples of this are, for positive operators, the Hardy operator by Muckenhoupt \cite{MR0311856}, 
the maximal function, Sawyer's Theorem of 1981 \cite{MR676801} and  Sawyer's 1988 theorem for the  fractional integrals \cite{MR930072}.  It is 
noteworthy that the two weight inequalities for the Hardy and the Poisson integral are used in the proof of our main theorem, 
as are various purely dyadic variants of these Theorems.  

It is interesting to that this is not only a chronological 
list, but it also reflects the depth of the results as well. The Hardy operator is easiest, characterized by an `$ A_2$-type condition,' 
as recalled in Theorem~\ref{p:hardy}.  
It was Sawyer's insight, however, that the maximal function  characterization requires a testing condition. 
The fractional integrals are harder still.
For the sake of comparison, let us state a special case of the result for the fractional integrals in one dimension. Besides Sawyer's results, one should also consult Casscante-Ortega-Verbitsky \cite{MR2269591}, 
as well as those of Vuorinen \cite{14122127}.  
Both results give a characterization in terms of testing conditions.  
And, while we state just one case of the general result, one should note that there is no Sobolev condition imposed on the $ L ^{p}$ indices.

\begin{priorResults} For two weights $ w, \sigma $, and $ 0< \alpha <1$, the operator $ R _{\sigma } f (x) \coloneqq  \int f (x-y) \frac {\sigma (dy)} {\lvert  y\rvert ^{\alpha } }$ maps $ L ^2 (\sigma )$ to $ L ^2 (w)$ if and only if the testing inequalities below hold. 
\begin{gather*}
\int _{I} R _{\sigma } (\mathbf 1_{I}) ^2 \; dw \le \mathscr T ^2 \sigma (I) \,, 
\qquad 
\int _{I } R _{w } (\mathbf 1_{I}) ^2 \; d \sigma  \le \mathscr T ^2 w  (I) \,.  
\end{gather*}
Moreover the norm of the operator is equivalent to $ \mathscr T$, the best constant in the inequalities above. 
\end{priorResults}

The analysis of the individual two weight inequality for  positive operators is much simpler, 
as is the case of dyadic operators.  For certain non-positive  dyadic operators, see the result of Nazarov-Treil-Volberg \cite{MR2407233}, and the much more recent works of Vuorinen \cites{14122127,150405759}.  
These results have found significant interest, due to the Haar shift operators of Petermichl \cite{MR1756958},  the remarkable 
median inequality of Lerner \cite{MR2721744} and its extension in \cite{150105818}, and the Hyt\"onen representation theorem \cite{2912709}. 

The Hilbert transform is the first non-positive continuous operator for which the individual two weight problem has been solved. 
And, one would only ever expect that the solution would be of interest (or even possible) for a few canonical choices of 
operators, such as Hilbert, Cauchy and Riesz transforms. 
Foundational to the solution for the Hilbert transform is the monotonicity of the kernel.  
No other canonical choice  will satisfy such a simple condition.  
For a special case of the Cauchy transform \cite{13104820} one can make progress. But the case of Riesz transforms is much harder \cites{13126163,14010467}.  

The individual two weight question makes sense for any $ 1< p < \infty $, and there are characterizations in this, and other 
off-diagonal cases for positive operators.   For dyadic analogs of singular integrals Vuorinen \cite{14122127} 
has shown that these inequalities can be characterized by \emph{quadratic} testing conditions. 
Also see \cite{150705570}.  
The extension of this characterization to the setting of the Hilbert transform  is challenging.

\subsection{The Hilbert Transform}

The two weight inequality for the Hilbert transform was addressed as early as 1976 by Muckenhoupt and Wheeden \cite{MR0417671}.\footnote{In particular, they noted that the simple $ A_2$ condition was not sufficient for the boundedness of the Hilbert transform, and conjectured that half-Poisson $ A_2$ conditions would be sufficient, an indication of the powerful sway held by the Muckenhoupt $ A_2$ condition in the early years of the weighted theory.}%
But, it received much wider recognition as an important problem with the  1988 work of Sarason \cite{MR1038352}.  The latter was part of important sequence of investigations that identified 
de Branges  spaces as an essential tool in operator theory.  
His question concerning the composition of Toeplitz operators, see \S\ref{s:sarason}, was raised therein, and 
advertised again in  \cite{sarasonConj}.  
This question related the individual two weight problem for the Hilbert transform to a profound question from operator theory.  

While not stated in the language of the Hilbert transform, Sarason wrote that it was `tempting' to conjecture that the full Poisson $ A_2$ condition would be sufficient for the two weight inequality.  
In an important development, F.~Nazarov \cite{N1} showed that this was not the case.  
The two weight problem was seen to be  important to  Model spaces, namely certain embedding questions for Model spaces can be 
realized as a two weight inequality for the Hilbert transform.  
In particular, a more delicate counterexample was developed by Nazarov-Volberg \cite{NV} to disprove a 
conjectured characterization of the Carleson measures for a model space.   
The Nazarov counterexample was also used by Nikol{\cprime }ski{\u \i }-Treil  \cite{MR1945291}, in the context of spectral theory.  

The Nazarov counterexample is by way of a Bellman function approach. 
In \S\ref{s:examples}, we give an explicit example.     It is worth noting that in Sarason's question, the weights have a density $ \lvert  f\rvert ^2  $, for analytic $ f$, and the subharmonicity could be an important part of the problem.  But, in the context of model spaces, completely singular arbitrary measures can arise.  
In \S\ref{s:examples},  one of the weights is uniform measure on a  Cantor set.  

 Nazarov-Treil-Volberg were creating the field of non-homogeneous Harmonic Analysis, in a series of ground-breaking  papers \cites{NTV4,NTV2,MR1998349}.  
Their work, and a revitalization of the perspective of Eric Sawyer from the 1980's, lead them to conjecture the characterization proved in this paper. 
Moreover, their influential proof strategy, devised in \cites{10031596,V}, lead to a verification of the conjecture in the case that both weights were doubling.   This paper uses their strategy, with several additional features.  
At the same time, their approach is generic, in that it applies to general Calder\'on-Zygmund operators. 
Specific properties of the Hilbert transform had to be used in the characterization. 
These properties were identified in \cites{10014043,11082319,MR3285857,MR3285858}, and the more precise description of what 
was accomplished at each stage is spread out throughout the paper.

\subsection{The Circle}

The two weight inequality has an equivalent formulation on the circle, which we formulate now. 
Given two weights $ w ,\sigma $ on the circle group $ \mathbb T \equiv \mathbb R /2 \pi  \mathbb Z $, we 
consider the norm inequality 
\begin{equation}\label{e.Hcircle}
\int _{\mathbb T }  \biggl\lvert \int f (y) \cdot \cot 
\bigl(\frac {x-y} 2  \bigr)
\; \sigma (dy)  \biggr\rvert 
^2 \; d w \le \mathscr N ^2 \lVert f\rVert _{L ^2 (\mathbb T , \sigma )} ^2 . 
\end{equation}
This is abbreviated to $ \lVert H ^{\mathbb T } _{\sigma } f \rVert _{L ^2 (w)} \le \mathscr N \lVert f\rVert _{L ^2 (\sigma ) }$.

\begin{theorem}\label{t:circle}
The inequality \eqref{e.Hcircle} holds if and only if the pair of weights below satisfy the 
conditions below and their duals. For all intervals $ I\subset \mathbb T $, with $ \lvert  I\rvert \le 1 $, 
there are finite constants $ \mathscr A_2$ and $ \mathscr T$, such that 
\begin{align}\label{e:A2-circ}
\frac {\sigma (I) } {\lvert  I\rvert } \cdot P ^{\mathbb T } (w \mathbf 1_{\mathbb T \setminus I}) (x_I, 1- \lvert  I\rvert ) &< \mathscr A_2 , 
\\ \label{e:test-circ}
\int _{I} \lvert  H ^{\mathbb T } _{\sigma } \mathbf 1_{I} \rvert ^2 \; d w 
\le \mathscr T ^2 \sigma (I). 
\end{align}
Moreover, letting $ \mathscr A_2$ and $ \mathscr T$ be the best constants in these inequalities and their 
duals, there holds $ \mathscr N \simeq \mathscr A_2 ^{1/2} + \mathscr T$.  
\end{theorem}

In \eqref{e:A2-circ}, the term $  P ^{\mathbb T } w(x_I, r )$ is the standard 
Poisson operator on the disk, evaluated at a point in the unit disk given by the center of the 
of the interval $ x_I$, and the radial factor $ r $.  

\medskip 

Let us indicate how to prove the theorem above from Theorem~\ref{t:noCommon}.  
Fix $ \sigma $ and $ w$ be two weights on $ \mathbb T $. 
Embed the weight $ w$ into $ [0,1]$ in the natural way, and call the resulting measure $ w'$. 
Place three copies of $ \sigma $ on the intervals $ [-1,0]$, $ (0,1]$ and $ (1,2]$, 
and call the resulting measure $ \sigma '$.  
Thus, $ \sigma ' $ and $ w'$ are two weights on $ \mathbb R $. 
It is clear that $ \sigma '$ and $ w'$ satisfy the Poisson $ A_2$ condition with holes on $ \mathbb R $.

For a function $ f \in L ^2 (\mathbb T ; \sigma  )$, let $ f'$ be three copies of $ f$ on the intervals 
$ [-1,0]$, $ (0,1]$ and $ (1,2]$.  
Viewing $ \mathbb T $ as $ [0,1]$, there is a subtle difference between $  H ^{\mathbb T } _{\sigma } f (x)$ 
and $ H _{\sigma ' } f '(x)$, the former computed on $ \mathbb T $, and the latter on $ \mathbb R $.  Namely 
\begin{equation*}
\lvert   H ^{\mathbb T } _{\sigma } - \tfrac 12 H _{\sigma '} f' (x)\rvert 
\lesssim \int _{-3  } ^{3}  \lvert  f' (y)\rvert   \cdot \lvert  x-y\rvert ^2 \; \sigma (dy) . 
\end{equation*}
It is easy to see that the $ A_2$ condition implies that the operator on the right is bounded. 
Hence, the testing conditions on $ \mathbb T$ imply those for $ w'$ and $ \sigma '$.  
Hence $ H _{\sigma '}$ maps $ L ^2 (\sigma ')$ to $ L ^2 (w')$. From that, we deduce the boundedness of $  H ^{\mathbb T } _{\sigma } $. 

\textbf{Cora Sadosky.}  Cora Sadosky and I met only a couple of times, which is a pity, since my research 
has been so strongly influenced by her passions and interests.  Her work with Cotlar on the $ L ^{p}$ variant 
of the Helson-Szeg\H o theorem is a beautiful complex variable result well beyond  the reach of  the current real-variable 
techniques.  Her  interest in Hankel forms on two and more complex variables has been my own for several years. 
And, in a number of small ways, I work to support more diversity in the profession, again following her lead.  

Cora Sadosky's family came up in 2005, during a three month stay in Argentina, in a antiquarian bookstore just a few steps from the Casa Rosada in  Buenos Aires.  The proprietor, upon hearing I was a mathematician, remembered his own 
youth and a compelling Professor Manuel Sadosky.  He remembered that the Professor had a daughter and 
asked after her.  This was the third or fourth conversation of this type I had in that lovely city! 
It is a privilege to work on the beautiful subject of mathematics. Even more so to have passion, and insights 
that others will carry forward.

\section{Preliminaries} 

\subsection{Principal Values}\label{s:principal_value}
We make no assertion about principal values of the Hilbert transform, and do not expect them to exist 
in the generality in which we are considering. One can then be concerned about how the definition 
is made.  There are a couple of different options.  One can impose some sort of truncation on the integrals, 
and the statements of the theorems are then understood to be uniform over all truncations.  
Many of the different possible truncations will be equivalent, since the $ A_2$ condition will hold, 
see \cite{MR3010121} for a general discussion of this issue.  
Alternatively, one can formally define 
\begin{equation*}
\langle H _{\sigma } f ,g  \rangle_w \coloneqq  
\int\int f (y) g (x) \frac {dy\,dx} {y-x}
\end{equation*}
for all $ f, g$ which have closed supports that are a positive distance apart, and extend $ H$ linearly from there. 

In our proof, all of the essential difficulties in the proof  arise when 
$ f$ and $ g$ have widely separated supports. The  definition of $ H _{\sigma }$ in this case is of course by the formula above.

\subsection{Dyadic Grids and Haar Functions}
A \emph{grid} is a collection $ \mathcal D$ of left closed, right open intervals so that for all $ I,J\in \mathcal D$, $ I\cap J = \emptyset , I, J$.  
Further say that $ \mathcal D$ is a \emph{dyadic grid} if for all integers $ n$, the collection $ \{ I \in \mathcal D \::\: \lvert  I\rvert= 2 ^{n} \}$ 
partitions $ \mathbb R $, aside from the endpoints of the intervals.  

For a sub collection $ \mathcal F $ of a dyadic grid $ \mathcal D$,  set $ \pi _{\mathcal F} I$ to be the minimal element of $ \mathcal F$ that contains $ I$; $ I$ need not be a member of $ F$. 
Set $ \pi _{\mathcal F} ^{1} I  $ to be the minimal member of $ \mathcal F$ that strictly contains $ I$, 
inductively define $ \pi _{\mathcal F} ^{t+1} I= \pi _{\mathcal F} ^{1} (\pi _{\mathcal F} ^{t}I)$.

Say that the collection $ \mathcal D$ is \emph{admissible for weight $ \sigma $} if $ \sigma $ does not have a point mass at any endpoint of an interval $ I\in \mathcal D$.

\subsection{Haar Functions}
Let  $ \mathcal D$ be admissible for $ \sigma $ be a weight on $ \mathbb R $.
If $ I\in \mathcal D$ is such that $ \sigma $ assigns non-zero weight to 
both children of $ I$, the associated Haar function is chosen to have a non-negative inner product with 
the independent variable, $ \langle x, h ^{\sigma }_I (x) \rangle _{\sigma } \ge 0$, a convenient choice due to the central role of the energy inequality, \eqref{e:energy}. 
\begin{align}\label{e:hs1}
	h_{I}^{\sigma} (x)& \coloneqq  \sqrt{\frac{\sigma(I_{-}) \sigma( I_{+})} {\sigma( I)}} 
\Biggl( \frac{{I_{+} (x)}}{\sigma( I_{+})} - \frac{{I_{-}(x)}}{ \sigma( I_{-})}  \Biggr)\,.    
\end{align}
In this definition, we are identifying an interval with its indicator function, and we will do so  throughout the remainder of the paper.  This  is  an $ L ^2 (\sigma )$-normalized  function, and   has $ \sigma $-integral zero. 
If $ \sigma $ is supported only on one child of $ I$, then we set $ h_I ^{\sigma } \equiv 0$.

For any dyadic interval $ I_{0}$ with $ \sigma (I_0) >0$, the non-zero functions among 
$  \{\sigma (I_0) ^{-1/2} {I_0}\} \cup \{ h ^{\sigma} _I \::\: I\in \mathcal D\,, I \subset I_0\}$ form an orthonormal basis for $ L ^2 ( I_0, \sigma )$.  
We will use the notation $ L ^2 _0(I_0, \sigma )$ for the subspace of $ L ^2 (I_0, \sigma ) $ of functions with mean zero. 
It has orthonormal basis  consisting of the non-zero functions in $  \{ h ^{\sigma} _I \::\: I\in \mathcal D\,, I \subset I_0\}$.  These are familiar properties.  But, another familiar property, that the positive and negative values of $ h ^{\sigma }_I$ are comparable in absolute value, fails in a dramatic fashion for non-doubling measures.  See Figure~\ref{f:haar}. 

\begin{figure}
\begin{tikzpicture}
\draw[|-|] (0,0) -- (2,0); 
\draw  (0,-1) -- (1,-1) -- (1,.5) -- (2,.5); 

\draw[|-|] (4,0) -- (6,0); 
\draw  (4,-1) -- (5,-1) -- (5,.1) -- (6,.1); 
\end{tikzpicture}
\caption{Two Haar functions. For the left function, the weight is nearly equally distributed between the two halves of the interval, in sharp contrast to the function on the right, in which the weight on the right half is much larger than on the left. } 
\label{f:haar}
\end{figure}
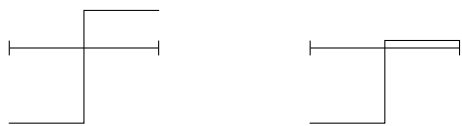

We will use the notations $ \mathbb E ^{\sigma } _{I} f = \sigma (I) ^{-1} \int _{I} f \; d \sigma $,   $ \hat f (I) = \langle f, h ^{\sigma } _{I} \rangle _{\sigma } $, as well as the equality below, holding for 
those $ I$ with $ h ^{\sigma }_I\not \equiv 0$. 
\begin{align}\label{e:mart}
\Delta ^{\sigma} _{I}f &= \langle f, h ^{\sigma} _{I} \rangle _{\sigma} h ^{\sigma} _{I}
= {I _{+}} \mathbb E ^{\sigma} _{I _{+}} f + 
{I _{-}} \mathbb E ^{\sigma} _{I _{-}}f - {I} 
\mathbb E ^{\sigma} _{I} f \,.  
\end{align}
This is the familiar martingale difference equality, and so 
we will refer to $ \Delta ^{\sigma} _{I} f$ as a martingale difference.  It implies the familiar telescoping identity 
$
\mathbb E _{J} ^{\sigma }f = \sum_{ I \::\: I\supsetneqq J} \mathbb E _{J} ^{\sigma } \Delta ^{\sigma } _{I} f \,. 
$

The \emph{Haar support} of a function $ f \in L ^2 (\sigma )$ is the collection $ \{I \::\: \hat f (I) \neq 0  \}$.

\subsection{Random Dyadic Grids}\label{s:rd}
Let $ \widehat{\mathcal D}$ be the standard dyadic grid in $ \mathbb R $, thus  all intervals $[0, 2 ^{n}] $ for $ n\in \mathbb N $ are in $ \widehat{ \mathcal D}$.  
A \emph{random dyadic grid}  $ \mathcal D$ is specified by $ \omega = \{\omega _n\} \in \{0,1\} ^{\mathbb Z }$, and the elements are 
\begin{equation*}
I = \hat   I \dot+ \omega \coloneqq \hat I + \sum_{n \::\: 2 ^{-n} < \lvert  I\rvert  }  2 ^{-n} \omega _{n}\,, \qquad \hat I\in \widehat{\mathcal D} \,.   
\end{equation*}
The natural uniform probability measure $ \mathbb P $ is placed upon $ \{0,1\} ^{\mathbb Z }$.

Fix  $ 0< \varepsilon < 1$ and $ r \in \mathbb N $.  
An interval $ I\in \mathcal D$ is said to \emph{$ (\varepsilon ,r)$-good} if for all intervals  
  $ J\in \mathcal D$ with $ \lvert  J\rvert \ge 2 ^{r-1}  \lvert  I\rvert $, the distance from $ \partial J$ and \emph{either child of $ I$} 
  is at least $ \lvert  I\rvert ^{\varepsilon } \lvert  J\rvert ^{1- \varepsilon }  $.  
Otherwise $ I$ is said to be \emph{$ (\varepsilon ,r)$-bad}. These are the basic properties of this definition.

\begin{proposition}\label{p:good}  These three properties hold. 
\begin{enumerate}
\item  The property of $I = \hat  I \dot+ \omega  $ being  $ (\varepsilon ,r)$-good only depends upon $ \omega $ and $ \lvert  I\rvert $.  
\item  $ \mathbf p _{\textup{good}} \coloneqq  \mathbb P (I   \textup{$ $ is $ (\varepsilon ,r)$-good} )$ is independent of  $ I$. 
\item  $ \mathbf p _{\textup{bad}} \coloneqq  1 - \mathbf p _{\textup{good}}  \lesssim \varepsilon ^{-1} 2 ^{- \varepsilon r}$. 
\end{enumerate}
\end{proposition}

\begin{proof}
An interval  $ I = \hat  I \dot+ \omega $ is equally likely to be the left or right half of its parent $ \pi ^{1} _{\mathcal D} I$, 
depending only on $ \omega _{n}$, where $ \lvert  I\rvert= 2 ^{n} $.  
Similarly, $ I$ is equally likely to be any one of the $ 2 ^{t}$ potential positions in  $ \pi ^{t} _{\mathcal D} I$, 
and its exact position is determined by $ \{\omega _{n} ,\dotsc, \omega _{n+t-1}\}$.  This proves the first two claims.  

For the last, if $ I$ is bad, then for some $t >r$,  
there holds $ \textup{dist} (I , \partial \pi ^{t} _{\mathcal D} I) \le  2 ^{ (1- \varepsilon )t}\lvert  I\rvert  $.  
For this to happen, it is necessary that the numbers $  \{\omega _{s} \::\: n + \lceil (1-\varepsilon) t \rceil < u \le n + t -1\}$ 
all be equal, and hence are either all $ 0$ or all $ 1$.  This clearly proves that 
\begin{equation*}
\mathbf p _{\textup{bad}} \le \sum_{t=r+1} ^{\infty } 2 ^{1-  (t -\lceil (1-\varepsilon) t \rceil )} \lesssim \varepsilon ^{-1}  2 ^{- \varepsilon r} \,. 
\end{equation*}
\end{proof}

This elementary proposition is used in the following fundamental way.  
Fix two weights $ w, \sigma $.   
With probability one, a random $ \mathcal D$ is admissible for both $ w$ and $ \sigma $. 
Indeed, the collection of points that are point masses for one of the two weights is a fixed countable collection of points. 
And any fixed point has probability zero of being an endpoint of an interval in $ \mathcal D$. 
Hence, we can, with probability one, define the Haar basis adapted to these  two weights.  
Write the identity operator on $ L ^2 (\sigma )$ by 
\begin{gather*}
  P ^{\sigma} _{\textup{good}} f + P ^{\sigma } _{\textup{bad}} f
\qquad 
\textup{where} \quad 
P ^{\sigma} _{\textup{good}} f \coloneqq  
\sum_{\substack{I \in \mathcal D \::\:  \textup{$ I$ is  $ (\varepsilon ,r)$-good}}}\langle f, h ^{\sigma } _{I} \rangle _{\sigma } h ^{\sigma }_I \,. 
\end{gather*}
Use the same notation for the weight $ w$. 

\begin{proposition}\label{p:bad} There holds  
\begin{equation*}
\mathbb E \lVert  P ^{\sigma} _{\textup{bad}} f \rVert_{ \sigma }  ^2 \lesssim  \varepsilon ^{-1} 2 ^{- \varepsilon r} \lVert f\rVert_{\sigma } ^2  \,. 
\end{equation*}
\end{proposition}

\begin{proof}
The location of $ I$ and the property of $ I$ being bad are independent, hence 
\begin{align*}
\mathbb E \lVert  P ^{\sigma} _{\textup{bad}} f \rVert_{ \sigma }  
&= \mathbb E  \sum_{I \in \mathcal D} \mathbf 1_{\textup{ $ I$ is bad}}  \hat f (I) ^2 
 = \mathbf p _{\textup{bad}} 
\mathbb E  \sum_{I \in \mathcal D}    \hat f (I) ^2 =  \mathbf p _{\textup{bad}} \lVert f\rVert_{\sigma } ^2 
\end{align*}
and then the proposition follows. 
\end{proof}

\begin{lemma}\label{l:goodBad}  
For any constant $ 1\leq C < \infty $,  $ 0 < \varepsilon < 1$, there is a choice of $ r \in \mathbb N $ sufficiently large so that this holds. 
Let  $ w, \sigma $ be a  pair of weights  for which the  constant $ \mathscr H$ and the constant $ \mathscr N$ 
in \eqref{e:N} are finite.  
Suppose   there holds uniformly over admissible dyadic grids $ \mathcal D$, 
\begin{equation} \label{e:gg}
\lvert  \langle H _{\sigma } P ^{\sigma } _{\textup{good}} f,  P ^{w } _{\textup{good}} g\rangle_w\rvert 
\leq  C\mathscr H \lVert f\rVert_{\sigma } \lVert g\rVert_{w} \,, 
\end{equation}
then,   $ \mathscr N \leq  2 C\mathscr H$. 

\end{lemma}

\begin{proof}
Use Proposition~\ref{p:bad} on the good and bad projections, as written and the same version for $ L ^2 (w)$.  
\begin{align*}
\lvert  \langle  H _{\sigma } f,g \rangle _{w}\rvert 
& \le  \mathbb E \bigl\{ \lvert  \langle H _{\sigma } P ^{\sigma } _{\textup{good}} f,  P ^{w } _{\textup{good}} g\rangle_w\rvert  
+   \lvert  \langle H _{\sigma } P ^{\sigma } _{\textup{good}} f,  P ^{w } _{\textup{bad}} g\rangle_w\rvert 
\\ & \qquad 
+ \lvert  \langle H _{\sigma } P ^{\sigma } _{\textup{bad}} f,  P ^{w } _{\textup{good}} g\rangle_w\rvert 
+ \lvert  \langle H _{\sigma } P ^{\sigma } _{\textup{bad}} f,  P ^{w } _{\textup{bad}} g\rangle_w\rvert \bigr\}\,. 
\end{align*}
The first term is controlled by the assumption \eqref{e:gg}, and the remaining terms are controlled by the  finiteness of $ \mathscr N $ and average-norm estimate on the 
bad projection.  By appropriate selection of $ f \in L ^2 (\sigma )$ and $ g\in L ^2 (w)$, there holds 
\begin{equation*}
\mathscr N_{\tau _0} \leq C\mathscr H +  C'\varepsilon ^{-1}2 ^{- \varepsilon r /2 } \mathscr N_{\tau_0}\,. 
\end{equation*}
 For any fixed $ \varepsilon $, we can take $ r \gtrsim \varepsilon ^{-1}  \log \varepsilon ^{-1} $, so that the second term can be absorbed into the left hand side. 
\end{proof}

\subsection{Context and Discussion}

\subsubsection{}
The random grid method was pioneered in \cite{NTV2}, and is a critical tool in non-homogeneous analysis \cite{V}, 
where the weights need not be doubling.  
It has a broader set of uses, as witnessed by a powerful representation of a general Calder\'on-Zygmund operator as 
a rapidly convergent sum of dyadic operators due to  Hyt\"onen \cite{2912709}. 

\subsubsection{}
The parameterization of the grids used here follows  Hyt\"onen \cite{MR2464252}, but the statistics of this parameterization are those of the random shift in Nazarov-Treil-Volberg \cites{NTV4,NTV2}.

\section{Necessary Conditions} 

Herein, we take up the necessity of  the  $ A_2$ condition from the norm inequality. 
Following that is the monotonicity property, an essential property of the Hilbert transform, 
and then showing  the necessity of the energy inequality from the $ A_2$ and interval testing condition.  
The energy inequality is foundational to the proof.

\subsection{The $ A_2$ Condition}\label{s:A2}
The $ A_2$ condition has  different forms, and so we clarify the language associated with the $ A_2$ condition here. 
The \emph{simple $ A_2$} condition is 
\begin{equation*}
\sup _{I} \frac {\sigma (I)} {\lvert  I\rvert } \cdot \frac {w (I)} {\lvert  I\rvert }, 
\end{equation*}
the supremum formed over all intervals $ I$.  
This reduces to the classical Muckenhoupt condition if $ w (dx)= w (x)dx$, where $ w (x) >0$ a.e., and $ \sigma (dx) = w (x) ^{-1} dx$.  
Next, are the \emph{half-Poisson} conditions: 
\begin{equation*}
\sup _{I}  P (\sigma , I)  \frac {w (I)} {\lvert  I\rvert } < \infty . 
\end{equation*}
Finally there is the full Poisson $ A_2$ condition   
\begin{equation} \label{e:full}
\sup _{I}  P (\sigma , I)  \cdot P (w, I) < \infty 
\end{equation}
and of course, we only use the Poisson condition with holes, of Hyt\"onen \cite{13120843}.   
We verify that the  Poisson $ A_2$ condition \eqref{e:A2}   is necessary for the two weight inequality \eqref{e:N}. 

\begin{proposition}\label{p:A2} Assume that the pair of weights do not share a common point mass, and that the norm inequality \eqref{e:N} holds. Then, the $ A_2$ condition \eqref{e:A2} holds. 
\end{proposition}

\begin{proof}
Fix the interval $ I = (a,b)$ as in \eqref{e:A2}, and  let $ a\in I $. 
We will estimate half of the Poisson integral of $ w$ using the notation 
\begin{equation}\label{e:pI}
p_I (x) ^2 \coloneqq   \frac  {\lvert  I\rvert } { (\lvert  I\rvert + \textup{dist} (x,I)  ) ^2 } \mathbf 1_{ [b , \infty )}
\end{equation}
so that $ P (w\cdot [b, \infty ) ,I) = \lVert p_I  \cdot [ b, \infty )\rVert_{ L^2(w)} ^2 $.  
Below, we estimate the right half of the Poisson integral of $ w$.  
\begin{align*}
\frac {\sigma (I)} {\lvert  I\rvert ^{1/2} }  \cdot P ( w  \cdot [b,  \infty ) , I)  
& \leq 
\int _{I} \int _{b } ^{\infty }  \frac 1 {\lvert  I\rvert + \textup{dist} (x,I)   } \cdot  \frac {1} {y-x} 
\; w  (dx) \, \sigma (dy) 
\\ & = 
\langle   H _{\sigma } (I) , p_I \rangle _{w } 
\lesssim  \mathscr N  \sigma (I) ^{1/2} \lVert p_I\rVert _{w }. 
\end{align*}
Rearranging,  
\begin{equation} \label{e:a'}
\frac {\sigma (I') }  {\lvert  I\rvert }  \cdot P ( w \cdot (a, \infty ), I)
\lesssim \mathscr N ^2 . 
\end{equation}
Clearly, the same inequality holds for $ (- \infty, a]$. 

\end{proof}

\subsection{The Monotonicity Principle}

Certain kinds of off-diagonal estimates for the Hilbert transform have  concrete estimates in terms of the Poisson integral. 
This estimate makes this precise, and shows moreover that we need not be that careful about exactly which function appears in the Poisson integral. 
It is at the core of the entire proof. 

\begin{lemma}[Monotonicity Principle] \label{mono}
Suppose that 
the two weights $ \sigma $ and $ w$ satisfy the $ A_2$ bound, and 
neither  has a point mass at an endpoint of $ I$. 
Let $  J\subset I$. 
There holds for any $ g\in L ^2(J, w )$, with $ w$-integral zero,  
\begin{equation} \label{e:mono0}
P (\sigma (\mathbb R - I) , I)  \Bigl\langle\frac  x { \vert I\vert }, \overline  g \Bigr\rangle _{w }
\lesssim  \langle H (\sigma (\mathbb R -I)) ,  \overline g\rangle _{w} \,. 
\end{equation}  
Here, $ \overline g = \sum_{J'} \lvert \widehat g (J') \rvert h ^{w} _{J'}$, is a Haar multiplier applied to $ g$. 
Suppose that $ J\subset I$ is good, with $ 2 ^{r} \lvert  J\rvert \le \lvert  I\rvert  $. Then for any two compactly supported 
weights $ \lvert  \nu \rvert \le \mu  $ supported off of the interval $ I$, there holds 
\begin{equation} \label{e:mono1}
	 \lvert  \langle H\nu , g \rangle _{w }\rvert
	\lesssim   \langle H\mu , \overline  g \rangle _{w }\simeq  {P( \mu, J )} 
	\bigl\langle\frac  x { \vert J \vert }, \overline  g \bigr\rangle _{w }.
\end{equation}  
\end{lemma}

Note that in the first estimate, the Poisson term is always estimated above by an inner product involving the Hilbert transform. 
In the second, note that the inner product can always be made larger by making the weight positive.  
Moreover, under moderate assumptions on the support of the weight, the first inequality can be reversed.    See Figure~\ref{f:mono}. 
In that figure, the function $ \mu $ is outside of $ 2 ^{r (1- \epsilon )} J$, 
so that $ H \mu $ is a smooth increasing function on $ J$. Moreover, the derivative of $ H \mu $ is approximately $ \lvert  J\rvert ^{-1} P (\mu  , J)$. 
So, if we form an inner product with the Haar function $ h ^{w} _{J}$, we only need to be concerned with the linear approximation to $ H \mu $.
However, the conditions to get the reversal are particular, and this drives  the case analysis in different sections of the proof. 

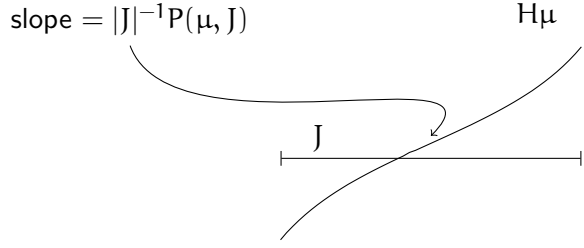
\begin{figure}
\begin{tikzpicture}[samples=50,domain=-2:2]
\draw plot  (\x,{tan(\x r/2.2)});   \draw (1.4,1.7) node {$ H \mu $}; 
\draw[|-|]  (2,-.2) -- (-2,-.2)  node[above, very near end] {$ J$}; 
\draw[<-]  (0,.1) to [out=45,in=-70]  (-4,1.3) node[above] { slope $= \lvert  J\rvert ^{-1} P (\mu  , J)  $}; 
\end{tikzpicture}
\caption{An illustration of the monotonicity principle.} 
\label{f:mono}
\end{figure}

\begin{proof}

We consider the first estimate.  
By linearity, it suffices to consider the case of $ g (x) = h ^{w} _{J} (x)$, for $ J\subset I$, and indeed we can take $ J=I$.  
We need to separate the two weights involved. 
The $ A_2$ condition is the only condition needed for the weak-boundedness principle, Proposition~\ref{p:weakB}. 
Applying  it in this setting, notice that it shows that for $ \lambda >1$, 
\begin{equation*}
\lvert  \langle  H _{\sigma } ( \lambda I - I)  ,   h ^{w} _{I}\rangle\rvert 
\lesssim \mathscr A_2 ^{1/2}  \sqrt { \sigma (\lambda I - I )  } . 
\end{equation*}
The assumption that $ \sigma $ does not have mass at the endpoints of $ I$ implies that $ \sigma (\lambda I - I) $ can be made 
arbitrarily small, as $ \lambda \downarrow 1$.  Therefore, it suffices to consider $ H _{\sigma } ( \mathbb R - \lambda I)$, for some fixed $ \lambda >1$. 

Then estimate 
\begin{align*}
\langle H ^{c}(\sigma \cdot (\mathbb R - \lambda I)) , h ^{w} _{I} \rangle_w & = 
\int _{\mathbb R - \lambda I} \int _{J} \frac {1} {y-x}  h ^{w} _{J} (x) \; w (dx) \sigma (dy) 
\\
&= \int _{\mathbb R- \lambda I } \int _{I}   \biggl[ \frac {1} {y-x} - \frac 1 {y-x_J} \biggr]    h ^{w} _{I} (x) \; w (dx)  \, \sigma (dy)  
\\
&= \int _{\mathbb R- \lambda I } \int _{I}   \ \frac {x-x_J} {(y-x) (y- x_J)}     h ^{w} _{I} (x) \; w (dx)  \, \sigma (dy)  
\\
& \gtrsim 
\int _{\mathbb R -I} \int _{I}   \frac { \lvert  I\rvert } { (\lvert  I\rvert + \textup{dist} (y,I)) ^2 }  
\cdot \frac {x-x_J} {\lvert  I\rvert }h ^{w} _{I} (x) \; w (dx) \,\sigma (dy)  
\\
& = P (\sigma \cdot (\mathbb R- \lambda I) , I) \bigl\langle  \frac x {\lvert  I\rvert },  h ^{w} _{I}\bigr\rangle_{w} \,. 
\end{align*}
Here, $ x_J$ is the center of $ J$, and it can be inserted for the usual reason that $ h ^{w} _{J}$ has $ w$-integral zero.  
Then, use the fact that  $ (x-x_J ) h ^{w}_J \geq 0$, and that $(y-x) (y-x_J) >0$.  
 So \eqref{e:mono0} holds.

\smallskip 

The second inequality  \eqref{e:mono1} comes with the assumption that $ J\subset I$, $ 2 ^{r}\lvert J  \rvert < \lvert  I\rvert  $, whence 
$ \textup{dist} (J,I) > \lvert  J\rvert ^{\epsilon } \lvert  I\rvert ^{1 - \epsilon } \ge 2 ^{r (1- \epsilon ) } \lvert  J\rvert  $.  Namely, the support of $ h ^{w} _{J}$ and that 
of $ \mu $ are separated.  Then, inserting a constant as we can since the Haar function has integral zero, 
\begin{align*}
 \langle H\nu ,h^{w}_J \rangle _{w } 
&= \int _{\mathbb R - I} \int _{J} \bigl\{ \frac 1 {y-x}- \frac 1{y - x_J}\bigr\}  h ^{w} _{J} (x) \;\nu(dy)\,w (dx)   
\\
& =  \int _{\mathbb R - I} \int _{J}  \frac { x- x_J} {(y-x) (y- x _J)}   h ^{w} _{J} (x) \;\nu(dy)\,w (dx)  
\end{align*}
Notice that the integrand is non-negative, hence we can make the integral bigger in absolute value 
by replacing $ \lvert  \nu \rvert $ by $ \mu $. This is the first inequality in \eqref{e:mono1}.  
 For the second equivalence,  by the separation in supports, we have $ \frac 1 {(y-x) (y- x _J)}  \simeq  \frac 1 {  (y- x _J) ^2 } $ 
in the range of integration. And this finishes the proof. 
\end{proof}

\subsection{The Energy Inequality}

The energy inequality is phrased in terms of the quantity 
\begin{align}\label{e:E}
E (w, I) ^2 &\coloneqq   \lvert  I\rvert ^{-2} \mathbb E _{I}  ^{w}\bigl\lvert  x \cdot I - \mathbb E _{I} ^{w} x \bigr\rvert ^2 
= \lvert  I\rvert ^{-2} \sum_{J \::\: J\subset I} \langle x, h ^{w} _{J} \rangle_w ^2 \,.  
\end{align}

\begin{lemma}\label{l:energy}[The Energy Inequality]   For any interval $ I_0$ and any partition $ \mathcal P$ of $ I_0$ into intervals such that neither $ \sigma $ nor $ w $ have point masses at the endpoints,  there holds 
\begin{equation}\label{e:energy}
\begin{split}
\sum_{I \in \mathcal P  }  
P (  \sigma  (I_0 \setminus I), I) ^2 E (w,I) ^2 w (I) 
\le C_0 \mathscr H ^2 \sigma (I_0) \,. 
\end{split}
\end{equation}
Here, $ C_0$ is an absolute constant.  
\end{lemma}

\begin{proof}
It follows from \eqref{e:mono0}, viewed in dual fashion, that 
\begin{align*}
P (  \sigma  (I_0 \setminus I), I) ^2 E (w,I) ^2 w (I) & \lesssim \lVert  H (\sigma (I_0 \setminus I))  \cdot I\rVert _{w} ^2 
\\
& \lesssim \lVert  H (\sigma \cdot I_0)  \cdot I\rVert _{w} ^2 +  \lVert  H (\sigma \cdot I)  \cdot I\rVert _{w} ^2 
\\
 & \lesssim \lVert  H (\sigma \cdot I_0)  \cdot I\rVert _{w} ^2 + \mathscr T ^2  \sigma (I). 
\end{align*}
Above, we have appealed to the testing assumption \eqref{e:T}. 
Summing over $ I\in \mathcal P$, the second term above is clearly no more than $ \mathscr T ^2  \sigma (I_0)$. 
And the second term is no more than 
\begin{equation*}
 \lVert  H (\sigma \cdot I_0)  \cdot I_0\rVert _{w} ^2 \leq \mathscr T ^2 \sigma (I_0). 
\end{equation*}
\end{proof}

\subsection{Context and Discussion}

\subsubsection{}
In the absence of common point masses, the necessity of the \emph{full} $\mathscr  A_2 $ condition, namely \eqref{e:full}, was easily available, with an argument of Sergei Treil already pointed out by Sarason in his note \cite{sarasonConj}.   This argument, based upon complex variables, has close analogs in \cites{V,10031596}.  A real variable proof   is  in \cite{10014043}, 
it is essentially an elaboration of the argument in the  early paper of Muckenhoupt and Wheeden \cite{MR0417671}. 
Despite the necessity, only the half Poisson $ A_2$ condition is used, together with testing, in the 
proof of sufficiency, in the case of no common point masses. 

\subsubsection{}
Higher dimensional extensions of the $ A_2$ which are not straight forward, are discussed in \cite{08070246}.
There are notable distinctions important to higher dimensions. First, the necessary Poisson type condition 
only comes in its `half' form.  Second, the power on the Poisson kernel comes as the square of the dimension of the kernel involved, 
a feature familiar from the analysis of reproducing kernel spaces.  
Third, the degree of the Poisson kernel matches the important derivative Poisson decay, important to 
energy considerations, only when the dimension  of the kernel is one.

\subsubsection{}
The energy inequality was influenced by the following assumption placed upon the pair of weights in \cites{V,10031596}.  
Assume that there is a finite constant $ \mathscr P$ so that for all intervals $ I_0$, and all partitions $ \mathcal P$ of $ I_0$ into intervals, 
\begin{equation} \label{e:pivotal}
\sum_{I\in \mathcal P} P ( \sigma \cdot  I_0, I) ^2 w (I) \le \mathscr P ^2 \sigma (I_0) \,.
\end{equation}
Also assume that the dual inequality holds.  In the language of Nazarov-Treil-Volberg, this is the \emph{pivotal condition}.   They proved 
\begin{priorResults}\label{t:pivotal} Assume that $ w$ and $ \sigma $ do not share a common point mass.  
Then, there holds $ \mathscr N \lesssim \mathscr A_2 ^{1/2} + \mathscr T + \mathscr P$. 
\end{priorResults} 
 
This is a very strong Theorem, with an important proof. It decisively used the tools of non-homogeneous harmonic analysis, namely random grids, and  good-bad projections.  The pivotal condition controlled certain degeneracies in the pair of weights, compare to Definition~\ref{d:stopping}.  
To illustrate the difficulties in the general case, we prove this theorem in \S\ref{s:pivotal}. 

The pivotal condition holds if the pair of maximal function estimates hold, namely $ M _{\sigma } \::\: L ^2 (\sigma ) \mapsto L ^2 (w)$ and $ M _{w} \::\: L ^2 (w) \mapsto L ^2 (\sigma )$.  This is easy to see. From \eqref{e:pivotal}, 
\begin{align*}
\sum_{I\in \mathcal P} P ( \sigma \cdot  I_0, I) ^2 w (I)  &\le \sum_{I\in \mathcal P} \inf _{x\in I} M (\sigma \cdot I_0) (x) ^2 w (I)
\\
&\le \int _{I_0} M (\sigma \cdot I_0) ^2 \; d w \lesssim \sigma (I_0)\,, 
\end{align*}
by the assumed norm bound on the maximal function. 
One sees that Theorem~\ref{t:pivotal}  offered a complete characterization of the two weight inequality for the triple of operators 
$ (H _{\sigma }, M _{\sigma }, M _{w})$.  If the pair of weights are doubling, then the boundedness of the maximal functions is a consequence of the $ A_2$ condition.\footnote{Alternatively, under the assumption of $ w$ being doubling, check that the energy satisfies $ E (w, I) \gtrsim 1$, 
with the implied constant depending upon the doubling constant.  Thus, the necessary energy inequality implies the pivotal condition.} 
The full characterization of the boundedness of the Hilbert transform was thus known for doubling measures.  See \cite{V}.

The pivotal condition is generic in the following sense. Assuming the pivotal condition, the Hilbert transform can be replaced by a generic Calder\'on-Zygmund operator with one derivative on its kernel.   
This, and its extension to operators with a rougher kernel,  was fundamental in the paper \cite{ptv}, whose main result was an important intermediate one in the solution of the $ A_2$ conjecture \cite{2912709}.

\subsubsection{}
Nazarov-Treil-Volberg, in language reminiscent of Sarason, wrote that `perhaps the pivotal condition is necessary' for the 
boundedness of the Hilbert transform.  
This turned out to have a strong measure of truth, in that using the specific structure of the Hilbert transform, the \emph{energy inequality} was shown  necessary in \cite{10014043}.  Note that one can formally obtain the pivotal condition \eqref{e:pivotal} from the energy inequality \eqref{e:energy} by raising the energy term $ E (w,I)$ to the zero power, rather than the necessary power $ 2$. 
The  paper \cite{10014043} then adapted the approach of \cites{V,10031596}, essentially imposing a new weaker condition on the pair of weights in which one raised the energy to a power intermediate between $ 0$ and $ 2$.   In addition, that paper provided an explicit example, recounted in \S\ref{s:examples}, that showed that the pivotal condition \eqref{e:pivotal} is \emph{not necessary} for the boundedness of the Hilbert transform.

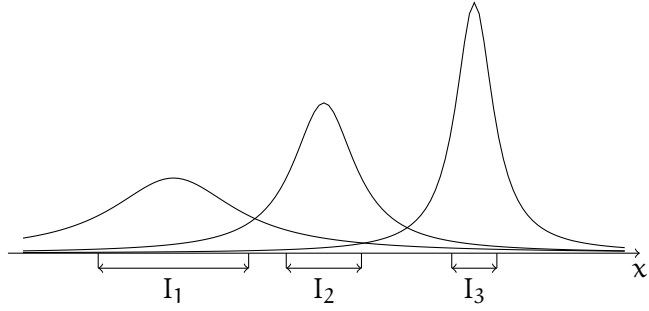
\begin{figure}
\begin{tikzpicture}[domain=0:8, samples=125]
  \draw[->] (-0.2,0) -- (8.2,0) node[right,below] {$x$};
  \foreach \a/\b/\c in {2/1/1, 4/.5/2, 6/.3/3}  
  {\draw  plot (\x,{ \b/(\b^2+(\x-\a)^2});      \draw (\a+\b,0) -- (\a+\b,-.2);    	    \draw (\a-\b,0) -- (\a-\b,-.2); 
    \draw[<->] (\a-\b, -.2) -- (\a +\b, -.2) node[midway,below] {$ I _{\c}$};};    
\end{tikzpicture}
\caption{The function $ \frac { \lvert  I\rvert } { \lvert  I\rvert ^2 + \lvert  x- x_I\rvert ^2   }$ are graphed for three separate intervals.}
\label{f:P}
\end{figure}

\subsubsection{}
The energy inequality is rather subtle.  The Poisson term  $P (\sigma , I) $ can be much larger than the simple average, 
but this is compensated for with the terms $ E (w, I) ^2 w (I)$.  The Figure~\ref{f:P} is offered to provide some insight into the 
`long tails' that the Poisson term can have. 

Another indication of this subtlety is the observation that the energy inequality will not follow from just the $ A_2$ condition.  
Given interval $ I_0$, and partition  $ \mathcal P$ of $ I_0$, one can write 
\begin{align*}  
\sum_{I \in \mathcal P  }  
P ( \sigma , I) ^2 E (w,I) ^2 w (I) 
&\le \mathscr A_2 
\sum_{I \in \mathcal P  }  
 \lvert  I\rvert \cdot P (  \sigma, I) ^2 
\\
&= \mathscr A_2 
\int _{I_0}   \sum_{I \in \mathcal P}   \frac {\lvert  I\rvert ^2  } {(\lvert  I\rvert + \textup{dist} (x,I)) ^2 } \;  \sigma (dx)\,. 
\end{align*}
To finish, one would have to know that the  function inside the integral is bounded. 
But, this is not true in general.  
Though a very tame $ BMO$ function, this fact does not help, since 
$ \sigma $  is a general measure, and need not satisfy any $ A _{\infty }$ type condition.  
Indeed, the  proof of the main theorem would be  more or less classical if the weights satisfy a  $ A _{\infty }$ type conditions.

\subsubsection{}
The monotonicity principle, Lemma~\ref{mono}, was noted in \cite{11082319}.  It, with the energy inequality, are essential aspects of the proof.

\section{Global to Local Reduction} \label{s:global}

Our aim is to prove  the estimate \eqref{e:gg}, 
\begin{equation*}
\sup 
\lvert  \langle H _\sigma   P ^{\sigma } _{\textup{good}}f, P ^{w} _{\textup{good}}g \rangle _{w}\rvert \lesssim \mathscr H \lVert f\rVert_{\sigma } \lVert g\rVert_{w} \,.  
\end{equation*}
That is, the bilinear form only needs to be controlled for \emph{$ (\varepsilon ,r)$-good functions $ f= P ^{\sigma } _{\textup{good}}f$ and 
similarly for $ g$, goodness being defined with respect to a fixed dyadic grid.} 
Suppressing the notation, we write `good' for `$(\varepsilon,r)$-good,' and  it is always   assumed that  the dyadic grid $ \mathcal D$ is fixed, and only good intervals are in the Haar support of $ f$ and $ g$.  
We clearly remark on goodness when the property is used; any value of $ 0< \varepsilon \le \tfrac 14 $ is sufficient for our purposes.  
The symbol $ \varepsilon $ is kept throughout, as a guide to the appearance of the good property of intervals. 

The inequality above is reduced to the local estimate, \eqref{e:BF}, at the end of this section. 
It is sufficient to assume that $ f$ and $ g$ are supported  on an interval $ I^0$; by trivial use of the interval testing condition, we can further assume that $ f$ and $ g$ are of integral zero in their respective spaces.  
Thus, $ f $ is in the linear span of (good) Haar functions $ h ^{\sigma } _I$ for $ I\subset I^0$, and similarly for $ g$. 

The distinction between $ J\subset I$ and $ J\Subset I$ ($ J\subset I$ and $ 2 ^{r} \lvert  J\rvert\le \lvert  I\rvert  $) forces some 
case analysis. This is further simplified by this assumption on the Haar supports of $ f, g$. There are two integers $ s _{f}, s _{g}$ such that 
\begin{equation} \label{e:mod}
f = \sum_{\substack{I \::\: I\subset I^0\\ \log_2 \lvert  I\rvert \in s_f + r \mathbb Z  }} \Delta ^{\sigma } _{I} f
\end{equation}
and similarly for $ g$. Thus, the lengths of the (good) intervals $ I$ are restricted to an equivalence class mod $ r$, 
which is to say that the \emph{scales of $ f$ are separated by $ r$}, and the same for $ g$.  This will be a convenience 
at a few technical points below.  
Set $ \mathcal D _{f} \coloneqq  \{  I \::\: \log_2 \lvert  I\rvert \in s_f -1 + r \mathbb Z \}$, so these are the children of the intervals that 
appear in \eqref{e:mod}.   Due to the probabilistic way in which the grids are constructed, we can further assume that 
$ I ^{0} \in \mathcal D _{f}$.   Also set $ \mathcal D _{g} \coloneqq   \{  I \::\: \log_2 \lvert  I\rvert \in s_g  + r \mathbb Z $.

We are to control the bilinear form 
\begin{equation} \label{e:00}
\langle H _{\sigma } f,g \rangle_w = \sum_{I, J \::\: I,J\subset I^0} 
\langle H _{\sigma } \Delta ^{\sigma }_I f, \Delta ^{w} _{J} g \rangle _w \,. 
\end{equation}
The sum is broken into many summands, as is typical in these arguments, but the manner in which it is done 
has some important points below. The most important of these are the two `triangular' forms 
\begin{equation}  \label{e:ABOVE} 
B ^{\textup{above}} (f,g) \coloneqq  \sum_{I \::\: I\subset I^0} \sum_{J \::\: J\Subset I_J} 
\mathbb E ^{\sigma } _{I_J} \Delta ^{\sigma }_I f \cdot   \langle H _{\sigma } I_J, \Delta ^{w} _{J} g \rangle _w 
\end{equation} 
and the dual form, $ B ^{\textup{below}} (f,g) $.  
Here, $ J\Subset I$ means that $ J\subset I$ and $ 2 ^{r} \lvert  J\rvert \le \lvert  I\rvert  $, in words `$ J$ is strongly contained in $ I$'. And the interval $ I_J$ is the child of $ I$ that contains $ J$.  
Goodness of $ J$ justifies the use of this condition.
A basic fact, proved in \S\ref{s:elem}, is 

\begin{lemma}\label{l:above}  There holds 
\begin{equation*}
\begin{split}
\bigl\lvert  \langle H _{\sigma } f ,g \rangle _{w} - B ^{\textup{above}}   (f,g) &
- B ^{\textup{below}}   (f,g)  \bigr\rvert 
\lesssim \mathscr H \lVert f\rVert_{\sigma } \lVert g\rVert_{w}  \,. 
\end{split}
\end{equation*}
\end{lemma}

Thus, the main technical result is as below; it immediately supplies our main theorem. 

\begin{theorem}\label{t:above}  There holds 
\begin{equation*}
\lvert B ^{\textup{above}} (f,g)\rvert 
\lesssim \mathscr H \lVert f\rVert_{\sigma } \lVert g\rVert_{w}  \,.
\end{equation*}
The same inequality holds for the dual form $ B ^{\textup{below}} (f,g) $.
\end{theorem}

The remainder of this section is devoted to a  reduction of the global Theorem~\ref{t:above} to a 
local estimate described in \S\ref{s:stop}. 
In the local estimate,   the function $ f$ is more structured in that it has bounded averages  on a fixed interval,   and the pair of functions $ f,g$ are   more structured in that their Haar supports avoid intervals that strongly violate the energy inequality.  Still the argument to control this term requires a subtle recursion.

We construct stopping data, which accomplishes two ends, in that it will control certain telescoping sums of martingale 
differences of $ f$, and that it controls certain degeneracies  in an energy estimate on the weights.

\begin{definition}\label{d:stopping} Define $ \mathcal F$, the stopping intervals, recursively by 
initializing $ I ^{0} \in \mathcal F$, and in the recursive step, if $ F \in \mathcal F$ is minimal, 
add to $ \mathcal F$ the maximal subintervals $F'\subset F  $, with $ F'\in \mathcal D _{f}$, 
either 
\begin{description}
\item[$ f$ stopping]   $ \mathbb E ^{\sigma } _{F'} \lvert  f\rvert > C \alpha _{f} (F) \coloneqq   \mathbb E ^{\sigma } _{F} \lvert  f\rvert $. 
\item[Energy Stopping]    $  \lVert   H _{\sigma } (F \setminus F') \cdot F' \rVert ^2 _{w} \ge C \mathscr H ^2 \sigma (F')$. 
\end{description}
That is, we stop if either the average of $ f$ becomes too large, or, essentially, the energy condition becomes too large. 
\end{definition}

For appropriate constant $ C$, it follows that $ \mathcal F$ is $ \sigma $-Carleson, namely 
\begin{equation} \label{e:s-Car}
\sum_{F' \in \mathcal F \;:\; F'\subset F} \sigma (F') \leq \tfrac 1 {10} \sigma (F), \qquad F\in \mathcal F. 
\end{equation}
Many properties of the $ \sigma $-Carleson property are used below.  But, also note the following property: 
\begin{equation}\label{e:<alpha}
 \lvert \mathbb E _{I} ^{\sigma }  f \rvert  \le  C\alpha _{f} (\pi _{\mathcal F} I)
\end{equation}
 
We will  use the notation 
\begin{equation} \label{e:PQ}
 P ^{\sigma } _{F} f \coloneqq  \sum_{ I \in \mathcal D \::\: \pi _{\mathcal F}I=F} \Delta ^{\sigma } _{I} f\,, \qquad F\in \mathcal F\,. 
\end{equation} 
and  a dual projection $ Q ^{w} _{F}g$, is defined similarly, but importantly, we replace $ \pi _{\mathcal F} J=F$ 
by $ \dot \pi _{\mathcal F} J=F$, meaning that $ F$ is the smallest interval in $ \mathcal F$ such that 
$ J\Subset F$.  
(Note that both are projections, but $ P ^{\sigma } _{F} f$ is a structured function, while $ Q ^{w} _{F}g$ is not.)
The $ \sigma $-Carleson property  allows us to estimate  
\begin{align}  \label{e:quasi} 
	\begin{split}
		\sum_{F\in \mathcal F} 
	\{ \alpha _{f} (F) \sigma (F) ^{1/2}& + \lVert P ^{\sigma } _{F} f\rVert_{\sigma }\} 
	\lVert Q ^{w} _{F} g\rVert_{w} 
\\	& \le 
	\Biggl[
	\sum_{F\in \mathcal F} 
	\{\alpha _{f} (F) ^2  \sigma (F) + \lVert P ^{\sigma } _{F} f\rVert_{\sigma } ^2 \} 
	\times 
	\sum_{F\in \mathcal F} 
	\lVert Q ^{w} _{F} g\rVert_{w} ^2 
	\Biggr] ^{1/2}   \lesssim   \lVert f\rVert_{\sigma } \lVert g\rVert_{w} \,. 
	\end{split}
\end{align}
We will refer to as the \emph{quasi-orthogonality}  argument.  It holds only under the assumption that the projections $ Q ^{w} _{F}$ are pairwise 
orthogonal.   It is very useful. 

\medskip 
We henceforth concentrate on the `above' forms, with all considerations applying in their dual formulation to control the `below' forms. 
Return to the double sum \eqref{e:00}, and define  
\begin{align}
B ^{\textup{above}} _{\mathcal F, \textup{loc} } (f,g)
& \coloneqq  \sum_{F\in \mathcal F} B ^{\textup{above}} (P ^{\sigma } _{F} f, Q ^{w} _{F} g). 
\end{align}
The  global to local reduction is:  

\begin{corollary}\label{t:aboveCorona}[Global to Local Reduction] 
There holds 
\begin{gather*}
\bigl\lvert
B ^{\textup{above}}  (f,g)
- 
B ^{\textup{above}} _{\mathcal F,\textup{loc}}  (f,g)\bigr\rvert \lesssim \mathscr H 
\lVert f\rVert_{\sigma } \lVert g\rVert_{w} . 
\end{gather*}
\end{corollary}

\begin{proof}
Observe that $ B ^{\textup{above}} (f,g)$ is a sum over pairs of intervals $ (I,J)$ 
with $ J\Subset I_J$, whence $ \dot \pi _{\mathcal F} J \subset \pi _{\mathcal F} I$.  
Now, the case of $  \dot \pi _{\mathcal F} J = \pi _{\mathcal F} I$ is contained in the 
form $ B ^{\textup{above}} _{\mathcal F, \textup{loc}}  (f,g)$, hence we need only concern ourselves 
with the case of $  \dot \pi _{\mathcal F} J \subsetneq \pi _{\mathcal F} I$, that is, 
we need only bound 
\begin{equation*}
\sum_{F \in \mathcal F} \sum_{\substack{F'\in \mathcal F\\ F'\subsetneq F }} 
B ^{\textup{above}} _{\mathcal F}  (P ^{\sigma } _{F}f, Q ^{w} _{F'}g). 
\end{equation*}

Set $ g_F \coloneqq  Q ^{w} _{F} g$.  The sum in question is 
\begin{equation}  
\sum_{F\in \mathcal F} \sum_{I \::\: I\supsetneq F} 
\mathbb E ^{\sigma } _{I_J} \Delta ^{\sigma }_I f \cdot   \langle H _{\sigma } I_F,  g_F\rangle _w .
\end{equation}
We invoke, for the first time, the \emph{Hilbert-Poisson exchange argument}:  (a) Replace the argument of the Hilbert transform by  a stopping interval. 
(b) Invoke the stopping tree construction to control the sum of martingale differences of $ f$.  (c) Apply interval testing, on the stopping interval. 
 (d) Use the monotonicity principle to dominate the complementary term in terms of a Poisson integral.  (e)  
 Analyze the  Poisson term.  (f) Use quasi-orthogonality, as needed. 
 
 The argument of the Hilbert transform is $ I_F$, the child of $ I$ that contains $ F$. 
	Write $ I_F= F + (I_F-F)$, and use linearity of $ H _{\sigma }$.  Note that by the standard martingale difference identity and 
	the construction of stopping data, 
\begin{equation*}
	\Bigl\lvert \sum_{I \;:\; I\supsetneq  F} \mathbb E ^{\sigma } _{I _{F}} \Delta ^{\sigma } _{I} f  \Bigr\rvert \lesssim \alpha _{f} (F)\,, 
	\qquad F\in \mathcal F \,. 
\end{equation*}
Hence, invoking interval testing, 
\begin{align*}
	\Bigl\lvert \sum_{F\in \mathcal F} 
		\sum_{I \;:\; I\supsetneq  F} \mathbb E ^{\sigma } _{I _{F}} \Delta ^{\sigma } _{I} f \cdot 
		\langle H _{\sigma } F, g_F \rangle_w \Bigr\rvert & \lesssim
		\sum_{F\in \mathcal F}  \alpha _{f} (F) 
		\bigl\lvert \langle H _{\sigma } F, g_F \rangle_w \bigr\rvert
		\\
		& \lesssim \mathscr  H 	\sum_{F\in \mathcal F}  \alpha _{f} (F) \sigma (F) ^{1/2} \lVert g_F\rVert_{w} \,.
\end{align*}
 Quasi-orthogonality bounds this last expression. 

For the second expression, when the argument of the Hilbert transform is $ I_F - F$, 
is the objective of \S \ref{s:fe}.    We have proved 
\begin{equation}\label{e:xTildeB}
\Bigl\lvert 
\sum_{F\in \mathcal F} \sum_{I \::\: I\supsetneq F} 
\mathbb E ^{\sigma } _{I_J} \Delta ^{\sigma }_I f \cdot   \langle H _{\sigma } I_F,  g_F\rangle _w 
\Bigr\rvert
\lesssim \mathscr H \lVert f\rVert_{\sigma } \lVert g\rVert_{w}.   
\end{equation}
 This completes the Hilbert-Poisson exchange argument.

\begin{equation*}
\Bigl\lvert 
\sum_{I \;:\; I\supsetneq  F} \mathbb E ^{\sigma } _{I _{F}} \Delta ^{\sigma } _{I} f \cdot  (I_F - F)
\Bigr\rvert \lesssim \Phi \coloneqq  \sum_{F'\in \mathcal F} \alpha _{f} (F') \cdot F' \,, \qquad F\in \mathcal F \,. 
\end{equation*}
Therefore,   the monotonicity property \eqref{e:mono1} applies, and yields  
\begin{equation}  \label{e:HPx}
\Bigl\lvert 
\sum_{I \;:\; I\supsetneq  F} \mathbb E ^{\sigma } _{I _{F}} \Delta ^{\sigma } _{I} f \cdot \langle H _{\sigma } (I_F - F),  g_F\rangle _{w}
\Bigr\rvert \lesssim \sum_{J\in \mathcal J ^{\ast} (F)}P _{\sigma } (\Phi  \cdot F ^{c} , J) 
\Bigl\langle  \frac {x} {\lvert  J\rvert } , J  \overline  g_F\Bigr\rangle_{w} \,, \qquad F\in \mathcal F\,. 
\end{equation}
Here $ \overline g  _F \coloneqq  \sum_{J\in \mathcal J (F) \::\: J\Subset F} \lvert  \hat g (J) \rvert \cdot h ^{w} _{J} $, so that every term has a positive inner product with $ x$, and $ \mathcal J ^{\ast} (F)$ are the maximal good intervals $ J\Subset F$, 
and $ J \in \mathcal D _{g}$.  (If $ J\not\in \mathcal D _{g}$, then $ \langle g, h ^{w}_J \rangle_w =0$, by choice of $ g$ at the beginning of the proof.) 

The control of the  sum over $ F\in \mathcal F$ of \eqref{e:HPx}

\end{proof}

It remains to control $ B _{\mathcal F, \textup{loc}} ^{\textup{above}} (f,g)$. 
Keeping the quasi-orthogonality argument in mind,   appropriate control on the individual summands is enough to control it.
To describe what has been done, one must note that the functions $ P ^{\sigma }_F f$ 
\emph{need not be bounded}.  But, we are only concerned with averages over intervals where the average will be  bounded. 
In addition this function and $ Q ^{w} _{F}g$ are   well-adapted to the pair of weights $ w, \sigma $. 
The next lemma, combined with the quasi-orthogonality estimate clearly completes the proof of the Theorem.

\begin{lemma}\label{l:BF}[The Local Estimate] For each $ F\in \mathcal F$, there holds 
\begin{equation}  \label{e:BF}
\lvert  B ^{\textup{above}} (P ^{\sigma } _Ff, Q ^{w} _{F}g)\rvert \le \mathscr H \{\alpha _{f} (F)\sigma (F) ^{1/2} +  \lVert  P ^{\sigma } _F f\rVert_{\sigma }\} \lVert Q ^{w} _{F}g\rVert_{w}. 
\end{equation}
\end{lemma}

The first step in the proof of the Lemma above is to invoke the  Hilbert-Poisson exchange argument again, but we will arrive at a Poisson term which falls outside the immediate scope of the  energy inequality.  
Focusing on the argument of the Hilbert transform in \eqref{e:BF}, we write $ I_J = F - (F - I_J)$. 
When the interval is $ F$, and $ J$ is in the Haar support of $ Q ^{w} _{F} g$, notice that the scalar 
\begin{equation*}
\alpha _{f} (F)\varepsilon _J \coloneqq  \sum_{I \::\: J\Subset I \subset F} 
\mathbb E ^{\sigma } _{J} \Delta ^{\sigma } _{I} f   
\end{equation*}
is bounded by an absolute constant, by construction of the stopping intervals. 
Indeed, by the telescoping identity for martingale differences, 
\begin{equation*}
\alpha _{f} (F)\varepsilon _J = \sum_{I \::\:  I ^{-}\subsetneq I \subset F} 
\mathbb E ^{\sigma } _{I ^{-}} \Delta ^{\sigma } _{I} f =  \mathbb E ^{\sigma } _{I_J} f \,,    
\end{equation*}
which is at most $ C \alpha _{f} (F)$, since $ \dot \pi _{\mathcal F} J=F$.   Therefore, we can write 
\begin{align} 
 \Bigl\lvert \sum_{I \::\: I\subset F} \sum_{J \::\: J\Subset I} 
 \mathbb E ^{\sigma } _{J} \Delta ^{\sigma } _{I} f \cdot \langle H _{\sigma } F, \Delta ^{w} _{J} g\rangle  \Bigr\rvert
 & \leq  \alpha _{f} (F)
\Bigl\lvert 
\bigl\langle  H _{\sigma } F ,  \sum_{\substack{J \::\: J\Subset F }}  \varepsilon _J \Delta ^{w} _{J} g \bigr\rangle_{w}
\Bigr\rvert
 \\ \label{e:badMart}
 & \le \mathscr T \alpha _{f} (F) \sigma (F) ^{1/2} 
 \Bigl\lVert    \sum_{\substack{J \::\: J\Subset F }}  \varepsilon _J \Delta ^{w} _{J} g  \Bigr\rVert_{w} \le  \mathscr T \sigma (F) ^{1/2}  \lVert g\rVert_{w} \,. 
\end{align}
This uses only interval testing  and orthogonality of the martingale differences, and it matches the first half of the right hand side of \eqref{e:BF}.  

\smallskip 

When the argument of the Hilbert transform is $ F - I_J$, this is the 
\emph{stopping form}, the last component  of the local part of the problem.  It requires a subtle recursion, described in \S\ref{s:stop}.

\subsection{Context and Discussion}

\subsubsection{}
Many $ T1$ theorems have arguments, sometimes subtle ones, about telescoping sums which collapse. These arguments are systematically handled 
herein with the stopping data, as opposed to more intricate Carleson measure arguments.  

\subsubsection{}
The use of the energy stopping intervals is motivated by the use of the corresponding intervals, under the pivotal condition \eqref{e:pivotal},  in \cites{V,10031596}.  However, the pivotal condition is not necessary for the two weight inequality, while the energy inequality is necessary from the $ A_2$ and interval testing conditions. 

\subsubsection{}
Initial arguments had largely ignored the structure of the pair of functions $ f, g$ in the inner product $ \langle H _{\sigma } f,g \rangle _{w}$, 
instead concentrating on proving an intricate series of Carleson measure type estimates.  
This changed with the argument of \cite{11082319}, which introduced Calder\'on-Zygmund stopping intervals, and the quasi-orthogonality argument into the subject.  It was only then that the role of the global to local step was identified, but not proved.
Stopping data also allows us to avoid the subtle problem of  \emph{absence of canonical paraproducts}.  
Attempts to introduce them induce \emph{ad hoc} elements into the proof.

\subsubsection{}
This section begins with the elementary and familiar Lemma~\ref{l:above}, and then argues that the control of the triangular form $ B ^{\textup{above}} (f,g)$ splits into the `global to local' and the `local' part.  
The authors of \cite{MR3285857} only had the first reduction. And, using the techniques of that paper, could prove 

\begin{priorResults}\label{t:partial}\cite{MR3285857} There holds 
$ \lvert  B ^{\textup{above}} (f,g)  \rvert \lesssim \{\mathscr H + \mathscr B _{\infty }\} \lVert f\rVert_{\sigma } \lVert g\rVert_{w} $, 
where $ \mathscr H = \mathscr A_2 ^{1/2} + \mathscr T$, and the remaining constant is the best constant in 
\begin{equation*}
\lvert  B ^{\textup{above}} (f,g)  \rvert \lesssim \mathscr B _{\infty } \sigma (I_0) ^{1/2}  \lVert g\rVert_{w} \,, 
\end{equation*}
where $ \lvert  f\rvert\le \mathbf 1_{I_0} $, and $ I_0$ is any interval.  The corresponding estimate holds for the dual from $ B ^{\textup{below}} (f,g)$. 
\end{priorResults}

This is a powerful Theorem, strongly suggesting that the $ A_2$ condition and  testing the Hilbert transform over bounded functions is sufficient for the $ L ^2 $ boundedness of $ H _{\sigma }$.  
But, there is no obvious way to deduce such a result from the Theorem above.  
Phrasing things differently, it can be very difficult to translate partial information about the triangular form $ B ^{\textup{above}} (f,g)$ to information about $ \langle H _{\sigma } f, g \rangle _{w}$, a potentially serious obstacle if a richer theory of two weight inequalities for singular integrals is to be developed.  

The \emph{parallel corona} was introduced in  \cite{parallel} to surmount this obstacle.  With it, 
the result that could be proved  the  first real variable characterization of the two weight inequality for any continuous singular integral. 

\begin{priorResults}\label{t:infty}[Lacey Sawyer Shen Uriarte-Tuero \cite{parallel}] 
There holds $ \mathscr N \simeq \mathscr A_2 ^{1/2} + \mathscr T _{\infty }$, where 
the latter constant  is the best constant in the inequalities below, uniform over all intervals $ I$, and  Borel subsets $ E\subset I$.  
\begin{gather*}
\int _{I} \lvert  H _{\sigma } \mathbf 1_{E}\rvert ^2 \; d w \le \mathscr T _{\infty } ^2 \sigma (I) \,, 
\qquad 
\int _{I} \lvert  H _{w } \mathbf 1_{E}\rvert ^2 \; d \sigma  \le \mathscr T _{\infty } ^2 w (I) \,.  
\end{gather*}
(One tests the Hilbert transform on $ \mathbf 1_{E} $, but only the weight of the interval $ I$ appears on the right.) 
\end{priorResults}

The parallel corona delays the application of Lemma~\ref{l:above}, this feature combined with a  special function theory specific to Haar expansions for non-doubling measures, were the critical ingredients. 

The parallel corona has been used to give short transparent proofs of two weight inequalities for singular integrals. See the last page of 
Hyt\"onen's survey \cite{MR3204859} and the article of Tanaka \cite{MR3232035}.

\subsubsection{}

It is natural to wonder if there are any $ L ^{p}$ analogs of the main Theorem.  
We have some clues as to how this might work, in the more complicated testing conditions of Vuorinen \cites{14122127,150405759}.  One could see that the global to local reduction would work under variants 
of these more complicated testing conditions.  The control of the local term is however a heavily Hilbertian 
argument, and so potentially very difficult to extend to an $ L ^{p}$-setting.

\section{The Remaining Part of the Global Estimate} \label{s:fe}

The last  part of the global-to-local part of the arugment is this Lemma. 

\begin{lemma}\label{l:remaining} Using the notation of \S \ref{s:global}, there holds 
\begin{equation}\label{e:remaining}
\Bigl\lvert \sum_{F\in \mathcal F} 
\sum_{I \;:\; I\supsetneq  F} \mathbb E ^{\sigma } _{I _{F}} \Delta ^{\sigma } _{I} f \cdot 
\langle H _{\sigma } (I_ F \setminus F), g_F \rangle_w \Bigr\rvert
		\lesssim 
\mathscr H \lVert f\rVert _{\sigma } \lVert g\rVert _{w}. 
		\end{equation}
\end{lemma}

Our method of proof has these elements. 
(a) Use monotonicity to pass to a positive operator. 
(b) Identify the inequality needed as an instance of a 
two weight inequality, but not for general functions, only 
one fixed function, and a derived weight $ \mu = \mu _{\sigma , w , f}$ that is 
well-adapted to the function;  
(c) Invoke the \emph{parallel corona} method to prove the desired two weight inequality. 
Along the way, we will identify simplifications of the general case of  a two weight inequality 
for a positive operator. 

\medskip 
Begin the proof by observing that 
\begin{equation*}
\Bigl\lvert 
\sum_{I \;:\; I\supsetneq  F} \mathbb E ^{\sigma } _{I _{F}} \Delta ^{\sigma } _{I} f \cdot  (I_F - F)
\Bigr\rvert \lesssim \Phi \coloneqq  \sum_{\substack{F'\in \mathcal F\\ F'\supsetneq F }} \alpha _{f} (F') \cdot \overline F'_F \,, \qquad F\in \mathcal F \,. 
\end{equation*}
where $\overline F'_F = F' \setminus F''$, with $ F''$ being the $ \mathcal F$-child of $ F'$ that contains $ F$. 
Also, by monotonicity,  the left-side of \eqref{e:remaining} is at most 
\begin{equation*}
 \sum_{F\in \mathcal F} \sum_{J ^{\ast} \in \mathcal J ^{\ast} (F)}  
P \Bigl(  \sum_{\substack{F'\in \mathcal F\\ F'\supsetneq F }} \alpha _{f} (F') \cdot \overline F'_F , J ^{\ast}    \Bigr)
 \sum_{ \substack{J \;:\; J \subset J ^{\ast} \\ \dot \pi _{\mathcal F}   J=  J ^{\ast} } }  
\Bigl\langle  \frac x {\lvert  J ^{\ast} \rvert } , h ^{\sigma } _{J} \Bigr\rangle _{\sigma } \lvert  \langle g, h ^{w} _{J} \rangle_w \rvert.   
\end{equation*}

The desired estimate is a consequence of new $ L ^2 $-estimate for the modified Poisson operator 
\begin{equation} \label{e:tilde}
\tilde P f (x,t) = \int \frac {f (y)} { t ^2 + (x-y) ^2 } \; dy
\end{equation}
which is extended to $ \tilde P f (I) = P f( x_I, \lvert  I\rvert )$.   
The relevant measure on the upper half-plane is given by 
\begin{equation} \label{e:mu}
\mu := \sum_{F\in \mathcal F} \sum_{J ^{\ast} \in \mathcal J ^{\ast} (F)}  
\delta _{ x_ {J ^{\ast} }, \lvert  J ^{\ast} \rvert } 
\sum_{ \substack{J \;:\; J \subset J ^{\ast} \\ \dot \pi _{\mathcal F}   J=  J ^{\ast} } }  \langle x, h ^{w} _{J} \rangle_w ^2 . 
\end{equation}
Finally, the estimate we need is as below, in which we have eliminated  the sum of $ J\in \mathcal J ^{\ast} $.    
\begin{equation}\label{e:n1}
\Biggl\lVert 
\sum_{\substack{F'\in \mathcal F \\F \in \textup{Ch} _{\mathcal F} (F')}} 
\alpha _{f} (F')   \sum_{J  \in \mathcal J ^{\ast} (F)}  
\tilde P ( \overline F'_F  , J   ) \cdot Q _{J   } 
\Biggr\rVert _{\mu } \lesssim \mathscr H \lVert f\rVert_ \sigma .  
\end{equation}
Here, $ \textup{Ch} _{\mathcal F} (F')$ is the collection of $ \mathcal F$-children of $ F'$, 
and $ Q_J = J \times [0, \lvert  K\rvert] $ is the Carleson box over interval $J$.

This last inequality is in fact \emph{universal}, in that we could fix the measuer $ \mu $, 
replace  $ f$ by an arbitrary function, and the  inequality is still true. But this fact is not 
needed.  And, we can use the fact that $ f$ and the measure $ \mu $ are related through 
the stopping data, to simplify the proof of \eqref{e:n1}.  

Our knowledge of two weight estimates suggest that the inequality \eqref{e:n1} is easiest to prove 
by duality, and using the joint stopping data on $ f$ and the dual function $ \gamma  \in L ^2 (\mathbb R ^2 _+, \mu )$, 
a technique refered to as the \emph{parallel corona.} 
We will reduce the inequality \eqref{e:n1} to two testing inequalities. One will be a reformulation of the energy 
inequality and the other will be a consequence of the $ A_2 $ condition. 

By duality,  the inequality we  establish is 
\begin{equation}\label{e:n2}
\sum_{\substack{F'\in \mathcal F \\F \in \textup{Ch} _{\mathcal F} (F')}} 
\alpha _{f} (F')   \sum_{J  \in \mathcal J ^{\ast} (F)}  
\tilde P ( \overline F'_F  , J   )  \int _{Q_J} \gamma \; d \mu  \lesssim 
\mathscr H \lVert f\rVert _{\sigma } \lVert \gamma \rVert _{\mu }.  
\end{equation}
Here, $ \gamma $ is a non-negative function, supported  on a Carleson cube $ Q _{J_0} $, 
where $ J_0 \in \mathcal J ^{\ast} := \bigcup _{F\in \mathcal F} \mathcal J ^{\ast} (F) $. 
We construct stopping intervals $ \mathcal G$ for  $ \gamma $, by initializing $ \mathcal G = \{J_0\}$, 
and setting $ \alpha _{J_0} (g) = \mathbb E _{Q _{J_0}} ^{\mu } \gamma $.  
In the recursive step, for minimal $ J\in \mathcal G$, we add to $ \mathcal G$ the maximal 
subintervals $ J' \subsetneq J$ with $ J'\in \mathcal J ^{\ast} $ such that 
$ \alpha _{g} (J') :=  \mathbb E _{Q _{J'}} ^{\mu } \gamma > 10 \alpha _{g} (J)$.  
We let $ \pi _{\mathcal G} I$ be the minimal element of $ G$ that contains $ I$.

Now, in the sum \eqref{e:n2}, a given interval $ J$ that occurs satisfies 
either $ \pi _{\mathcal G} J \subset   F'$ or $ F'\subsetneq \pi _{\mathcal G} J$.  
(Keep in mind that there could be many intervals $ G\in \mathcal G$ that lie between $ J$ and $ F'$.)
This division splits the sum into two terms, the first is the sum over $ F'\in \mathcal F$ of 
\begin{gather}\label{e:n3}
\alpha _{f} (F')  
\sum_{\substack{F \in \textup{Ch} _{\mathcal F} (F')}} 
 \sum_{\substack{J  \in \mathcal J ^{\ast} (F)\\ \pi _{\mathcal G} J \subset F' }}  
\tilde P ( \overline F'_F  , J   )  \int _{Q_J} \gamma \; d \mu  . 
\end{gather}
And the second is sum over $ G\in \mathcal G$ of 
\begin{equation}
\label{e:n4}
\sum_{\substack{F'\in \mathcal F \\ \pi _{\mathcal G} F'=G \\ F'\neq G}}
\alpha _{f} (F')  \sum_{\substack{F \in \textup{Ch} _{\mathcal F} (F')}} 
 \sum_{\substack{J  \in \mathcal J ^{\ast} (F)\\ \pi _{\mathcal G} J=G }}  
\tilde P ( \overline F'_F  , J   )  \int _{Q_J} \gamma \; d \mu  . 
\end{equation}
The first testing inequality is this inequality, uniform over $ F'\in \mathcal F$.  
\begin{equation}\label{e:n33}
\eqref{e:n3} \lesssim \mathscr H 
\alpha _{f} (F') \sigma (F') ^{1/2} 
\Biggl[  
\sum_{\substack{G\in \mathcal G\\ \pi _{\mathcal F} G=F' }}  \alpha _{\gamma } (G) ^2 \mu (Q_G)  
\Biggr] ^{1/2} . 
\end{equation}
That this completes the bound of \eqref{e:n3} is an immediate consequence of quasi-orthogonality. 
By Cauchy-Schwarz applied to the right of \eqref{e:n33}, note that 
\begin{equation*}
\sum_{F\in \mathcal F} \alpha _{f} (F') ^2 \sigma (F')  \lesssim \lVert f\rVert _{\sigma } ^2 , 
\end{equation*}
and as well, by the construction of the stopping data for $ \gamma $, 
\begin{equation*}
\sum_{F\in \mathcal F} \sum_{\substack{G\in \mathcal G\\ \pi _{\mathcal F} G=F' }}  \alpha _{\gamma } (G) ^2 \mu (Q_G)   
\lesssim \lVert \gamma \rVert _{\mu } ^2 . 
\end{equation*}
This completes half of the proof of \eqref{e:n2}.  The other half follows from the second testing inequality: 
Uniformly in $ G\in \mathcal G$, there holds 
\begin{equation}\label{e:n44}
\eqref{e:n4} \lesssim \mathscr H 
\alpha _{\gamma } (G) \mu (Q_G) ^{1/2} 
\Biggl[
\sum_{\substack{F'\in \mathcal F \\ \pi _{\mathcal G} F'=G}}
\alpha _{f} (F')  ^2 \sigma (F') 
\Biggr] ^{1/2} . 
\end{equation}
It is bounded again by quasi-orthogonality.   It remains to prove the two testing inequalities \eqref{e:n33} and \eqref{e:n44}. 

\begin{proof}[Proof of \eqref{e:n33}] 
This is just the energy inequality.  By construction 
\begin{align*}
\eqref{e:n3} \lesssim \alpha _{f} (F')  
\sum_{F\in \textup{Ch} _{\mathcal F} (F)}   \sum_{\substack{J\in \mathcal J ^{\ast} (F) \\ \pi _{\mathcal G} J \subset F' }} 
\alpha _{\gamma } ( \pi _{\mathcal G} J)  
\tilde P  _{\sigma }( F' \setminus F,  J)  \mu ( \{( x_J , \lvert  J\rvert ) \}) . 
\end{align*}
Of course we use Cauchy-Schwarz on the right above.
Recall the definition of $ \mu $, to see that this inequality 
\begin{align*}
\sum_{F\in \textup{Ch} _{\mathcal F} (F)}   \sum_{\substack{J\in \mathcal J ^{\ast} (F) \\ \pi _{\mathcal G} J \subset F' }}  
&
\tilde P  _{\sigma }( F' \setminus F,  J) ^2 \mu ( \{( x_J , \lvert  J\rvert ) \}) 
\\& \leq 
\sum_{F\in \textup{Ch} _{\mathcal F} (F)}   \sum_{\substack{J\in \mathcal J ^{\ast} (F) \\ \pi _{\mathcal G} J \subset F' }}  
\tilde P  _{\sigma }( F' \setminus F,  J) ^2 
\sum_{J' \;:\; J'\subset J} \langle x, h ^{w} _{J'} \rangle_w ^2 
\lesssim \mathscr H ^2  \sigma (F') 
\end{align*}
is simply a reformulation of the energy inequality \eqref{e:energy}.  

The other part of the application of Cauchy-Schwarz is 
\begin{align*}
\sum_{F\in \textup{Ch} _{\mathcal F} (F)}   \sum_{\substack{J\in \mathcal J ^{\ast} (F) \\ \pi _{\mathcal G} J \subset F' }} 
&
\alpha _{\gamma } ( \pi _{\mathcal G} J)   ^2\mu ( \{( x_J , \lvert  J\rvert ) \}) 
\lesssim \sum_{\substack{G\in \mathcal G\\ \pi _{\mathcal F} G \in  \{ F' \} \cup \textup{Ch} _{\mathcal F} (F')}}   
\alpha _{\gamma } ( \pi _{\mathcal G} J)  ^2 \mu (  Q _{G}).  
\end{align*}
This completes the proof of \eqref{e:n33}. 
\end{proof}

\begin{proof}[Proof of \eqref{e:n44}]  
In \eqref{e:n4}, we dominate $ \int _{Q_J} \gamma \; d \mu  \leq 10 \mathbb E ^{\mu } _{G} \gamma \cdot  \mu (Q_J)$, 
and then express \eqref{e:n4} using the dual to the operator $ \tilde P $ defined in \eqref{e:tilde}.  
We have 
\begin{align}
\eqref{e:n4} & \lesssim 
\mathbb E ^{\mu } _{G} \gamma  \times 
 \sum_{\substack{F'\in \mathcal F \\ \pi _{\mathcal G} F'=G \\ F'\neq G}}
\alpha _{f} (F')  \sum_{\substack{F \in \textup{Ch} _{\mathcal F} (F')}} 
 \sum_{\substack{J  \in \mathcal J ^{\ast} (F)\\ \pi _{\mathcal G} J=G }}  
\tilde P ( \overline F'_F  , J   )  \mu  (Q_J) 
\\
&= 
\mathbb E ^{\mu } _{G} \gamma  \times  \int _{G} 
  \sum_{\substack{F'\in \mathcal F \\ \pi _{\mathcal G} F'=G \\ F'\neq G}}
 \sum_{\substack{F \in \textup{Ch} _{\mathcal F} (F')}} \alpha _{f} (F')  \cdot F' _{F}
 \sum_{\substack{J  \in \mathcal J ^{\ast} (F)\\ \pi _{\mathcal G} J=G }}    \tilde P ^{\ast}_ \mu  (Q_J) \; d \sigma  
\end{align}
Apply Cauchy-Schwartz in the variable $ F'$, and $ L ^2 (\sigma )$.  One of the terms that result is 
\begin{align*}
 \sum_{\substack{F'\in \mathcal F \\ \pi _{\mathcal G} F'=G}} \alpha _{f} (F') ^2 \sigma  (F') . 
\end{align*}
Compare to the right side of \eqref{e:n44}. 
The other term is the following   inequality, holding uniformly in $ G\in \mathcal G$: 
\begin{equation}  \label{e:n6}
\int _{G} 
  \sum_{\substack{F'\in \mathcal F \\ \pi _{\mathcal G} F'=G \\ F'\neq G}}
\Biggl[
\sum_{\substack{F \in \textup{Ch} _{\mathcal F} (F')}} 
 \sum_{\substack{J  \in \mathcal J ^{\ast} (F)\\ \pi _{\mathcal G} J=G }}   
  F'_F   \cdot \tilde P ^{\ast}_ \mu  (Q_J) 
\Biggr] ^2 \; d \sigma \lesssim \mathscr A_2 \mu (Q_G).  
\end{equation}
As the inequlaity shows, this follows from the $ A_2$ condition.  
  
An obstacle to a proof is that the sets  $ Q_J$ overlap.  This is addressed with the 
definition  $ W _{J} ^{j } = J\times  (2 ^{- j -1}, 2 ^{- j }]$, for $ j\geq 0$.  These sets 
are disjoint in $ j$ and $ J$.   We will then show that for each $ F'\in \mathcal F$, 
\begin{equation}\label{e:n7}
\int _{ \mathbb R } 
\Biggl[
\sum_{\substack{F \in \textup{Ch} _{\mathcal F} (F')}} 
 \sum_{\substack{J  \in \mathcal J ^{\ast} (F)\\ \pi _{\mathcal G} J=G }}   
  F'_F   \cdot \tilde P ^{\ast}_ \mu  (W ^{j}_J) 
\Biggr] ^2 \; d \sigma \lesssim  2 ^{-j}\mathscr A_2   
\sum_{\substack{F \in \textup{Ch} _{\mathcal F} (F')}} 
 \sum_{\substack{J  \in \mathcal J ^{\ast} (F)\\ \pi _{\mathcal G} J=G }}   \mu  (W ^{j}_J) .  
\end{equation}
This easily implies \eqref{e:n6}.

Two additional summing variables are convenient.  
For integers $ k\geq r$, we restrict the sum to $ J\in \mathcal J ^{\ast} (F)$ with $ 2 ^{k} \lvert  J\rvert= \lvert  F\rvert  $. 
And, for integers $ \ell \geq k (1- \epsilon )$, we further require that 
\begin{equation}  \label{e:n7}
 2 ^{\ell -1} \lvert  J\rvert \leq \textup{dist} (J, \partial F) < 2 ^{\ell } \lvert  J\rvert.   
\end{equation}
By goodness, $ \ell \ge k (1 - \epsilon )$, but there is in general no other condition that we have here.  
Then, we prove this estimate, which  is \eqref{e:n7}, with these two additional restrictions on $ J$.  
Uniformly in $ F' \in \mathcal F$, 
\begin{equation}\label{e:n8}
\int _{ \mathbb R } 
\Biggl[
\sum_{\substack{F \in \textup{Ch} _{\mathcal F} (F')}} 
 \sum_{\substack{J  \in \mathcal J ^{\ast} (F)\\ \pi _{\mathcal G} J=G \\  2 ^{k} \lvert  J\rvert= \lvert  F\rvert ,\  
 \textup{\eqref{e:n7} holds }}}    
  F'_F   \cdot \tilde P ^{\ast}_ \mu  (W ^{j}_J) 
\Biggr] ^2 \; d \sigma \lesssim  2 ^{-j - k - \ell }\mathscr A_2   
\sum_{\substack{F \in \textup{Ch} _{\mathcal F} (F')}} 
 \sum_{\substack{J  \in \mathcal J ^{\ast} (F)\\ \pi _{\mathcal G} J=G }}   \mu  (W ^{j}_J) .  
\end{equation}

In \eqref{e:n8}, there are at most $ 2 ^{\ell }$ intervals $ J$. We can therefore pass the square inside the 
sum, at cost of a factor of $ 2 ^{\ell }$.  But, 
\begin{align}
\int _{ F'_F }   \tilde P ^{\ast}_ \mu  (W ^{j}_J)  (x) ^2 \; d \sigma (x) 
& \lesssim 
\mu (W ^{j}_J)  \int _{ F'_F } \int _{W ^{j}_J}  \frac 1 { [ y_2 ^2 + \lvert  x - y_1\rvert ^2 ] ^2  } \; d \mu (y) \, d \sigma (x) 
\\
& \lesssim \frac { \mu (W ^{j}_J)  ^2 } { 2 ^{2 \ell } \lvert  J\rvert ^2  }  \tilde P _{\sigma } (F'_F, J) 
\\ \label{e:n9}
& \lesssim 2 ^{-2 \ell - 2 j}   \mu (W ^{j}_J)    w (J) \tilde P _{\sigma } (F'_F, J)  
\lesssim  2 ^{-2 \ell - 2 j} \mathscr A_2   \mu (W ^{j}_J)  .  
\end{align}
Here, we have used Cauchy-Schwartz, followed by the estimate below, which holds for $ x \in F'_F$,  
\begin{equation*}
 \int _{W ^{j}_J}  \frac 1 { [ y_2 ^2 + \lvert  x - y_1\rvert ^2 ] ^2  } \; d \mu (y) 
 \lesssim \frac {\mu (W ^{j}_J)  } { 2 ^{2 \ell } \lvert  J\rvert ^2  ( \lvert  J\rvert ^2 + \lvert  x- x_J\rvert ^2   ) } 
\end{equation*}
Then, besides disjointness, the sets $ W ^{j}_J$ enjoy the estimate $ \mu (W ^{j}_J)   \lesssim 2 ^{-2j} \lvert  J\rvert ^2 w (J) $, 
which follows from the definition of $ \mu $ in \eqref{e:mu}, and the estimate $ \lvert  x, h ^{w} _J\rvert ^2 
\leq \lvert  J\rvert ^2 w (J) $.  Finally, we just appeal to the $ \mathscr A_2$ condition.  
The bound in \eqref{e:n9} is multiplied by $ 2 ^{ \ell }$, to prove \eqref{e:n8}. This finishes the proof.

\end{proof}

\subsection{Context and Discussion}

The inequality \eqref{e:n1} is universal.  This was first proved in \cite{MR3285857}, 
in the case tthat the weights did not share a common point mass.  It was down by appealing to the 
Sawyer theorem \cite{MR930072} on two weight inequalities for the Poisson operator.   
This technique does not allow common point masses, however. 
Addressing this, Hyt\"onen \cite{13120843} found a clever way to use dyadic approximates to the 
`Poisson operator with holes,' by using dyadic approximates to an arbitrary interval, and 
prinving a novel dyadic two weight inequality. 

The proof herein does not attempt to prove the \emph{universal} form of \eqref{e:n1}. 
Indeed, this inequality is not needed.  Indeed, the close relationship between the function $ f$, 
and the derived measure $ \mu $ in \eqref{e:mu} permits a short self-contained proof.

\section{The Stopping Form} \label{s:stop}

The last step in the proof of Theorem~\ref{t:above}, hence in the proof of the main theorem,
is to show that the local inequality \eqref{e:BF} holds. 
Using the discussion  at the end of the previous section, this amounts to controlling the \emph{stopping form}.  
Given an interval $F\in \mathcal F$, the stopping form is 
\begin{equation}  \label{e:stop}
B ^{\textup{stop}} _{F} (f,g) \coloneqq  
 \sum_{I \::\:  \pi _{\mathcal F} I=F} \sum_{J \::\: J\Subset I_J, \dot \pi _{\mathcal F} J=F} 
 \mathbb E ^{\sigma } _{I_J} \Delta ^{\sigma }_I f \cdot \langle H _{\sigma } (I_0 - I_J) ,  \Delta ^{w} _{J} g\rangle _{w}\,. 
\end{equation}

\begin{lemma}\label{l:stop<}  There holds for each $ F\in \mathcal F$,  
\begin{equation}\label{e:stop<}
\lvert  B ^{\textup{stop}} _{I_0} (f,g) \rvert \lesssim \mathscr H \lVert P ^{\sigma } _{F}f\rVert_{\sigma } \lVert Q ^{w} _{F} g\rVert_{w} \,.  
\end{equation}
\end{lemma}

The stopping form arises naturally in any proof of a $ T1$ theorem using Haar or other bases.  
In the non-homogeneous case, or in the $ Tb$ setting, where (adapted) Haar functions are important tools, it frequently appears in more or less this form.  
Regardless of how it arises, the stopping form is treated as a error, in that it is bounded by some simple geometric series, obtaining decay as e.\thinspace g.\thinspace the ratio $ \lvert  J\rvert/\lvert  I\rvert  $ is held fixed. 
(See for instance \cite{10031596}*{(7.16)}.) 

These sorts of arguments, however, implicitly require some additional hypotheses, such as the  weights being  mutually $ A _{\infty }$. 
Of course, the two weights above can be mutually singular. There is no \emph{a priori}  control of the stopping form in terms of simple parameters  like $ \lvert  J\rvert/\lvert  I\rvert  $, even supplemented by  additional pigeonholing of various parameters.

Our method is inspired by proofs of Carleson's Theorem on Fourier series \cites{lacey-thiele-carleson,fefferman,MR0199631}, and has one particular precedent in the current setting, a much simpler bound for the stopping form in \cite{parallel}

\subsection{Admissible Pairs}
We can assume that $ f = P ^{\sigma } _{F} f$ and $ g = Q ^{w} _{F} g$. For all 
pair of intervals $ J \Subset I \subset F$ that we need to consider, we have  $ \dot \pi _{\mathcal F} J=F$, and hence 
by the Energy Stopping condition, there holds 
\begin{equation}\label{e:Estop}
P (\sigma (F - I_J), I_J) E (w, I_J) ^2 w (I_J) \leq C \sigma (I_J).  
\end{equation}
For if not, by monotonicity \eqref{e:mono0}, we would have that the interval $ I_J$ would be an energy stopping interval, 
hence $ I_J\in \mathcal F$, and $ \dot \pi _{\mathcal F} J= I_J$.  It is this condition that is our starting point for the recursion.

A range of decompositions  of the stopping form necessitate a somewhat heavy notation that we introduce here. 
The individual summands in the stopping form involve four distinct intervals, namely $ F, I, I_J$, and $ J$.  
The interval $ F$ will not change in this argument, and the pair $ (I,J)$ determine $ I_J$. 
Subsequent decompositions are  easiest to phrase as actions on collections $ \mathcal Q$ of pairs of intervals 
$
Q= (Q_1, Q_2)  $ with $F\supset Q_1\Supset Q_2  
$. 
(The letter $ P $ is already taken for the Poisson integral.) 
And we consider the bilinear forms 
\begin{equation*}
B _{\mathcal Q} (f,g) \coloneqq  \sum_{Q\in \mathcal Q}    \mathbb E ^{\sigma } _{(Q_1)_{Q_2}}  \Delta ^{\sigma }_ {Q_1} f \cdot \langle H _{\sigma } (F - (Q_1)_{Q_2}) ,  \Delta ^{w} _{Q_2} g\rangle _{w} \,. 
\end{equation*}
We will have the standing assumption that for all collections $ \mathcal Q$ that we consider are \emph{admissible}.

\begin{definition}\label{d:admiss} A collection of pairs $ \mathcal Q$ is \emph{admissible} if it meets these criteria. 
For any $ Q = (Q_1, Q_2) \in \mathcal Q$, 
\begin{enumerate}
\item    $Q_2\Subset Q_1\subset F$,  and both $ Q_1 $ and $ Q_2$ are good. 
 \item (convexity in $ Q_1$) If  $ Q''\in \mathcal Q$ with $ Q''_2=Q_2$ and  $ Q_1'' \subset I\subset Q_1$, with $ I$ good, then there is a $ Q' \in \mathcal Q$ with $ Q '_1= I$ and $ Q_2'=Q_2$.  
\end{enumerate}
The first  property is self-explanatory. The second property is convexity in $ Q_1$,  subject to goodness, holding $ Q_2$ fixed, which is used in the estimates 
on the stopping form which conclude the argument. 
A third property is described below. 

We exclusively use the notation $ \mathcal Q _{k}$, $ k=1,2$ for the collection of intervals $ \bigcup \{ Q_k \::\: Q\in \mathcal Q\}$, not  counting multiplicity.  Similarly, set $ \tilde {\mathcal Q}_1 \coloneqq  \{ (Q _{1}) _{Q_2} \::\: Q\in \mathcal Q\}$, and $ \tilde Q_1 \coloneqq  (Q_1)_{Q_2}$.  
\begin{enumerate}
\item[(3)]  Every interval $ Q_2\in \mathcal Q_2$ satisfies $ \dot\pi _{\mathcal F} Q_2 =F$
(And so,  every $  \tilde {\mathcal Q}_1$ has $ \mathcal F$-parent $ F$.)  
\end{enumerate}

\end{definition}

The last requirement comes from the assumption that the  functions $ f$ and $ g$ be adapted to $ \mathcal F _{\textup{energy}} (F)$. 
We will be appealing to different Hilbertian arguments below, so we prefer to make this  an assumption about the pairs rather than the functions $ f, g$. 
The Hilbert space will be the space of good functions in $ L ^2 (\sigma )$ and $ L ^2 (w)$.  

Typically, one only ever needs goodness of the \emph{small} interval, in this case $ Q_2$. We will use the term $ \textup{size} (\mathcal Q)$ below, 
in which it will be apparent that goodness of the intervals $Q_1 $ will be helpful.  Namely, at this point goodness is used to 
as in the monotonicity principle, to estimate off-diagonal inner products involving the Hilbert transform by Poisson averages, and to 
regularize Poisson averages. Both are made more explicit in \S\ref{s:UB}.

The stopping form is obtained with the admissible collection of pairs given by 
\begin{equation} \label{e:Q0}
\mathcal Q_0 =\{ (I, J) \::\:    J\Subset I\subset F\,,  \textup{ $ I$ and $ J$ are good},\    \dot \pi _{\mathcal F} J=F \} \,. 
\end{equation}
There holds $ B ^{\textup{stop}} _{F} (f,g) = B _{\mathcal Q_0} (f,g)$.  

\medskip 

There is a very important notion of the size of $ \mathcal Q$.  
\begin{equation}\label{e:size}
\textup{size} (\mathcal Q) ^2  \coloneqq  
\sup _{ K \in \tilde {\mathcal Q}_1 \cup \mathcal Q_2}   \frac {\mathsf P(\sigma  (F -K), K) ^2 } {\sigma (K) \lvert  K\rvert ^2  } 
\sum_{J \in \mathcal Q _2 \::\: J\subset K}  \langle  x , h ^{w} _{J} \rangle_{w} ^2 \,. 
\end{equation}
For admissible $ \mathcal Q$, there holds $ \textup{size} (\mathcal Q) \lesssim \mathscr H$, as follows \eqref{e:Estop}.

More definitions follow.  
Set the norm $  \mathbf B _{\mathcal Q}$ of the bilinear form $ \mathcal Q$ to be the best constant in the inequality 
\begin{equation*}
\lvert  B _{\mathcal Q} (f,g)\rvert \le \mathbf B _{\mathcal Q} \lVert f\rVert_{\sigma } \lVert g\rVert_{w } \,.  
\end{equation*}
Thus, our goal is show that $ \mathbf B _{\mathcal Q} \lesssim \textup{size} (\mathcal Q)$ for admissible $ \mathcal Q$, but we will only be able to do this directly in the case that the pairs $ (Q_1, Q_2)$ are weakly decoupled in a collection $ \mathcal Q$. 
The relevant decoupling is precisely described in \S\ref{s:UB}.

Say that collections of pairs $ \mathcal Q ^{j}$, for $ j\in \mathbb N $, are \emph{mutually orthogonal} if  on the one hand, the collections 
$ (\mathcal Q ^{j}) _{2}$, of second coordinates of the pairs, are pairwise disjoint,  and on the other, that the collections$ \widetilde {(\mathcal Q ^{j})}_1$ are pairwise disjoint. 
The concept has to be different in the first and second coordinates of the pairs, due to the different role of the intervals $ \tilde Q_1$ and $ Q_2$, 
which comes up again in the next paragraph. 

The meaning of  mutual orthogonality is best expressed through the norm of the associated bilinear forms. 
Under the assumption that $ B _{\mathcal Q} = \sum_{j\in \mathbb N } B _{ \mathcal Q ^{j}}$, 
and that the $ \{\mathcal Q ^{j} \::\: j\in \mathbb N \} $ are mutually orthogonal, the following essential inequality holds. 
\begin{equation}\label{e:subadd}
\mathbf B _{\mathcal Q} \le  \sqrt 2\sup _{j\in \mathbb N } \mathbf B _{\mathcal Q ^{j}} \,. 
\end{equation}
Indeed, for $ j\in \mathbb N $, let $ \Pi ^{w} _{j}$ be the projection onto the linear span of the Haar functions $\{ h ^{w} _{J} \::\: J\in \mathcal Q ^{j}_2\}$, and use a similar notation for $ \Pi ^{\sigma  }_j$.  We then have the two inequalities 
\begin{equation*}
\sum_{j \in \mathbb N } \lVert \Pi ^{w}_j g \rVert_{w} ^2 \le \lVert g\rVert_{w} ^2 \,, \qquad 
\sum_{j\in \mathbb N } \lVert \Pi ^{\sigma }_j f \rVert_{\sigma } ^2 \le  2\lVert f\rVert_{\sigma} ^2 \,. 
 \end{equation*}
Since a given interval $ I$ can be in two collections $ \mathcal Q^j_1$, we have the factor of $ 2$ in the second inequality.  
Therefore, we have 
\begin{align*}
\lvert  B _{\mathcal Q} (f,g)\rvert  & \le \sum_{j \in \mathbb N }  \lvert  B _{\mathcal Q ^{j}} (f,g) \rvert 
\\
&=  \sum_{j\in \mathbb N }\lvert   B _{\mathcal Q ^{j}} (\Pi ^{\sigma }_jf, \Pi ^{w} _{j}g) \rvert 
\\
&\le \sum_{j\in \mathbb N } \mathbf B _{\mathcal Q ^{j}}  \lVert \Pi ^{\sigma }_j f \rVert_{\sigma }
 \lVert \Pi ^{w}_j g \rVert_{w} 
 \le   \sqrt 2 \sup _{j\in \mathbb N } \mathbf B _{\mathcal Q ^{j}}   \cdot \lVert f\rVert_{\sigma   } \lVert g\rVert_{w} \,. 
\end{align*}
This proves \eqref{e:subadd}.

\subsection{The Recursive Argument}

This is the essence of the matter.

\begin{lemma}\label{l:Decompose}[Size Lemma] 
An admissible  collection of pairs $ \mathcal Q$  can be partitioned into  collections 
$ \mathcal Q ^{\textup{large}}$ and admissible  $ \mathcal Q ^{\textup{small}} _{t}$, 
for $  t \in \mathbb N $ such that 
\begin{gather}\label{e:BQ<}
\mathbf B _{\mathcal Q} \le 
C\textup{size} (\mathcal Q) + (1+ \sqrt 2) 
\sup _t
\mathbf B _{ \mathcal Q ^{\textup{small}} _{t}} \,,
\\ \label{e:small2}
\text{and} \quad \sup _{t\in \mathbb N }\textup{size} ( \mathcal Q ^{\textup{small}} _{t}) \le \tfrac 14   \textup{size} (\mathcal Q) \,.
\end{gather}
Here, $ C>0 $ is an absolute constant. 
\end{lemma}

The point of the lemma is that all of the constituent parts are better in some way, and that the right hand side of \eqref{e:BQ<} involves  a favorable  supremum. We can quickly prove the main result of this section. 

\begin{proof}[Proof of Lemma~\ref{l:stop<}] 
The stopping form of this Lemma is of the form $ B _{\mathcal Q} (f,g)$ for admissible choice of $ \mathcal Q$, 
with $ \textup{size} (\mathcal Q) \le C\mathscr H$, as we have noted in \eqref{e:Q0}.  
Define 
\begin{equation*}
\zeta (\lambda ) \coloneqq  \sup \{  \mathbf B _{\mathcal Q} \::\: \textup{size} (\mathcal Q) \le C \lambda \mathscr H \}\,, \qquad 0< \lambda \le 1\,,
\end{equation*}
where $ C>0$ is a sufficiently large, but absolute constant, and the supremum is over admissible choices of $ \mathcal Q$.  
We are free to assume that $ \mathcal Q_1$ and $ \mathcal Q_2$ are further constrained to be in some fixed, but large, collection of intervals $ \mathcal I$. 
Then, it is clear that $ \zeta (\lambda )$ is finite, for all $ 0< \lambda \le 1$.
Because of the way the constant $ \mathscr H$ enters into the definition, it remains to show that $ \zeta (1)$ admits an absolute upper bound, independent of how $ \mathcal I$ is chosen. 

\smallskip 

It is the consequence of Lemma~\ref{l:Decompose} that there holds 
\begin{align}
\zeta (\lambda) &\le C \lambda  +  (1+ \sqrt 2)  \zeta ( \lambda /4 )\,, \qquad 0 < \lambda \leq 1 \,. 
\end{align}
Iterating this inequality beginning at $ \lambda  =1$ gives us 
\begin{align*}
\zeta (1 )& \le C +   (1+\sqrt 2)\zeta (1/4) \le \cdots \le C \sum_{t=0} ^{\infty } \bigl[\tfrac {1+ \sqrt 2} 4\bigr]  ^{t} \le  4C  \,. 
\end{align*}
So we have established an absolute upper  bound on $ \zeta (1)$.
\end{proof}

\subsection{Proof of Lemma~\ref{l:Decompose}}

  We restate the conclusion of Lemma~\ref{l:Decompose} to more closely follow the  line of argument to follow. 
The collection $ \mathcal Q$ can be partitioned into two collections $ \mathcal Q ^{\textup{large}}$ and $ \mathcal Q ^{\textup{small}}$ 
such that 
\begin{enumerate}
\item  $ \mathbf  B _{\mathcal Q ^{\textup{large}}}   \lesssim   \tau    $, where $ \tau \coloneqq  \textup{size} (\mathcal Q) $. 

\item $ \mathcal Q ^{\textup{small}} =\mathcal Q ^{\textup{small}}_1  \cup \mathcal Q ^{\textup{small}}_{2} $.  

\item The collection $\mathcal Q ^{\textup{small}}_1 $ is admissible, 
 and $ \textup{size} (\mathcal Q ^{\textup{small}} _1) 
\le \frac \tau 4  $.

\item For a collection of  dyadic intervals $ \mathcal L$,   the collection $ \mathcal Q ^{\textup{small}}_2$ is the union of mutually orthogonal admissible  collections $  \mathcal Q ^{\textup{small}}_{2,L}$, for $ L\in \mathcal L $,  with 
\begin{equation*}
 \textup{size} (\mathcal Q ^{\textup{small}} _{2,L} )  
\le  \tfrac \tau 4  \,, \qquad L\in \mathcal L\,. 
\end{equation*}

\end{enumerate}
Thus, we have by inequality \eqref{e:subadd} for mutually orthogonal collections, 
\begin{align*} 
\mathbf B _{\mathcal Q} & \le \mathbf B _{\mathcal Q ^{\textup{large}}}
+ \mathbf B _{\mathcal Q ^{\textup{small}}_1 \cup  \mathcal Q ^{\textup{small}}_2 }
\\ & \le 
 \mathbf B _{\mathcal Q ^{\textup{large}}} 
+ \mathbf B _{\mathcal Q ^{\textup{small}}_1 } + \mathbf B _{ \mathcal Q ^{\textup{small}}_2 }
\\& \le 
C \tau + (1+ \sqrt 2) \max\bigl\{\mathbf B _{\mathcal Q ^{\textup{small}}_1}, 
\sup _{ L\in \mathcal L } 
\mathbf B _{\mathcal Q ^{\textup{small}}_{2, L}}  \bigr\}  \,. 
\end{align*}
This, with the properties of size listed above prove Lemma~\ref{l:Decompose} as stated, after a trivial re-indexing. 

\bigskip 
In a manner similar to the argument of \S \ref{s:global}, there is   an induced measure on the upper half-plane that 
is relevant to our considerations.  This time it is given by 
\begin{equation*}
\mu _{\mathcal Q} = \mu \coloneqq  \sum_{J \in \mathcal Q_2\::\: J\subset F} \langle x, h ^{w} _{J} \rangle_w ^2 \delta _{ (x_J, \lvert  J\rvert )}\,, \qquad \textup{$ x_J$ is the center of $ J$.} 
\end{equation*}
The tent over $ L$ is the triangular region $ T_L \coloneqq  \{  (x,y) \::\:  \lvert  x-x_L\rvert \le \lvert  L\rvert - y  \}$, so that 
\begin{equation*}
\mu (T_L) =  \sum_{J \in \mathcal Q_2 \::\: J\subset L} 
\langle   x , h ^{w } _{J} \rangle_{w} ^2 \,. 
\end{equation*}
Observe that 
\begin{equation*}
\textup{size} (\mathcal Q) ^2  = 
\sup _{ K \in \tilde {\mathcal Q}_1 \cup \mathcal Q_2}   \frac {\mathsf P(\sigma  (F -K), K) ^2 } {\sigma (K) \lvert  K\rvert ^2  } 
\mu (T_K).  
\end{equation*}

All else flows from this construction of  a  subset  $ \mathcal L $ of dyadic subintervals of $ F$.  
The initial intervals in $ \mathcal L$ are  the minimal  intervals $ L  \in \tilde {\mathcal Q}_1 \cup \mathcal Q_2$  such that 
\begin{equation} \label{e:Kdef}
\frac {\mathsf P(\sigma  (F- L) , L) ^2   } {\lvert  L\rvert ^2  }
\mu (T_L) 
\ge \frac {\tau ^2 } {16}  \sigma (L)\,.  
\end{equation}
Since $ \textup{size} (\mathcal Q)=\tau $, there are such intervals $ L$. 

Initialize $ \mathcal S $ (for `stock' or `supply') to be all the dyadic intervals in $   \tilde {\mathcal Q}_1 \cup \mathcal Q_2 $ which 
strictly contain some interval in $ \mathcal L$. 
In the recursive step,   let $\mathcal L'$ be the minimal elements $ S\in \mathcal S$ such that 
\begin{equation}\label{e:Lconstruct}
\mu (T_S)
\ge  \rho    \sum_{\substack{L\in \mathcal L \::\: L\subset S\\ \textup{$ L$ is maximal} }} \mu (T_L)\,, \qquad  \rho =  \tfrac  {17} {16} \,. 
\end{equation}
(The inequality would be trivial if $ \rho =1$.)  
If $ \mathcal L'$ is empty the recursion stops. 
Otherwise,  update $ \mathcal L \leftarrow \mathcal L \cup \mathcal L'$, 
and 
$
\mathcal S \leftarrow \{ K\in \mathcal S \::\:  K \not\subset L\ \forall L\in \mathcal L\} 
$.
See Figure~\ref{f:tents}.

Once the recursion stops, report the collection $ \mathcal L$.   It has this crucial  property: 
For $ L\in \mathcal L$, and integers $ t\ge 1$, 
\begin{equation} \label{e:ddecay}
\sum_{ L' \::\: \pi _{\mathcal L} ^{t} L'=L}  \mu (T _{L'})
\le \rho ^{-t} \mu (T _{L}) \,. 
\end{equation}
Indeed, in the case of $ t=1$,  is  a criteria for membership in $ \mathcal L$, and a simple induction proves the statement for all $ t\ge 1$.

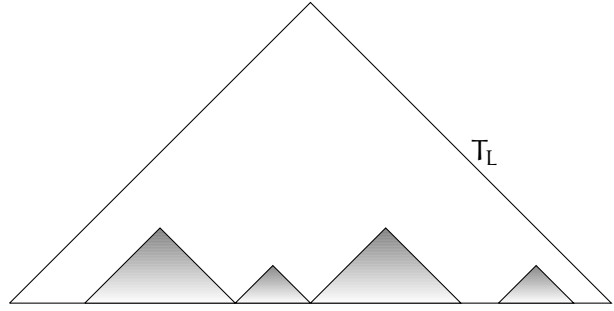
\begin{figure}
 \begin{tikzpicture}
 \draw (0,0) -- (4,4) -- (8,0)  node[right,midway] () {$ T_L$} -- (0,0); 
 \shadedraw (1,0) -- (2,1) -- (3,0) -- (1,0); 
  \shadedraw (3,0) -- (3.5,.5) -- (4,0) -- (3,0); 
   \shadedraw (4,0) -- (5,1) -- (6,0) -- (4,0);  
   \shadedraw (6.5,0) -- (7,.5) -- (7.5,0) -- (6.5,0);  
 \end{tikzpicture}
\caption{The shaded smaller tents have been selected, and $ T_L$ is the minimal tent with $ \mu (T_L)$ larger than $ \rho $ times 
the $ \mu $-measure of the shaded tents.}
\label{f:tents}
\end{figure}

The decomposition of $ \mathcal Q$ is based upon the relation of the pairs to the collection $ \mathcal L$, namely a pair $ \tilde Q_1 ,Q_2$ can (a) both have the same parent in $ \mathcal L$; (b) have distinct parents in $ \mathcal L$; (c) $ Q_2$ can have a parent in $ \mathcal L$, but  not $ \tilde Q_1$; and  (d) $ Q_2$ does not have a parent in $ \mathcal L$.  

A particularly vexing aspect of the stopping form is the linkage between the martingale difference on $ g$,  which is given by $ J$,  and the argument of the Hilbert transform, $ F- I_J$.  
The  `large' collections constructed below will, in a certain way,  decouple the   $ J$ and the $ F-I_J$, enough so that norm of the associated bilinear form can be estimated by the size of $ \mathcal Q$. 

In the `small' collections, there is however no decoupling, but critically,  the size of the collections is smaller, and  we only have to estimate 
the maximal operator norm among the small collections.

\subsubsection*{Pairs comparable to $ \mathcal L$}
Define 
\begin{equation*}
\mathcal Q   _{L,t} \coloneqq  
\{  Q\in \mathcal Q \::\:  \pi _{\mathcal L} \tilde Q_1 = \pi ^{t} _{\mathcal L} Q_2=L \} \,, \qquad L\in \mathcal L\,,\ t\in \mathbb N \,.  
\end{equation*}
These are admissible collections, as the convexity property in $ Q_1$, holding $ Q_2$ constant, is clearly inherited from $ \mathcal Q$. 
Now, observe that for each $ t\in \mathbb N $, the collections $ \{ \mathcal Q  _{L,t} \::\: L\in \mathcal L\}$ 
are mutually orthogonal.   
The collection of intervals $ ( \mathcal Q  _{L,t}) _{2}$ are obviously disjoint in $ L\in \mathcal L$, with $ t\in \mathbb N $ held fixed.  
And, since membership in these collections is determined in the first coordinate by the interval $ \tilde  Q_1$, and the two children of $ Q_1$ can have two different parents in $ \mathcal L$, a given interval $ I$ can appear in at most two collections  $ (\widetilde{\mathcal Q  _{L,t}}) _{1}$, 
as $ L\in \mathcal L$ varies, and $ t \in \mathbb N $  held fixed.

Define $  \mathcal Q ^{\textup{small}}_1 $ to be  the union over $ L\in \mathcal L$ of the collections 
\begin{equation*}
  \mathcal Q ^{\textup{small}} _{L,1} \coloneqq  \{ Q \in  \mathcal Q _ { L , 1}    \::\: \tilde Q_1 \neq L\} \,. 
\end{equation*}
Note in particular that we have only  allowed $ t=1$ above, and $ \tilde Q_1 =L$ is not allowed. 
For these collections, we need only verify that 

\begin{lemma}\label{l:1small} There holds 
\begin{equation} \label{e:small1<}
\textup{size}  (\mathcal Q  ^{\textup{small}}_ {L,1} ) \le \sqrt { (\rho -1)}  \cdot \tau =\frac \tau {4}  \,, \qquad L\in \mathcal L \,,\ t \in \mathbb N \,.  
\end{equation}
\end{lemma}
\begin{proof}
An interval $ K\in \widetilde {(\mathcal Q  ^{\textup{small}}_ {L,1}) }_1\cup \mathcal Q_2 $ is not in $ \mathcal L$, by construction. 
Suppose that $ K$ does not contain any interval in $ \mathcal L$.  By the selection of the initial intervals in $ \mathcal L$, 
the minimal intervals in $ \tilde {\mathcal Q}_1 \cup \mathcal Q_2$ which satisfy \eqref{e:Kdef}, 
it follows that the interval $ K$ must fail  \eqref{e:Kdef}.  And so we are done.  

Thus, $ K$ contains some element of $ \mathcal L$, whence the inequality \eqref{e:Lconstruct} must fail. 
Namely, rearranging that inequality, and using the measure $ \mu $ associated with $\mathcal Q  ^{\textup{small}}_ {L,1} $,  
\begin{align} \label{e:18}
\mu _{\mathcal Q  ^{\textup{small}}_ {L,1} } (T_L)
& \le  (\rho-1) \sum_{\substack{L'\in \mathcal L \::\: L'\subset K\\ L' \textup{ is maximal}}}  \mu (T_L) 
\\
&\le \frac 1 {16} \mu (T_L) \le  \frac {\tau ^2} {16} \cdot   \frac{\lvert  K\rvert ^2 \cdot \sigma (K)}{ \mathsf P(\sigma (L-K), K) ^2}  \,. 
\end{align}
Here, note that we begin with the measure $ \mu _{\mathcal Q  ^{\textup{small}}_ {L,1} }$; use $ \rho = 1+ \frac 1 {16}$;  and 
the last inequality follows from the definition of size.  This finishes the proof of \eqref{e:small1<}. 
\end{proof}

The collections below are the first contribution to $  \mathcal Q ^{\textup{large}}$.  
Take $  \mathcal Q ^{\textup{large}}_1 \coloneqq  \bigcup \{  \mathcal Q ^{\textup{large}} _{L,1} \::\: L\in \mathcal L\}$, where 
\begin{equation*}
  \mathcal Q ^{\textup{large}} _{L,1} \coloneqq  \{ Q \in  \mathcal Q _ { L , 1}    \::\: \tilde Q_1 = L\} \,. 
\end{equation*}
Note that Lemma~\ref{l:holes} applies to this Lemma, take the collection $ \mathcal S$ of that Lemma to be $ \{L\}$, and the 
quantity $ \eta $ in \eqref{e:S<} satisfies $ \eta \lesssim \tau = \textup{size} (\mathcal Q)$, by \eqref{e:eta<size}.
From the mutual orthogonality \eqref{e:subadd},  we then have 
\begin{equation*}
\mathbf B _{\mathcal Q ^{\textup{large}} _{1} } \le \sqrt 2 \sup _{L\in \mathcal L} \mathbf B _{\mathcal Q ^{\textup{large}} _{L,1} }  
\lesssim \tau \,. 
\end{equation*}

The collections $ \mathcal Q _{L,t}$, for $ L\in \mathcal L$, and $ t\ge 2$ are the second contribution to $  \mathcal Q ^{\textup{large}}$, namely 
\begin{equation*}
 \mathcal Q ^{\textup{large}} _{2} \coloneqq  \bigcup _{L
 \in \mathcal L} \bigcup _{t \ge 2}  \mathcal Q _ { L,t} \,. 
\end{equation*}
For them, we need to estimate $ \mathbf B _{\mathcal Q _{L,t}}$.  

\begin{lemma}\label{l:Y} There holds 
$
 \mathbf B _{\mathcal Q _{L,t}} \lesssim \rho ^{-t/2} \tau 
$. 
\end{lemma}

From this, we can conclude from \eqref{e:subadd} that 
\begin{align*}
\mathbf B _{ \mathcal Q ^{\textup{large}} _2} 
& \le \sum_{t \ge 2 } \mathbf B _{\bigcup  \{\mathcal Q   _{L,t} \::\: L\in \mathcal L \} } 
\\
& \le  \sqrt 2 \sum_{t \ge2 } \sup _{L\in \mathcal L}\mathbf B _{\mathcal Q _{L,t} } 
 \lesssim \tau \sum_{t\ge 2 } \rho  ^{-t/2} \lesssim \tau \,. 
\end{align*}

\begin{proof}
For $ L \in \mathcal L$, let  $ \mathcal S_{L}$, the $ \mathcal L$-children of $ L$.  
For each $ Q\in  \mathcal Q_{L,t} $, we must have $ Q_2 \subset \pi _{\mathcal S_L} Q_2 \subset \tilde Q_1 $. 
Then, divide the collection $ \mathcal Q _{L,t}$ into  three collections  $ \mathcal Q ^ \ell _{L,t} $, $ \ell =1,2,3$,  where  
\begin{align*}
 \mathcal Q ^{1} _{L,t} &\coloneqq  \{Q \in \mathcal Q _{L,t} \::\: Q_2 \Subset \pi _{\mathcal S_L} Q_2\} \,, 
\\
\mathcal Q ^{2} _{L,t} &\coloneqq  \{Q \in \mathcal Q _{L,t} \::\: Q_2 \not\Subset \pi _{\mathcal S_L} Q_2 \Subset \tilde Q_1\} \,, 
\end{align*}
and $  \mathcal Q ^{3} _{L,t} \coloneqq  \mathcal Q _{L,t} - ( \mathcal Q ^{1} _{L,t}\cup  \mathcal Q ^{2} _{L,t}  ) $ is the complementary collection. 
Notice that $  \mathcal Q ^{1} _{L,t} $ equals the whole collection $  \mathcal Q  _{L,t} $ for $ t>r+1$.

\smallskip 

We treat them in turn. The collections $  \mathcal Q ^{1} _{L,t} $  fit the hypotheses of Lemma~\ref{l:holes},  just take the collection of intervals $ \mathcal S$ of that Lemma to be  $ \mathcal S_{L}$.
It follows that  $\mathbf B _{ \mathcal Q ^{1} _{L,t} } \lesssim  \boldsymbol \beta (t)$, where the latter is the best constant in the inequality 
\begin{equation} \label{e:CS}   
\sum_{J\in (\mathcal Q_ {L,t})_2 \::\: J\Subset K} 
\mathsf P(\sigma (F- K), J) ^2 \bigl\langle \frac {x} {\lvert  J\rvert } , h ^{w} _{J} \bigr\rangle_{w} ^2 
\le \boldsymbol \beta (t) ^2   \sigma (K)\,, \qquad  K \in \mathcal S _L\,,\  L\in \mathcal L\,,\ t \ge 2 \,.  
\end{equation}
We will prove the estimate below, which is clearly summable in $ t \in \mathbb N $ to the estimate we want. 

\begin{lemma}\label{l:beta} There holds $ \boldsymbol \beta (t) \lesssim \rho ^{-t/2} \tau $. 

\end{lemma}

\begin{proof}

We have the estimate without decay in $ t$, $ \boldsymbol \beta (t) \lesssim \textup{size} (\mathcal Q)$, as follows from \eqref{e:eta<size}. 
Use this estimate for $ 1\le t \le r+3$, say.  
In the case of $ t> r+3$, the  essential property  is \eqref{e:ddecay}.
The left hand side of \eqref{e:CS} is dominated by the sum below.  
Note that we index the sum first over $ L'$, which are  $ r+1$-fold  $ \mathcal L$-children of $ K$, whence $ L'\Subset K$,  followed by $ t-r-2$-fold $ \mathcal L$-children of 
$ L'$.  
\begin{align}
\sum_{\substack{L'\in \mathcal L \\   \pi _{\mathcal L} ^{r+1} L'=K }}  &
\sum_{\substack{L''\in \mathcal L \\   \pi ^{t-r-2} _{\mathcal L} L''= L'}}  
\sum_{J \in \mathcal Q_2 \::\: J\subset L''}
\mathsf P(\sigma (F-K), J) ^2 \bigl \langle\frac x {\lvert  J\rvert }  , h ^{w} _{J} \bigr\rangle _{w} ^2 
\\
&\stackrel {{\eqref{e:PP}}} \lesssim 
\sum_{\substack{L'\in \mathcal L \\   \pi _{\mathcal L} ^{r+1} L'=K }}  
\frac {\mathsf P(\sigma (F-K), L') ^2} {\lvert  L'\rvert ^2  }
\sum_{\substack{L''\in \mathcal L \\   \pi ^{t-r-2} _{\mathcal L} L''= L'}}   
\mu (T _{L''})
 \\  \label{e:inOut}
 &\stackrel {\eqref{e:ddecay}} \lesssim  \rho ^{-t+r+2} 
\sum_{\substack{L'\in \mathcal L \\   \pi _{\mathcal L} ^{r+1} L'=K }}  
\frac {\mathsf P(\sigma (F-K), L') ^2} {\lvert  L'\rvert ^2  } 
\mu (T_{L'})
\\
&\stackrel {\phantom{\eqref{e:ddecay}}} \lesssim \rho ^{-t} \tau ^2 
\sum_{\substack{L'\in \mathcal L \\   \pi _{\mathcal L} ^{r+1} L'=K }}   \sigma (L') 
\lesssim \tau ^2 \rho ^{-t}  \sigma (K) \,. 
 \end{align}
We have also used \eqref{e:PP}, and then   the central property 
\eqref{e:ddecay} following from the construction of $ \mathcal L$,  finally appealing  to  the definition of size.   
Hence,  $  \boldsymbol \beta (t) \lesssim \tau \rho ^{-t/2}$. 
This completes the analysis of $  \mathcal Q ^{1} _{L,t} $. 

\end{proof}

We need only consider the collections $  \mathcal Q ^{2} _{L,t} $ for $ 1\le t \le r+1$, and they fall under the scope of Lemma~\ref{l:Holes}. 
A variant of \eqref{e:eta<size} shows that  $ \mathbf B _{  \mathcal Q ^{2} _{L,t}} \lesssim \tau $. 
Similarly, we  need only consider the collections $  \mathcal Q ^{3} _{L,t} $ for $ 1\le t \le r+1$.   
It follows that we must have $ 2 ^{r} \le \lvert  Q_1\rvert/\lvert  Q_2\rvert \le 2 ^{2r+2}  $. Namely, this ratio can take only one of a 
finite number of values, implying that Lemma~\ref{l:equal} applies easily to this case to complete the proof.  
\end{proof}

\subsubsection*{Pairs not strictly comparable to $ \mathcal L$}
It remains to consider the pairs $ Q\in \mathcal Q$ such that $ \tilde Q_1$ does not have a parent in $ \mathcal L$. 
The collection $  \mathcal Q ^{\textup{small}}_2 $  is taken to be the (much smaller) collection 
\begin{equation*}
\mathcal Q  ^{\textup{small}}_2  \coloneqq  \{Q \in \mathcal Q \::\:   \textup{$ {Q_2}$ does not have a parent in $ \mathcal L$}\}\,.  
\end{equation*}
Observe that 
$
\textup{size} (\mathcal Q  ^{\textup{small}}_2 ) \le \sqrt { (\rho -1)}  \tau \le \frac \tau 4 
$.
This is as required for this collection. (The collections $ \mathcal Q  ^{\textup{small}}_1$ and $ \mathcal Q  ^{\textup{small}}_2$ are also mutually orthogonal, but this fact is not needed for our proof.)

\begin{proof}
Suppose  $ \eta <\textup{size} (\mathcal Q ^{\textup{small}}_2 ) $. Then, there is an interval $ K \in \widetilde {(\mathcal Q ^{\textup{small}}_1 )}_1 \cup (\mathcal Q  ^{\textup{small}}_2)_2 $ so that 
\begin{align*}
\eta ^2   \sigma (K) \le 
\frac {P (\sigma (F-K), K) ^2 } {\lvert  K\rvert ^2  }
\mu _{\mathcal Q ^{\textup{small}}_2} (T_K)\,. 
\end{align*}
Suppose that $ K$ does not contain any interval in $ \mathcal L$. 
It follows from the initial intervals added to $ \mathcal L$, see \eqref{e:Kdef}, that we must have $ \eta \le \frac \tau 4$.

Thus, $ K$ contains an interval in $ \mathcal L$. This means that $ K$ must fail the inequality \eqref{e:Lconstruct}.  
Therefore,   we have 
\begin{align*}
\eta ^2 \sigma (K) & \le (\rho -1)
\frac {P (\sigma (F-K), K) ^2 } {\lvert  K\rvert ^2  }
\mu (T_K) \le \frac {\tau ^2 } {16} \sigma (K) \,. 
\end{align*}
This relies upon the definition of size, and proves our claim.
\end{proof}

For the pairs not yet in one of  our collections, it must be that $ Q_2$ has a parent in $ \mathcal L$, but not $ \tilde Q_1$. 
Using $ \mathcal L ^{\ast} $, the maximal intervals in $ \mathcal L$, divide them into the three collections 
\begin{align}
 \mathcal Q ^{\textup{large}} _{3} 
& \coloneqq  \{ Q\in \mathcal Q \::\:  
Q_2 \Subset  \pi _{\mathcal L ^{\ast} } Q_2 \subset \tilde Q_1 
 \} \,,
 \\
 \mathcal Q ^{\textup{large}} _{4} 
& \coloneqq  \{ Q\in \mathcal Q \::\:  
Q_2 \not\Subset  \pi _{\mathcal L ^{\ast} } Q_2 \Subset \tilde Q_1 
 \} \,, 
 \\
 \mathcal Q ^{\textup{large}} _{5} & 
 \coloneqq  \{ Q\in \mathcal Q \::\:  
Q_2 \not\Subset  \pi _{\mathcal L ^{\ast} } Q_2 \subsetneq \tilde Q_1 
 \,, \textup{and}\  \pi _{\mathcal L ^{\ast} } Q_2 \not\Subset  \tilde Q_1  \} \,. 
\end{align}

Observe that Lemma~\ref{l:holes}, with \eqref{e:eta<size}, gives  
\begin{equation} \label{e:Q3}
\mathbf B _{ \mathcal Q ^{\textup{large}}_3} \lesssim \tau \,. 
\end{equation}
Take the collection $ \mathcal S$ of Lemma~\ref{l:holes} to be $ \mathcal L ^{\ast} $. 

Observe that Lemma~\ref{l:Holes} applies to show that the  estimate \eqref{e:Q3} holds for $  \mathcal Q ^{\textup{large}}_4$. 
Take $ \mathcal S$ of that Lemma to be $ \mathcal L ^{\ast} $.  The estimate from Lemma~\ref{l:Holes} is given in terms of $ \eta $, as defined in \eqref{e:Holes}.  But,  is at most $ \tau $.  

In the last collection, $  \mathcal Q ^{\textup{large}}_5$, notice that the conditions placed upon the pair implies that 
$ \lvert  Q_1\rvert\le 2 ^{2r+2} \lvert  Q_2\rvert  $, for all $ Q\in  \mathcal Q ^{\textup{large}}_5$.  It therefore follows from a straight forward application of Lemma~\ref{l:equal}, that \eqref{e:Q3} holds for this collection as well.

\subsection{Upper Bounds on the  Stopping Form}\label{s:UB}
We   prove upper bounds on the norm of the stopping form in a situation in which 
there is some decoupling between the martingale difference on $ g$, and the argument of the Hilbert transform.
First, an elementary observation. 
 
\begin{proposition}\label{p:}
For intervals $ J\subset L\Subset K$,  with $ L$ either good, or the child of a good interval, 
 \begin{equation} \label{e:PP}
\frac {P(\sigma (F-K), J) } {\lvert  J\rvert } \simeq   \frac {P(\sigma (F-K), L) } {\lvert  L\rvert } \,. 
\end{equation}
\end{proposition}

\begin{proof}
The property of interval $ I$ being good,  says that if $ I \subset \tilde I$, and $ 2 ^{r-1} \lvert  I\rvert \le \lvert  \tilde I\rvert  $, 
then the distance of either child of $ I$ to the boundary of $ \tilde I$ is at least $ \lvert  I\rvert ^{\epsilon } \lvert  \tilde I\rvert ^{1- \epsilon } $.   
Thus, in the case that $ L$ is the child of a good interval,  the parent $ \hat L  $ of $ L$ is contained in $ K$, and $ 2 ^{r-1} \lvert  \hat  L\rvert\le \lvert  K\rvert  $, so by the definition of goodness, 
 \begin{align*}
\textup{dist} (J, F - K)  & \ge \textup{dist} (L, F - K)  
\\
& \ge   \lvert    L\rvert ^{\epsilon } \lvert  K\rvert ^{1- \epsilon } 
\ge 2 ^{ r(1- \epsilon) } \lvert  L\rvert\,.   
\end{align*}
The same inequality holds if $ L $ is good.  
Then, one has the equivalence above, by inspection of the Poisson integrals.   
\end{proof}

\begin{lemma}\label{l:holes}  Let $ \mathcal S$ be a collection of pairwise disjoint intervals in $ F$. 
Let $ \mathcal Q$ be admissible such that for each $ Q\in \mathcal Q$, there is an $ S\in \mathcal S $ 
with $ Q_2 \Subset S \subset \tilde Q_1$.  
 Then, there holds 
\begin{gather}
\lvert  B _{\mathcal Q} (f,g)\rvert \lesssim  \eta  \lVert f\rVert_{\sigma } \lVert g\rVert_{w} \,, 
\\
\label{e:S<} 
\textup{where} \quad 
\eta ^2 \coloneqq  
\sup _{S\in \mathcal S} \frac {1} {\sigma (S)    }
\sum_{J \in \mathcal Q_2 \::\: J\Subset S} 
P(\sigma (F- S), J) ^2 \bigl\langle \frac x {\lvert  J\rvert } , h ^{w} _{J} \rangle_{w} ^2 \,. 
\end{gather}
\end{lemma}
 
It is useful to note that  $ \eta $ is always smaller than the size: 
For $ S\in \mathcal S$,   let $ \mathcal J ^{\ast} $ be the maximal intervals $ J\in \mathcal Q_2$ with $ J\Subset S$, 
and note that \eqref{e:PP} applies to see that 
\begin{align}
\sum_{J \in \mathcal Q_2 \::\: J\Subset S} 
P(\sigma (F- S), J) ^2 \bigl\langle \frac {x} {\lvert  J\rvert } , h ^{w} _{J} \bigr\rangle_{w} ^2 
& = 
\sum_{ J ^{\ast}  \in \mathcal J ^{\ast} } 
\sum_{J \in \mathcal Q \::\: J\subset J ^{\ast} } 
P(\sigma (F- S), J) ^2 \bigl\langle \frac {x} {\lvert  J\rvert } , h ^{w} _{J} \bigr\rangle_{w} ^2 
\\
& \lesssim 
\sum_{ J ^{\ast}  \in \mathcal J ^{\ast} }  \frac {P(\sigma (F- S), J ^{\ast} ) ^2 } {\lvert  J ^{\ast} \rvert ^2  }
\sum_{J \in \mathcal Q \::\: J\subset J ^{\ast} }  \langle x, h ^{w} _{J} \rangle_{w} ^2 
\\  \label{e:eta<size}   
& \lesssim   \sum_{ J ^{\ast}  \in \mathcal J ^{\ast} } \sigma (J ^{\ast} ) \lesssim 
 \textup{size} (\mathcal Q) \sigma (S).  
\end{align}

\begin{proof}
An  interesting part of the proof is that it depends very much on cancellative properties of the martingale differences of $ f$. 
(Absolute values must be taken \emph{outside} the sum defining the stopping form!)  
This argument will invoke the stopping data, and part of the Hilbert-Poisson exchange argument.

Assume, as we can,  that the Haar support of $ f$ is contained in $ \mathcal Q_1$.  
Take $ \mathcal F$ and $ \alpha _{f} ( \cdot )$ to be stopping data 
defined in this way:  First, add to $ \mathcal F$ the interval $ F$, and set 
$ \alpha _{f} (F)\coloneqq  \mathbb E ^{\sigma } _{F} \lvert  f\rvert $. 
Inductively, if $ F\in \mathcal F$ is minimal, add to $ \mathcal F$ the maximal children $ F'$ such that  $ \alpha _{f} (F') \coloneqq    \mathbb E _{F'}  ^{\sigma } \lvert f \rvert > 4 \alpha _{f} (F)$.   This is a simple form of the stopping data construction in \S\ref{s:global}. In particular quasi-orthogonality \eqref{e:quasi} holds.

Write the bilinear form as 
\begin{align} 
B _{\mathcal Q} (f,g) 
&= \sum_{ J  }\langle  H _{\sigma } \varphi _{J} ,  \Delta _{J} ^{w} g \rangle _{w} 
\\
\label{e:zvf} \textup{where} \quad 
\varphi _{J } &  \coloneqq  \sum_{ \substack{Q  \in \mathcal Q\::\: Q_2=J }}  
 \mathbb E ^{\sigma } _{J} \Delta ^{\sigma }_ {Q_1} f \cdot (F - \tilde Q_1)  \,. 
\end{align}
The function $ \varphi _J$ is well-behaved, as we now explain.  
At each point  $ x$ with  $  \varphi _{J} (x) \neq 0$, the sum above is over pairs $ Q$ such that 
 $ Q_2=J$ and $ x\in F-\tilde Q_1 $. 
 By the convexity property of admissible collections,   the sum is over consecutive (good)  martingale  differences of $ f$. 
The basic telescoping property of these differences shows that the sum is bounded by the stopping value $ \alpha _{f} (\pi_{\mathcal F}  J)$. 
Let $ I ^{\ast }$ be the maximal interval of the form $ \tilde Q_1$ with $ x\in F- \tilde Q_1$, and let $ I _{\ast} $
be the child of the minimal such interval which contains $ J$.  
Then, 
\begin{align} \label{e:Zvf}
\begin{split}
\lvert  \varphi _  {J} (x) \rvert& = \Bigl\lvert  \sum_{\substack{Q \in \mathcal Q \::\:   Q_2=J \\ x \in I - \tilde Q_1} } 
\mathbb E ^{\sigma } _{J} \Delta ^{\sigma } _{Q_1} f  (x)
\Bigr\rvert
\\&= \bigl\lvert \mathbb E ^{\sigma } _{I^\ast } f  -  \mathbb E ^{\sigma } _{I_ \ast } f  \bigr\rvert 
\lesssim  \alpha _{f} (\pi_{\mathcal F}  J) (F -S)\,,  
\end{split}
\end{align}
 where $ S$ is the $ \mathcal S$-parent of $ J$. 

We can estimate as below, for $ F\in \mathcal F$:  
\begin{align} 
\Xi (F)&  \coloneqq  \Bigl\lvert   \sum_{Q\in \mathcal Q \::\: \pi _{\mathcal F} Q_2 =F} 
\mathbb E _{Q_2} \Delta ^{\sigma } _{Q_1 }f \cdot \langle  H _{\sigma } ( F - \tilde Q_1) ,  \Delta _{J} ^{w} g \rangle _{w} 
\Bigr\rvert
\\  & \stackrel{\eqref{e:zvf}}= 
 \Bigl\lvert\sum_{J \in \mathcal Q_2 \::\:  \pi _{\mathcal F}J =F }    
 \langle H _{\sigma } \varphi _{J},  \Delta ^{w} _{J} g\rangle _{w} 
\Bigr\rvert
\\&  
 \stackrel {\eqref{e:Zvf}}\lesssim  
\alpha _{f} (F) 
\sum_{\substack{S\in \mathcal S \\ \pi _{\mathcal F}S=F}} 
 \sum_{\substack{J\in \mathcal Q_2  \\  J \subset S }}  
 P(\sigma (F - S), J ) \bigl\lvert \bigl\langle  \frac x {\lvert   J\rvert }
 ,  \Delta ^{w} _{J} g \bigr\rangle_w\bigr\rvert 
 \\
 &
 \stackrel { \phantom{\eqref{e:Zvf}}}\lesssim  \alpha _{f} (F)  \Bigl[
 \sum_{\substack{S\in \mathcal S \\ \pi _{\mathcal F}S=F}} 
 \sum_{\substack{J\in \mathcal Q_2  \\  J \subset S }}  
 P(\sigma (F - S), J ) ^2  \bigl\langle \frac x {\lvert  J\rvert }, h ^{w} _{J} \bigr\rangle_{w} ^2  
  \times  \sum_{\substack{ J\in \mathcal Q_2\\  \pi _{\mathcal F}J =F }}   \hat g (J) ^2  
 \Bigr] ^{1/2} 
 \\
 & \stackrel {\eqref{e:S<}}\lesssim \textup{size}(\mathcal Q)  \alpha _{f} (F)  
  \Bigl[
   \sum_{\substack{S\in \mathcal S \\ \pi _{\mathcal F}S=F}}  \sigma (S) \times 
\sum_{\substack{J\in \mathcal Q_2 \\  \pi _{\mathcal F}J =F }}   \hat g (J) ^2  
 \Bigr] ^{1/2} 
 \\
 &  \stackrel {\phantom{\eqref{e:S<}}}\lesssim  \textup{size}(\mathcal Q)  \alpha _{f} (F)  \sigma (F) ^{1/2} 
   \Bigl[
\sum_{\substack{J\in \mathcal Q_2 \::\:   \pi _{\mathcal F}J =F }}   \hat g (J) ^2  
 \Bigr] ^{1/2}  \,. 
\end{align}
The top line follows from \eqref{e:zvf}. 
In the second, we appeal to \eqref{e:Zvf} and  monotonicity  principle, the latter being available to us since $ J \subset S$ implies $ J\Subset S$, by hypothesis.   
We also take advantage of the strong assumptions on the intervals in $ \mathcal Q_2$: If $ J\in \mathcal Q_2$, we must have $ \pi _{\mathcal F}J = \pi _{\mathcal F} (\pi _{\mathcal S} J)$.  
The third line is Cauchy--Schwarz, followed by the appeal to the hypothesis \eqref{e:S<}, while the last line uses the fact that the intervals in $\mathcal S $ are pairwise disjoint.  

The  quasi-orthogonality argument \eqref{e:quasi} completes the proof, namely we have 
\begin{equation} \label{e:Xi}
\sum_{F\in \mathcal F} \Xi (F) \lesssim \textup{size}(\mathcal Q) \lVert f\rVert_{\sigma } \lVert g\rVert_{w} \,. 
\end{equation}
\end{proof}

\begin{lemma}\label{l:Holes}  Let $ \mathcal S$ be a collection of pairwise disjoint intervals in $ F$. 
Let $ \mathcal Q$ be admissible such that for each $ Q\in \mathcal Q$, there is an $ S\in \mathcal S $ 
with $ Q_2 \subset S \Subset \tilde Q_1$.  
 Then, there holds 
\begin{gather}
\lvert  B _{\mathcal Q} (f,g)\rvert \lesssim \eta  \lVert f\rVert_{\sigma } \lVert g\rVert_{w} \,, 
\\ \label{e:Holes}
\textup{where} \quad \eta ^2  \coloneqq  \sup _{S\in \mathcal S} 
 \frac {\mathsf P(\sigma  (Q_1 -\pi _{ \tilde {\mathcal Q}_1} S), S) ^2 } {\sigma (S) \lvert  S\rvert ^2 } 
\sum_{J \in \mathcal Q _2 \::\: J\subset S}  \langle  x , h ^{w} _{J} \rangle_{w} ^2  \,. 
\end{gather}
\end{lemma}

\begin{proof}
Construct stopping data $ \mathcal F$ and $ \alpha _{f} ( \cdot )$ as in the  proof of Lemma~\ref{l:holes}.   
The fundamental inequality \eqref{e:Zvf} is again used. 
Then,  by  the monotonicity principle \eqref{e:mono1}, there holds for $ F\in \mathcal F$, 
\begin{align*}
\Xi (F) \coloneqq  &
\Bigl\lvert 
\sum_{Q\in \mathcal Q \::\: \pi _{\mathcal F}Q_2=F} \mathbb E _{Q_2} \Delta ^{\sigma } _{Q_1} f \cdot 
\langle  H _{\sigma } (F- \tilde Q_1),  \Delta ^{w} _{Q_2} g\rangle_w 
\Bigr\rvert
\\
& \lesssim  \alpha _{f} (F)
\sum_{S\in \mathcal S \::\: \pi _{\mathcal F}S=F} 
\mathsf P(\sigma  (F-\pi _{ \tilde {\mathcal Q}_1} S),S ) \sum_{J\in \mathcal Q_2 \::\:   J \subset S}  
 \bigl\langle \frac x {\lvert  S\rvert }, h ^{w} _{J} \bigr\rangle_{w} \cdot \lvert   \hat g (J)\rvert  
\\
& \lesssim 
\alpha _{f} (F) 
\Bigl[
\sum_{S\in \mathcal S \::\: \pi _{\mathcal F}S=F} 
\mathsf P(\sigma  (F-\pi _{ \tilde {\mathcal Q}_1} S),S ) ^2 
 \sum_{J\in \mathcal Q_2 \::\:   J \subset S}  
 \bigl\langle \frac x {\lvert  S\rvert }, h ^{w} _{J} \bigr\rangle_{w} ^2 
\times 
 \sum_{J\in \mathcal Q_2 \::\:  J \subset S} \hat g (J)  ^2 
\Bigr] ^{1/2} 
\\
& \lesssim \eta 
\alpha _{f} (F) 
\Bigl[
\sum_{S\in \mathcal S \::\: \pi _{\mathcal F}S=F}  \sigma (S)
\times 
 \sum_{J\in \mathcal Q_2 \::\:   J \subset S } \hat g (J)  ^2 
\Bigr] ^{1/2} 
\\
& \lesssim 
\eta  \alpha _{f} (F) \sigma (F) ^{1/2} 
\Bigl[
\sum_{\substack{J\in \mathcal Q_2  \::\:  \pi _{\mathcal F}J =F }}   \hat g (J) ^2  
\Bigr] ^{1/2}  \,. 
\end{align*}
After the monotonicity principle \eqref{e:mono1}, we have used Cauchy-Schwarz, and the definition of $ \eta $.  
The quasi-orthogonality argument \eqref{e:quasi} then completes the analysis of this term, see \eqref{e:Xi}.
\end{proof}

The last Lemma that we need  is elementary, and is contained in the methods of \cite{10031596}. 

\begin{lemma}\label{l:equal} Let $ u\ge r+1$ be an integer, and  $ \mathcal Q$ be an admissible collection of pairs such that $\lvert  Q_1\rvert= 2 ^{u} \lvert  Q_2\rvert  $ for all $ Q\in \mathcal Q$.  There holds 
\begin{equation*}
\lvert  B _{\mathcal Q} (f,g) \rvert \lesssim \textup{size} (\mathcal Q)  \lVert f\rVert_{\sigma } \lVert g\rVert_{w}  \,. 
\end{equation*}
\end{lemma}

\begin{proof}
  Recall the form of the stopping form in \eqref{e:stop}.  
Observe, from inspection of the definition of the Haar function \eqref{e:hs1}, that 
\begin{equation*}
\lvert  \mathbb E ^{\sigma } _{I_J} \Delta ^{\sigma }_I f\rvert \le \frac {\lvert  \hat f (I) \rvert  } {\sigma (I_J) ^{1/2} } \,. 
\end{equation*}
Then, an elementary application of the monotonicity principle gives us 
\begin{align*}
\lvert  B _{\mathcal Q} (f,g)\rvert 
& \le \sum_{I \in \mathcal Q_1}  {\lvert  \hat f (I) \rvert  } 
\sum_{\substack{J \::\:  (I,J)\in \mathcal Q   }} {\sigma (I_J) ^{-1/2} }
  \mathsf P(\sigma (F-I_J), J )  \bigl\langle \frac x {\lvert  J\rvert }, h ^{w} _{J} \bigr\rangle_{w} 
\lvert  \hat g (J) \rvert 
\\
& \le 
\lVert f\rVert_{\sigma } 
\biggl[
 \sum_{I \in \mathcal Q_1} 
 \biggl[ 
\sum_{\substack{J \::\:  (I,J)\in \mathcal Q   }}  \frac {1} {\sigma (I_J)  } 
 \mathsf P(\sigma (F-I_J), J )  \bigl\langle \frac x {\lvert  J\rvert }, h ^{w} _{J} \bigr\rangle_{w} 
\lvert  \hat g (J) \rvert 
 \biggr] ^2 
\biggr] ^{1/2} 
\\
& \le \textup{size} (\mathcal Q) \lVert f\rVert_{\sigma } \lVert g\rVert_{w} 
\end{align*}
This follows immediately from Cauchy-Schwarz, and the fact that for each $ J\in \mathcal Q _2$, there is a unique $ I\in \mathcal Q_1$ 
such that the pair $ (I,J)$ contribute  to the sum above. 
\end{proof}

\subsection{Context and Discussion}

\subsubsection{}
The proof herein succeeds because the notion of size approximates the operator norm of the stopping form. 
Moreover,  the `large' portions of the stopping form, there is a decoupling that takes place.

\subsubsection{}
It is very interesting that one can prove unconditional results about the two weight Hilbert transform, following the techniques in \cite{MR3285857}, 
without solving the local problem.

\section{Elementary Estimates} \label{s:elem}

This section is devoted to the proof of Lemma~\ref{l:above}.   
The estimates fall into many subcases, and are of a more classical nature, albeit the $ A_2$ assumption is critical. (In fact, all the estimates in this section 
	depend only on the  half-Poisson $ A_2$ hypothesis, but this is not systematically tracked in the notation.) 
In addition, all estimates should be interpreted as uniform over all smooth truncations. Some of these 
are off-diagonal estimates, for which the smooth truncations are important.  The uniformity over truncations is however 
suppressed in notation. 

First  some basic estimates are collected. 
This is  property of good intervals, which can be effectively used in non-critical situations. 

\begin{lemma}\label{l.donotuse} For three intervals $ J, I, I'\in \mathcal D$ with 
$ J \subset I \subset I'$, $ \lvert  J\rvert = 2 ^{-s} \lvert  I\rvert  $,  with $ s\ge r$ and $ J$ good, then 
\begin{equation}  \label{e:donotuse}
P (\sigma \cdot (I'-I) , J ) \le 2 ^{- (1- \varepsilon ) s} P (\sigma \cdot I ', I )\,.
\end{equation}
\end{lemma}

\begin{proof}
Note that for $ x\in I'-I$ we have 
\begin{equation*}
\textup{dist} (x,J) \ge   \lvert  I\rvert ^{1- \varepsilon } \lvert  J\rvert ^{\varepsilon } 
= 2 ^{s (1- \varepsilon )} \lvert  J\rvert\,.  
\end{equation*}
Using this in the definition of the Poisson integral, we get 
\begin{align*}
P (\sigma \cdot (I'-I) , J ) 
& \le 2\int _{I'-I} \frac {\lvert  J\rvert } { \textup{dist} (x,J)^2 }  \; \sigma (dx) 
\\
& \lesssim \frac {\lvert  J\rvert } {\lvert  I\rvert } 
 \int _{I'-I} \frac {\lvert  I\rvert } { (\lvert  J\rvert  + \textup{dist} (x,J) ) ^2 } \; \sigma (dx) 
 \\
 & \lesssim 2 ^{- s (1- 2 \varepsilon )} 
 \int _{I'-I} \frac {\lvert  I\rvert } { (\lvert  I\rvert  + \textup{dist} (x,I) ) ^2 } \; \sigma (dx) 
 = 2 ^{- s (1- 2 \varepsilon )} P (\sigma (I'-I),I) \,. 
\end{align*}
\end{proof}

\begin{proposition}\label{p:Eip} Suppose that two intervals $ I,J \in \mathcal D$ satisfy $ \lvert  I\rvert \ge \lvert  J\rvert  $, and $ 3I\cap J = \emptyset $, then 
\begin{equation}\label{e:Eip}
\sup _{0 < \alpha < \beta }\lvert  \langle H ( \sigma  I), h ^{w} _{J} \rangle _{w}\rvert \lesssim \sigma (I) \sqrt {w (J)} \frac { \lvert  J\rvert } { (\lvert  J\rvert +\textup{dist} (I,J) ) ^2 }
\end{equation}
\end{proposition}

\begin{proof}
Since $ h ^{w} _{J}$ has $ w$-integral zero, estimate as below, where $ x_J$ is the center of $ J$. 
\begin{align*}
\lvert  \langle H I, h ^{w} _{J} \rangle _{w} \rvert &=
\Bigl\lvert \int _{I} \int _{J}  K _{\alpha , \beta}(y-x) \cdot h ^{w} _{J} (x) \; w (dx) \sigma (dy)  \Bigr\rvert
\\
&= 
\Bigl\lvert \int _{I} \int _{J} \bigl \{ K _{\alpha , \beta}(y-x ) - K _{\alpha , \beta}(y-x_J)\bigr\}  h  ^{w} _{J} (x) \; w (dx) \sigma (dy)  \Bigr\rvert
\\
&\lesssim  \int _{I} \int _{J}  \frac { \lvert  J\rvert } { (\lvert  J\rvert +\textup{dist} (I,J) ) ^2 } \lvert   h  ^{w} _{J} (x)\rvert  \; w (dx) \sigma (dy) . 
\end{align*}
The Lemma follows by inspection. 

\end{proof}

\begin{proposition}\label{p:EEip} 
Suppose that two intervals $ I,J \in \mathcal D$ satisfy $ 2 ^{s}\lvert J  \rvert = \lvert  I\rvert  $,  where $ s>r$,  the interval $ J$ is good, and 
$J\subset 3I \setminus I $, then  
\begin{equation}\label{e:EEip}
\sup _{0 < \alpha < \beta }\lvert  \langle H ( \sigma  I), h ^{w} _{J} \rangle _{w}\rvert \lesssim 2 ^{- (1- 2\varepsilon )s }  \sigma (I) \sqrt {w (J)} 
\lvert  I\rvert ^{-1}  
\end{equation}
\end{proposition}

\begin{proof}
Under the assumption of the Lemma, the proof of Proposition~\ref{p:Eip} holds, supplying the estimate 
estimate of that Lemma. But, the extra assumption that $ J$ is good implies that 
$ \textup{dist} (J,I)> 2 ^{s (1- \varepsilon )} \lvert  J\rvert $, and then the estimate follows by inspection. 

\end{proof}

\subsection{The Weak Boundedness Inequality} 

The following inequality is a weak-boundedness inequality, a consequence of the $ A_2$ inequality.  
Here, we look at the Hilbert transform inequality on two disjoint intervals. 

\begin{proposition} \label{p:weakB}
There holds for all disjoint intervals  $ I, J$ with no point masses at their endpoints,    
\begin{equation} \label{e:weakB}
\sup _{0 < \alpha < \beta }\lvert  \langle H ( \sigma  f \cdot I),  g \cdot J \rangle _{w}\rvert \lesssim \mathscr A_2 ^{1/2}   \lVert f\rVert_{\sigma } \lVert g\rVert_{w}\,.  
\end{equation}
The constant on the right can in fact be taken as follows. For a point $ a$ that separates the interiors of $ I$ and $ J$, 
with $ I$ to the left of $ a$, 
\begin{equation} \label{e:4compact}
\sup _{r >0}   P (\sigma \mathbf 1_{(- \infty ,a)} , (a, a+r)) \frac {w (a, a+r)} {r} +P (w  \mathbf 1_{(a, \infty )}, (a, a+r)) \frac {\sigma  (a-r, a)} {r} .  
\end{equation}
In particular, for arbitrary intervals $ I$ and $ J$ with no point masses at the endpoints, 
\begin{equation} \label{e:weakIJ}
\lvert  \langle  H _{\sigma } I, J \rangle _{w}\rvert \lesssim \mathscr A_2 ^{1/2}  [ \sigma (I)  w (J)] ^{1/2} 
\end{equation}

\end{proposition}

It is useful to note that the \emph{global} integrability of indicators is then a consequence of the 
$ A_2$ and interval testing conditions.

Since the intervals are disjoint, there is no possibility of cancellation in the estimate, and it therefore is closely relate to the Hardy inequality. 
In the two weight setting, this has been characterized by  Muckenhoupt \cite{MR0311856}. 

\begin{priorResults}\label{p:hardy} 
 For weights $ \widehat{w }$ and $\sigma $  supported on $ \mathbb R _+$.  
\begin{gather}
\Bigl\lVert  \int_{(0,x)}f \;\sigma(dy)  \Bigr\rVert_{\hat  w}
\leq \mathscr B \lVert f \rVert_{\sigma } \label{e:Muck}\,, 
\\ 
\textup{where} \quad 
\mathscr B^{2}\simeq \sup_{0< r < \infty }  \int_{(r, \infty)  }\widehat{w } (dx) \times 
\int_{(0,r)}\sigma(dy)   \,. \label{e:Bis}
\end{gather}
\end{priorResults}

For the sake of completeness, we recall Muckenhoupt's proof of this result.  This preparation is proved by integration by parts. 

\begin{proposition}\label{p:BV} Let $ \phi $ be an increasing  function on $ (0, \infty )$, with $ \phi (0)=0$ and $ \phi $ strictly positive on $ (0, \infty )$.   Then,  
\begin{equation}\label{e:BV}
\int _{(0,x]} \phi (t) ^{-1/2} d \phi (t) \leq  2 \phi (x) ^{1/2}, 
\end{equation}
with equality if $ \phi $ is continuous.  
\end{proposition}

\begin{proof}[Proof of Theorem \ref{p:hardy}] 
We are free to assume that the function $ \phi (x) = \sigma ((0,x))$ is strictly positive on $ (0, \infty )$.  
Then, multiply and divide by $ \phi (x) ^{1/4}$, and use Cauchy--Schwarz to see that 
\begin{align*}
\Bigl\lVert  \int_{(0,x)}f \;\sigma(dy)  \Bigr\rVert_{\hat  w} ^2 
& \le \int _{(0,\infty) } \int _{(0,x)} f (y) \phi (y) ^{1/2} \; \sigma (dy)  \cdot 
 \int _{(0,x)} \phi (y) ^{-1/2} \; \sigma (dy)  \; \hat w (dx) 
\\& \leq 2 
 \int _{(0,\infty) } \int _{(0,x)} f (y) \phi (y) ^{1/2} \; \sigma (dy)  \cdot  \phi (x) ^{1/2}  \; \hat w (dx) 
 \\
 & = 2\int _{(0,\infty) } f (y) \phi (y) ^{1/2}  \int _{(y,\infty) } \phi (x) ^{1/2}  \; \hat w (dx) \; \sigma (dy)
\end{align*} 
Above, we have used \eqref{e:BV}, and then Fubini.  Concentrate on the inner integral. 
Our definition of $ \mathscr B$ and Proposition~\ref{p:BV} gives us 
\begin{align*}
\mathscr B  \int _{(y,\infty) }  \Biggl[ \int _{(x,\infty) } \hat w (dt)   \Biggr] ^{-1/2}  \; \hat w (dx)  
& \leq 2\mathscr B  \Biggl[ \int _{(y,\infty) } \hat w (dt)   \Biggr] ^{1/2} 
\end{align*}
And, now we can estimate 
\begin{align*}
\Bigl\lVert  \int_{(0,x)}f \;\sigma(dy)  \Bigr\rVert_{\hat  w} ^2  
& \le 4 \mathscr B \int _{(0, \infty ) } f (y) \phi (y) ^{1/2}  \Biggl[ \int _{(y,\infty) } \hat w (dt)   \Biggr] ^{1/2}  \; \sigma (dy) 
\le 4 \mathscr B ^2 \lVert f\rVert_{\sigma } ^2 . 
\end{align*}
The proof is complete. 
\end{proof}

\begin{proof}[Proof of Proposition~\ref{p:weakB}] 
Interval testing and \eqref{e:weakB} prove the estimate \eqref{e:weakIJ}, so we turn to the proof of \eqref{e:weakB}.

After a translation, we can assume that $ 0$ separates the interiors of $ I$ and $ J$. Let us assume that $ I$ is to the left of zero. 
We change the problem.
Set $ \tilde \sigma (dx) = \sigma (-dx)$ for $ x\geq 0$, and for $ f \in L ^2 (I, \sigma )$, set $ \phi (x) = f (-x)$. Then, 
\begin{align}
\langle H _{\sigma } f, g \rangle _{w} &= \int_{(- \infty ,0)}\int_{(0,\infty )}\frac {f (y) g (x)} {y-x}  \; \sigma (dy) w (dx)
\\ \label{e:Xcompact}
&= - \int_{(0,\infty )}\int_{(0,\infty )} \frac {\phi  (y) g (x)} {x+y}  \; \tilde \sigma (dy) w (dx). 
\end{align}
The double integral is split into dual terms, one of which is 
\begin{equation} \label{e:xII}
\int_{(0,\infty )}\int_{(0,x )} \frac {\phi  (y) g (x)} {x+y}   \;\tilde \sigma (dy)\, \; w (dx) . 
\end{equation}
We analyze this bilinear form.  

Note that $ x+y \simeq x$ in \eqref{e:xII}.  Thus, it suffices to estimate 
\begin{align*}
\int _{(0,\infty) }\Bigl\lvert  \int_{(0,x )} \frac {\phi  (y) } {x}   \;\tilde \sigma (dy) \Bigr\rvert ^2 \; w (dx) 
& = 
\int _{(0,\infty) }\Bigl\lvert  \int_{(0,x )} \frac {\phi  (y) } {x}   \;\tilde \sigma (dy) \Bigr\rvert ^2 \; \frac {w (dx) } {x ^2 } 
\leq  \mathscr B ^2  \lVert \phi \rVert_{\tilde \sigma } ^{2}.   
\end{align*}
where $ \mathscr B$ is as in \eqref{e:Bis}, and $ \hat w (dx) =  \frac {w (dx) } {x ^2 } $ and $ \sigma = \tilde \sigma $

It remains to estimate the constant $ \mathscr B$,   For any $0 < r <  \infty $, 
\begin{align*}
 \int_{(0,r)}  \tilde \sigma(dy)  \int_{(r, \infty) }d\widehat{w }& 
= \frac {\sigma (-r,0) } r \int_{(r,\infty )}\frac r {x ^2 } \; w (dx) 
\lesssim \mathscr{A}_{2}.
\end{align*}
The more precise conclusion \eqref{e:4compact} can be read off from this inequality.  
Recall that \eqref{e:Xcompact} is split into two bilinear forms, and we have only considered one of them. 
This explains the symmetric form of \eqref{e:4compact}.  
\end{proof}

\subsection{The Different Subcases of Lemma~\ref{l:above}}

Lemma~\ref{l:above} follows from appropriate bounds on these bilinear forms, and their duals. 
\begin{gather}\label{e:nearby}
B ^{\textup{nearby}} (f,g) 
\coloneqq  \sum_{\substack{I, J \::\: 2 ^{-r-1} \lvert  I\rvert\le \lvert  J\rvert\le  \lvert  I\rvert  \\  3 I \cap  J \neq \emptyset }} 
\lvert  \langle H _{\sigma} \Delta ^{\sigma}_I f , \Delta ^{w} _{J} \phi  \rangle _{w}  \rvert \,, 
\\ \label{e:far}
 B ^{\textup{far}} (f,g) 
\coloneqq  \sum_{\substack{I, J \::\: 3 I \cap  3J = \emptyset }} 
\lvert  \langle H _{\sigma} \Delta ^{\sigma}_I f , \Delta ^{w} _{J} \phi  \rangle _{w}  \rvert \,, 
\\
\\ \label{e:close}
 B ^{\textup{close}} (f,g) 
\coloneqq  \sum_{\substack{I, J \::\:    2 ^{r}\lvert  J\rvert\le  \lvert  I\rvert  \\   J\subset 3I \setminus I }} 
\lvert  \langle H _{\sigma} \Delta ^{\sigma}_I f , \Delta ^{w} _{J} \phi  \rangle _{w}  \rvert \,, 
\\ \label{e:adj}
B ^{\textup{adjacent}} (f,g) 
\coloneqq  \sum_{\substack{I, J \::\:   J \Subset I_J}} 
\lvert  \mathbb E ^{\sigma } _{I- I_J} \Delta ^{\sigma}_I f   \langle H _{\sigma}  (I- I_J), \Delta ^{w} _{J} \phi  \rangle _{w}  \rvert \,. 
\end{gather}

\begin{lemma}\label{l:3}  For $ \star \in \{ \textup{nearby},\ \textup{far},\  \textup{close},\ \textup{adjacent}\}$, there holds 
\begin{equation*}
B ^{\star} (f,g) \lesssim \mathscr A_2 ^{1/2} \lVert f\rVert_{\sigma } \lVert g\rVert_{w} \,. 
\end{equation*}
\end{lemma}

\subsection{The  Nearby Term}\label{s.12}
One can check directly that for each interval $ I$, with child $ I'$, there holds 
$
\lvert  \mathbb E ^{\sigma } _{I'} h ^{\sigma } _{I}\rvert \le \sigma (I') ^{-1/2 } $.
It then follows from \eqref{e:weakB} that 
$
\lvert  \langle H _{\sigma } h ^{\sigma } _I, h ^{w}_J \rangle_w \rvert \lesssim \mathscr H 
$. 
And then, 
\begin{align*}
B ^{\textup{nearby}} (f,g) 
& \lesssim  \mathscr H
\sum_{\substack{I, J \::\: 2 ^{-r-1} \lvert  I\rvert\le \lvert  J\rvert\le  \lvert  I\rvert  \\  3 I \cap  J \neq \emptyset }} 
\lvert  \hat f (I)  \hat g (J) \rvert 
 \lesssim \mathscr H \lVert f\rVert_{\sigma } \lVert g\rVert_{w} \,. 
\end{align*}
The last line follows from the fact that for each $ I$, there are only a bounded number of $ J$ occurring in the sum. 

Here, and below, we will be using the notation $ \hat f (I) = \langle f, h ^{\sigma } _{I} \rangle _{\sigma } $.

 \subsection{The Far Term} 
 We consider the case of $ \lvert  J\rvert \le \lvert  I\rvert  $, and $ 3I \cap 3J \neq \emptyset $. 
 It follows that $ J\subset 3 ^{s+1} I \setminus 3 ^{s} I$ for some integer $ s\ge 1$.  
 For an interval $ K$, integer $ s\ge r$ and $ t\ge 0$, consider the two projections 
 \begin{align*}
\Pi _{K,s,t} ^{\sigma } f &\coloneqq  
\sum _{ \substack{ I \::\: I \subset 3 ^{t+2} K \setminus 3 ^{t+1} K \\ \lvert  I\rvert = \lvert  K\rvert  }} \Delta ^{\sigma } _{I} f 
\\
\Pi _{K,s,t} ^{w} g& \coloneqq  
\sum_{\substack{J \::\: J\subset 3 ^{t}K\\ 2 ^{s}\lvert  J\rvert= \lvert  K\rvert }} \Delta ^{w} _{J} g . 
\end{align*}
These projections satisfy, for fixed $ s, t$, 
\begin{equation} \label{e:pi-2}
\sum_{K} \lVert \Pi _{K,s,t} ^{\sigma } f\rVert_{\sigma } ^2 \le \lVert f\rVert_{\sigma } ^2 ,
\end{equation}
with a similar bound for $ \Pi _{K,s,t} ^{w} g$.  
Also,  we need to bound 
 \begin{equation}  \label{e:bound}
\sum_{s\ge r} \sum_{t \ge 0} 
\bigl\lvert \langle H _{\sigma } \Pi _{K,s,t} ^{\sigma }f ,  \Pi _{K,s,t} ^{w} g \rangle _{w}\bigr\rvert . 
\end{equation}

But, using the fact that $ \Delta ^{w} _{J} g $ has mean zero, and the distance between the support of $ \Pi _{K,s,t} ^{\sigma }f $ 
and $ \Pi _{K,s,t} ^{w} g$ is approximately $ 3 ^{t} \lvert  K\rvert $, we have 
\begin{align*}
\bigl\lvert \langle H _{\sigma } \Pi _{K,s,t} ^{\sigma }f ,  \Pi _{K,s,t} ^{w} g \rangle _{w}\bigr\rvert 
&\lesssim \frac {2 ^{-s} \lvert  K\rvert } { 3 ^{2t} \lvert  K\rvert ^2  } 
\lVert \Pi _{K,s,t} ^{\sigma }f \rVert_{L ^{1} (\sigma )} \lVert \Pi _{K,s,t} ^{w} g\rVert_{L ^{1} (w)} 
\\
& \lesssim 
\frac {\sqrt { \sigma (3 ^{t+2} K) w (3 ^{t} K)}} { 2 ^{s} 3 ^{2t} \lvert  K\rvert }   
\lVert \Pi _{K,s,t} ^{\sigma }f \rVert_{\sigma } \lVert \Pi _{K,s,t} ^{w} g\rVert_{w} 
\\& \lesssim 
2 ^{-s} 3 ^{-t} \mathscr A_2 ^{1/2} 
\lVert \Pi _{K,s,t} ^{\sigma }f \rVert_{\sigma } \lVert \Pi _{K,s,t} ^{w} g\rVert_{w} . 
\end{align*}
Since we have gained geometric decay in $ s$, and $ t$, and we have the inequality \eqref{e:pi-2}, 
we can easily complete the proof of \eqref{e:bound}.

\subsection{The Close Term}
For integers $ s\ge r$, the sum below a relative length of $ J$ with respect to $ I$. Applying \eqref{e:EEip}, 
\begin{align*}
\sum_{\substack{I, J \::\:    2 ^{s}\lvert  J\rvert= \lvert  I\rvert  \\   J\subset 3I \setminus I }} 
\lvert  \langle H _{\sigma} \Delta ^{\sigma}_I f , \Delta ^{w} _{J} \phi  \rangle _{w}  \rvert 
& \lesssim 
2 ^{ (1-2 \varepsilon )s} \sum_{\substack{I, J \::\:    2 ^{s}\lvert  J\rvert= \lvert  I\rvert  \\   J\subset 3I \setminus I }}  
\lvert  \hat f (I)  \hat  g (J)\rvert \frac {\sqrt {\sigma (I) w (J)}} {\lvert  I\rvert } 
\\
& \lesssim 
2 ^{ (1-2 \varepsilon )s} 
\sum_{I} \lvert  \hat f (I) \rvert \frac {\sqrt {\sigma (I)}}  {\lvert  I\rvert } 
 \sum_{\substack{ J \::\:    2 ^{s}\lvert  J\rvert= \lvert  I\rvert  \\   J\subset 3I \setminus I }}  
\lvert  \hat g (J) \rvert  \sqrt {w (J)}
\end{align*}
We have the geometric decay in $ s$. Apply Cauchy--Schwarz, one term is $ \lVert f\rVert_{\sigma }$. The other term, squared, is 
\begin{align*}
\sum_{I}   \frac {  {\sigma (I)}}  {\lvert  I\rvert ^2 } 
 \sum_{\substack{ J \::\:    2 ^{s}\lvert  J\rvert= \lvert  I\rvert  \\   J\subset 3I \setminus I }}   \hat g (J) ^2 
 \times &
 \sum_{\substack{ J \::\:    2 ^{s}\lvert  J\rvert= \lvert  I\rvert  \\   J\subset 3I \setminus I }}  w (J) 
 \lesssim 
\mathscr A_2  \sum_{I} 
 \sum_{\substack{ J \::\:    2 ^{s}\lvert  J\rvert= \lvert  I\rvert  \\   J\subset 3I \setminus I }}   \hat g (J) ^2  
\lesssim \mathscr A_2    \lVert g\rVert_{w} ^2  \,. 
\end{align*}
This completes the estimate. 

\subsection{The Adjacent Term}
We argue as in the previous case. 
It is easy to see that $ \lvert  \mathbb E ^{\sigma } _{I- I_J} \Delta ^{\sigma}_I f  \rvert \lesssim \lvert  \hat f (I) \rvert \sigma (I-I_J) ^{-1/2}$. 

For  $ \theta \neq \theta ' \in \{\pm\}$, and consider the sum below, where $ s$ plays the same role as before. 
\begin{align*} 
\sum _{ \substack{  I, J \;:\; 2 ^{s} \lvert  J\rvert= \lvert  I\rvert  \\ J \subset I + (\theta ' \lvert  I\rvert)   }}  
\bigl\lvert \mathbb E _{I _{\theta }} ^{\sigma } \Delta ^{\sigma } _{I}   f  \cdot 
\langle  H _{\sigma }  I _{\theta }, \Delta  ^{w } _{J} g \rangle _{w }\bigr\rvert  
& \lesssim  2 ^{- (1 - 2\varepsilon )s} 
\sum _{ \substack{  I, J \;:\; 2 ^{s} \lvert  J\rvert= \lvert  I\rvert  \\ J \subset I + (\theta ' \lvert  I\rvert)   }} 
\lvert  \hat f (I) \hat g (J)  \rvert \frac {\sqrt {\sigma (I _{\theta }) w (J)}} {\lvert  I\rvert }  
\\
& \lesssim 2 ^{- (1 - 2\varepsilon )s}  \mathscr A_2 ^{1/2} \lVert f\rVert_{\sigma } \lVert g\rVert_{w}\,. 
\end{align*}
The details are suppressed. 

\subsection{Context and Discussion} 

The techniques of this section are all drawn from the work of Nazarov-Treil-Volberg \cites{V,10031596}, aside from the use of the two weight Hardy inequality, which is drawn from \cite{10014043}.

\section{Proof under the Pivotal Assumption} \label{s:pivotal}

We prove an upper bound for a two weight inequality assuming a 
pivotal condition on a pair of weights.  The setup is as follows. 
Let  $ K  (y)$ satisfy the size and gradient condition 
\begin{gather*}
\lvert  x-y\rvert  \cdot \lvert  \nabla K (x,y)\rvert + \lvert  K (x,y)\rvert \le \lvert  x-y\rvert ^{-1} \,. 
\end{gather*} 
We will consider the operator $ T f$ given formally by $ \textup{p.v.} \int K (x,y) f (y) \; dy$.  In the two weight setting, 
no principal value need exist, so given two weights $ \sigma , w$, we consider the constant  $ \mathscr N _T$, which is be the best constant in the inequality 
\begin{equation*}
\Bigl\lVert  \int   K (x,y) f (y) \; \sigma (dy) \Bigr\rVert_{w} 
\le \mathscr N_T \lVert f\rVert_{\sigma } \,. 
\end{equation*}

Let $ \mathscr P$ be the best constant in the \emph{pivotal inequality}, defined as follows. 
 For any interval $ I_0$ and any partition $ \mathcal P$ of $ I_0$ into intervals such that neither $ \sigma $ nor $ w $ have point masses at the endpoints,  there holds 
\begin{equation}\label{e:Pivotal}
\begin{split}
\sum_{I \in \mathcal P  }  
P (   \sigma (I_0 \setminus I) , I) ^2  w (I) 
\le  \mathscr P ^2 \sigma (I_0) \,. 
\end{split}
\end{equation}
We also require that the dual inequality, with the roles of $w $ and $ \sigma $ reversed, holds.  
One can note that this inequality will hold if the maximal function satisfies the two weight inequality $ \lVert M _{\sigma }f\rVert_{w} \lesssim \lVert f\rVert_{\sigma }$, and its dual.  

\begin{theorem}\label{t:Pivotal}[Nazarov-Treil-Volberg \cite{V}]  Assume that the pair of weights $ w, \sigma $  satisfy the $ A_2$ condition \eqref{e:A2}, and the pivotal conditions hold, namely $ \mathscr P < \infty $. Then, there holds 
$ \mathscr N _{T} \lesssim \mathscr T_T + \mathscr A_2 ^{1/2} + \mathscr P$, where $ \mathscr T$ is the best constant in the inequalities 
\begin{equation*}
\int _{I} \lvert  T _{\sigma } I\rvert ^2 \; w (dx) \le \mathscr T_T ^2 \sigma (I)\,, \qquad  
\int _{I} \lvert  T _{w } I\rvert ^2 \; \sigma  (dx) \le \mathscr T_T ^2 w (I)\,.
\end{equation*}
\end{theorem}

We give the proof, with the goal of  highlighting some of the difficulties that one must face in the general case. 
In addition, a quantitative higher dimensional version of this Theorem was key to \cite{ptv}. We will use Calder\'on-Zygmund stopping data, to facilitate comparisons to the general case. This will also give an easier proof than is in \cites{V,ptv}.  

\subsection{Off-Diagonal Estimates}

We need a typical off-diagonal estimate, one that is far less refined than the  monotonicity principle. 

\begin{lemma}\label{l:off} For all $ 0 < \alpha < \beta $, good intervals $ J \Subset I$, and function $ f$ is 
supported off of $ I$, there holds 
\begin{equation}\label{e:off}
\lvert  \langle  T  \sigma f ,  g \rangle\rvert 
\lesssim P (\sigma \lvert  f\rvert \cdot I ^{c}, I )  w (J) ^{1/2} \lVert g\rVert_{w}. 
\end{equation}
for any function $ g \in L ^2 (w)$, supported on $ J$ and with integral zero. 

\end{lemma}

\begin{proof}
Use the standard subtraction argument to see that 
\begin{align*}
\lvert  \langle  T  \sigma f , g (x) \rangle\rvert  
&= 
\Bigl\lvert 
\int _{J} \int _{\mathbb R \setminus I} \{  K  (x,y) -  K  (x_J,y)\} f (y) g (x) \; \sigma (dy)\, w (dx) 
\Bigr\rvert
\\
& \lesssim 
\int _{J} \int _{\mathbb R \setminus I} 
\frac { \lvert  x-x_J\rvert } { (x_J- y) ^2 } 
\cdot \lvert f (y) g (x)\rvert  \; \sigma (dy)\, w (dx) . 
\end{align*}
The bound follows by Cauchy--Schwarz and  inspection. 
\end{proof}

\subsection{The Global To Local Reduction}
 
One need only prove that 
\begin{equation*}
\lvert  \langle  T _{\sigma } P ^{\sigma  } _{\textup{good}} f, P ^{w} _{\textup{good}}g \rangle _{w}\rvert  
\lesssim \mathscr T \lVert f\rVert_{\sigma } \lVert g\rVert_{w} \,, 
\end{equation*}
where $ \mathscr T \coloneqq  \mathscr T _T+ \mathscr A_2 ^{1/2} + \mathscr P $. 
The set up is much like \S\ref{s:global}.  
We will understand that the functions  $ f$ and $ g$ can be assumed to be good functions.  
 In fact, $ f$ has the `thin' Haar expansion in \eqref{e:mod}, and similarly for $ g$, in order to reduce some case analysis below.  

In analogy to \eqref{e:ABOVE}, define 
\begin{equation}   \label{e:Tabove}
B ^{\textup{above}} (f,g) \coloneqq  \sum_{I \::\: I\subset I_0} \sum_{J \::\: J\Subset I} 
\mathbb E ^{\sigma } _{I_J} \Delta ^{\sigma }_I f \cdot   \langle T _{\sigma } I_J, \Delta ^{w} _{J} g \rangle _w \,, 
\end{equation}
and define $ B ^{\textup{below}} (f,g) $ similarly.  Since Lemma~\ref{l:above} depends only on the $ A_2$ assumption, 
we have

\begin{lemma}\label{l:Tabove}  There holds 
\begin{equation*}
\bigl\lvert  \langle T _{\sigma } f ,g \rangle _{w} - B ^{\textup{above}} (f,g)  - B ^{\textup{below}} (f,g) \bigr\rvert 
\lesssim \mathscr A_2 ^{1/2}  \lVert f\rVert_{\sigma } \lVert g\rVert_{w}  \,. 
\end{equation*}
\end{lemma}

Thus, the main technical result is 

\begin{lemma}\label{l:Ttriangular} There holds 
\begin{equation}\label{e:Ttriangular}
\lvert  B ^{\textup{above}} (f,g) \rvert  \lesssim 
\mathscr T \lVert f\rVert_{\sigma } \lVert g\rVert_{w} \,. 
\end{equation}
The same inequality holds for $ B ^{\textup{below}} (f,g) $.  
\end{lemma}

The stopping intervals are defined similarly. 

\begin{definition}\label{d:Pivotal_stopping} Define $ \mathcal F$, the stopping intervals, recursively by 
initializing $ I ^{0} \in \mathcal F$, and in the recursive step, if $ F \in \mathcal F$ is minimal, 
add to $ \mathcal F$ the maximal subintervals $F'\subset F  $, with $ F'\in \mathcal D _{f}$, so 
that meet either of these conditions: 
\begin{description}
\item[$ f$ stopping]   $ \mathbb E ^{\sigma } _{F'} \lvert  f\rvert > C \alpha _{f} (F) \coloneqq   \mathbb E ^{\sigma } _{F} \lvert  f\rvert $. 
\item[Pivotal Stopping]    $ \mathsf P( \sigma \cdot  I_0, I) ^2 w (I) > 10 {\mathscr P} ^2 \sigma (I)  $.
\end{description}
That is, we stop if either the average of $ f$ becomes too large, or, essentially, the  pivotal quantity becomes too large. 
\end{definition}

We use the same notation as in \S\ref{s:global}, and  in analogy to Corollary~\ref{t:aboveCorona}, there holds 

\begin{lemma}\label{l:aboveCorona}[The Global to Local Reduction] 
There holds 
\begin{align}
\lvert  B ^{\textup{above}} _{\mathcal F, \textup{glob}}  (f,g)\rvert  &\lesssim \mathscr T
\lVert f\rVert_{\sigma } \lVert g\rVert_{w} \,, 
\\ 
\textup{where } \quad 
\label{e:Tglob}
B ^{\textup{above}} _{\mathcal F, \textup{glob}}  (f,g)& \coloneqq 
\sum_{ \substack{I,J \::\:  \dot\pi _{\mathcal F} J \subsetneq  I\\ J\Subset I }} 
\mathbb E ^{\sigma } _{I_J} \Delta ^{\sigma }_I f \cdot   \langle T _{\sigma } I_J, \Delta ^{w} _{J} g \rangle _w  . 
\end{align}
\end{lemma}

\begin{proof}
This variant of the `Hilbert-Poisson exchange' argument is needed.  
Holding $ F\in \mathcal F$ fixed, we sum over $ J$ with $ \dot \pi _{\mathcal F} J= F$ and $ I$ with $ F\subsetneq I$. 
Then, the argument of  $ T _{\sigma }$ is $ I_F$ which is written as $ I_F = F + (I_F \setminus F)$.  
Defining $ \varepsilon _F $ by 
\begin{equation*}
\sum_{I \;:\; I \supsetneq F} \mathbb E ^{\sigma } _{I_J} \Delta ^{\sigma }_I f 
\coloneqq  \varepsilon _F \alpha _{f} (F), 
\end{equation*}
these constants are bounded by a constant: $ \lvert  \varepsilon _F\rvert \lesssim 1 $.  
Then, 
\begin{align*}
\Phi (F) &\coloneqq \Bigl\lvert 
\sum_{\substack{I \::\:   I\supsetneq F}}  \sum_{J \;:\; \dot \pi _{\mathcal F} J=F} 
\mathbb E ^{\sigma } _{I_ {F} } \Delta ^{\sigma } _{I} f  \cdot \langle T_{ \sigma } F, \Delta ^{w} _{J} g \rangle _{w} 
\Bigr\rvert
\\
&= 
\Bigl\lvert 
\Bigl\langle 
 T_{ \sigma } F
,
 \sum_{J \;:\; \dot \pi _{\mathcal F} J=F} 
\varepsilon _J \Delta ^{w } _{J} g 
\Bigr\rangle_{w}
\Bigr\rvert
 \le 
\mathscr T \alpha _{f} (F)\sigma (F)  ^{1/2} \lVert Q ^{w} _{F} g\rVert_w 
 \end{align*}
 This depends upon the testing assumption on $ T _{\sigma }$ applied to intervals. 	
 The operator $ Q ^{w} _{F} g$ is the Haar projection  defined at \eqref{e:PQ}.
 Quasi-orthogonality as in \eqref{e:quasi} finishes the 
 sum over $ F\in \mathcal F$. 
 
 The complementary case is that of the global-to-local reduction. 
 But, under the pivotal condition there is a geometric decay along the stopping tree.  
For $ F  \in \mathcal F$, and integer $ j$, let $ \textup{ch} _{j} (F)$ be the $ j$-fold 
descendants of $ F$ in the collection $ \mathcal F$.  That is,   $ \textup{ch} _{0} (F) = \{F\}$, 
 and $ F' \in  F'\in  \textup{ch} _{j+1} (F)$ iff $ F'$ is the child 
of some interval $ F''\in  \textup{ch} _{j} (F)$.

We will index by  $ F\in \mathcal F$,   $ F'\in  \textup{ch} _{1} (F)$, 
 and $ F'' \in  \textup{ch} _{j} (F')$, where $ j\geq 0$.  
Using  \eqref{e:off} and critically, Lemma~\ref{l.donotuse}, we have 
\begin{align*}
\Bigl\lvert \sum_{I \;:\; \pi _{\mathcal F} I=F} 
\mathbb E ^{\sigma } _{I_ {F'} } f  \langle T _{\sigma  } (I _{F'} \setminus F' ),   Q ^{w} _{F''} g\rangle_w
\Bigr\rvert
& \lesssim  \alpha _{f} (F) 
\mathbb P (\sigma \cdot  (F \setminus F'), F'')  w (F'') \lVert  Q ^{w} _{F''} g\rVert_w 
\\
& \lesssim    \alpha _{f} (F)  2 ^{ (1- \epsilon ) j} \mathbb P (\sigma \cdot  (F \setminus F''), F'')  w (F'') ^{1/2} \lVert  Q ^{w} _{F''} g\rVert_w . 
\end{align*}
We have geometric decay in $ j$ above.  Moreover, summing over $ F'$ and $ F''$, we can appeal to the pivotal condition 
\eqref{e:pivotal} to see that 
\begin{align*}
\sum_{F''\in  \textup{ch} _{j+1} (F)} &\mathbb P (\sigma \cdot  (F \setminus F''), F'')  w (F'')  ^{1/2} \lVert  Q ^{w} _{F''} g\rVert_w 
\\&\leq                   
\biggl[\sum_{F''\in  \textup{ch} _{j+1} (F)} \mathbb P (\sigma \cdot  (F \setminus F''), F'') ^2  w (F'')  
\times 
\sum_{F''\in  \textup{ch} _{j+1} (F)}  \lVert Q ^{w} _{F''} g\rVert_w ^2 
\biggr] ^{1/2} 
\\
& \lesssim \mathscr P \biggl[ \sigma (F)  \sum_{F''\in  \textup{ch} _{j+1} (F)}  \lVert Q ^{w} _{F''} g\rVert_w ^2 
\biggr] ^{1/2} .  
\end{align*}
Then, quasi-orthogonality is used to estimate 
\begin{align*}
\sum_{F\in \mathcal F} \alpha _{f} (F) 
\biggl[ \sigma (F)  \sum_{F''\in  \textup{ch} _{j+1} (F)}  \lVert Q ^{w} _{F''} g\rVert_w ^2 
\biggr] ^{1/2} \lesssim \lVert f\rVert _{\sigma } \lVert g\rVert_w.  
\end{align*}
This completes the global to local reduction.  
\end{proof}

\subsection{The Local Estimate}
It remains to prove the following \emph{local estimate}: 
\begin{equation*}
\lvert  B ^{\textup{above}} (P ^{\sigma } _{F}f, g ) \rvert \lesssim \mathscr T \bigl\{ \alpha _{f} (F) \sigma (F) ^{1/2}
+  \lVert P ^{\sigma } _{F} f\rVert_{\sigma } \bigr\} \lVert g\rVert_{w}\,, \qquad  Q ^{w} _{F} g = g, 
\end{equation*}
for then quasi-orthogonality will complete the bound on $ B ^{\textup{above}} _{\mathcal F} (f,g)$. 

In the bilinear form above, the argument of $ T _{\sigma }$ is, for a pair of intervals $ J\Subset I$,  $ I_J = (F- I_J) + F$.  
Using linearity, and focusing on the argument of $ T _{\sigma }$ being $ F$, 
we can repeat the argument of \eqref{e:badMart}, which depends upon the fact that the averages of $ f$ are controlled. 
Below, there is an  requirement that $ I_J$ has $ \mathcal F$-parent $ F$, which we are free to add since $ Q ^{w} _{F} g = g$. 
\begin{equation*}
\Bigl\lvert 
\sum_{I \::\: \pi _{\mathcal F}I=F} \sum_{J \::\:   \dot\pi _{\mathcal F} I_J=F} 
\mathbb E ^{\sigma } _{I_J} \Delta ^{\sigma } _{I} f  \cdot \langle T_{ \sigma } F, \Delta ^{w} _{J} g \rangle _{w} 
\Bigr\rvert 
\lesssim \mathscr T \alpha _{f} (F) \sigma (F) ^{1/2} \lVert g\rVert_{w} \,. 
\end{equation*}
This bound follows the argument of \eqref{e:badMart}, and we suppress the details.  

It therefore remains to consider the \emph{stopping form} 
\begin{equation*}
B ^{\textup{stop}} _{F} (f,g) \coloneqq  
\sum_{I \::\: \pi _{\mathcal F}I=F} \sum_{J \::\:   \dot\pi _{\mathcal F} I_J=F} 
 \mathbb E ^{\sigma } _{I_J} \Delta ^{\sigma }_I f \cdot \langle T _{\sigma } (I_0 - I_J) ,  \Delta ^{w} _{J} g\rangle _{w}\,. 
\end{equation*}

\begin{lemma}\label{l:stopT} For all $ F\in \mathcal F$, there holds 
\begin{equation*}
\lvert  B ^{\textup{stop}} _{F} (f,g) \rvert \lesssim \mathscr P \lVert f\rVert_{\sigma } \lVert g\rVert_{w} \,.  
\end{equation*}
\end{lemma}

\begin{proof}
This depends very much on the selection of stopping intervals.  In fact there is geometric decay, holding the 
relative lengths of $ I$ and $ J$ fixed.  Estimate for integers $ s\ge r$, 
\begin{align*}
\sum_{I \::\: \pi _{\mathcal F}I=F} \sum_{\substack{J \::\: J\Subset I_J\,, \pi _{\mathcal F} I_J=F\\ \lvert  I\rvert= 2 ^{s} \lvert  J\rvert   }} &
\lvert   \mathbb E ^{\sigma } _{I_J} \Delta ^{\sigma }_I f \cdot \langle T _{\sigma } (I_0 - I_J) ,  \Delta ^{w} _{J} g\rangle _{w}\rvert 
\\& \le 
\sum_{I \::\: \pi _{\mathcal F}I=F} \sum _{\theta \in \{\pm\}}
\frac {\lvert  \hat f (I) \rvert } {\sigma (I _{\theta }) ^{1/2} } 
\sum_{\substack{J \::\: J\Subset I_ \theta \,, \pi _{\mathcal F} I_J=F\\ \lvert  I\rvert= 2 ^{s} \lvert  J\rvert   }} 
P (\sigma (F- I _{\theta }), J) \langle \tfrac x {\lvert  J\rvert },  h ^{w} _{J}\rangle _{w} \lvert  \hat g (J) \rvert 
\\
& \lesssim  M_s 
\Bigl[\sum_{I \::\: \pi _{\mathcal F}I=F}  \hat f (I) ^2 \Bigr] ^{1/2} 
\times 
\Bigl[\sum_{J \::\: J\subset F}  \hat g (I) ^2 \Bigr] ^{1/2} 
\\
\textup{where } \quad 
M_s ^2  &\coloneqq  \max _{\theta \in \{\pm\}}\sup _{I \::\: \pi _{\mathcal F} I _{\theta }=F} 
\frac 1 {\sigma (I _{\theta })} 
\sum_{\substack{J \::\: J\Subset I_ \theta \,, \pi _{\mathcal F} I_J=F\\ \lvert  I\rvert= 2 ^{s} \lvert  J\rvert   }} 
P (\sigma (F- I _{\theta }), J) ^2 w (J) \,. 
\end{align*}
Here, we have used  (a) used the bound  $ \lvert  \mathbb E ^{\sigma } _{I_J} \Delta ^{\sigma }_I f\rvert 
\le\frac {\lvert  \hat f (I) \rvert } {\sigma (I _{\theta }) ^{1/2} }  $; (b) appealed to  \eqref{e:off}; (c) used Cauchy--Schwarz, together with the fact that for $ J\Subset F$, there is a unique 
$ I$ containing it, with length $ 2 ^{s} \lvert  J\rvert $.  

It remains to bound $M_s$, gaining a geometric decay in $ s$, and appealing to the pivotal condition.  
Return to the inequality \eqref{e:donotuse}, to gain the geometric decay, 
\begin{align*}
\sum_{\substack{J \::\: J\Subset I_ \theta \,, \pi _{\mathcal F} I_J=F\\ \lvert  I\rvert= 2 ^{s} \lvert  J\rvert   }} 
P (\sigma (F- I _{\theta }), J) ^2 w (J)
& \lesssim 2 ^{- (1- \varepsilon )s} P (\sigma \cdot F, I _{\theta }) ^2 w (I _{\theta }) \lesssim 
2 ^{- (1- \varepsilon )s}  \mathscr P ^2 \sigma (I _{\theta }) \,, 
\end{align*}
where the  decisive point is that $ I _{\theta }$ has $ \mathcal F$-parent $ F$, hence it must \emph{fail} the 
pivotal stopping condition.  
\end{proof}

\section{Example Weights} \label{s:examples}

The sharpness of the different conditions in the main theorem is the subject 
of the this section.  

\begin{theorem}\label{t:examples}   There are pairs of weights $\sigma , w$, with no common point masses,  that satisfy any one of these conditions.  
\begin{enumerate}
\item  The pair of weights satisfies the  full Poisson $A_2$ condition, but the norm inequality  for the Hilbert transform \eqref{e:N} does not hold. 
\item  The pair of weights satisfies the  full Poisson $A_2$ condition, and the testing inequality \eqref{e:T}, but  the norm inequality 
for the Hilbert transform  \eqref{e:N} does not hold. 
\item The pair of weights satisfy the two weight norm inequality \eqref{e:N}, but not the pivotal condition 
	\eqref{e:pivotal}. 
\end{enumerate}
\end{theorem}

Point (1) is a counterexample to Sarason's Conjecture, first disproved by Nazarov \cite{N1}.  
In contrast to his argument,  an explicit pair of weights are exhibited.

\subsection{The Initial Steps in the Main Construction}

Let $ C = \bigcap _{n=0} ^{\infty } C_n$ be the standard middle third Cantor set in the unit interval.  
Thus, $ C_0=[0,1]$, $ C_1 = [0,\frac 13] \cup [\frac 23,1]$, and more generally
\begin{equation*}
	C _{n} = \bigcup \bigl\{ [x, x+ 3 ^{-n}] \::\:  x = \sum_{j=1} ^{n} \epsilon _j 3 ^{-j}\,,\  \epsilon _j \in \{0,2\}\bigr\}\,.
\end{equation*}
Let $ w$ be the standard uniform measure on $ C$. 
Thus $ w (I)= 2 ^{-n}$ on each component of $ C_n$, $ n\in \mathbb N _0$.  
This is phrased slightly differently.  Let $ \mathcal K$ be the collection of components of all the sets $ C_n$. 
Then, for each $ K\in \mathcal K$, there holds $ w (K) = \lvert  K\rvert ^{\frac {\ln 2} {\ln 3}} $.

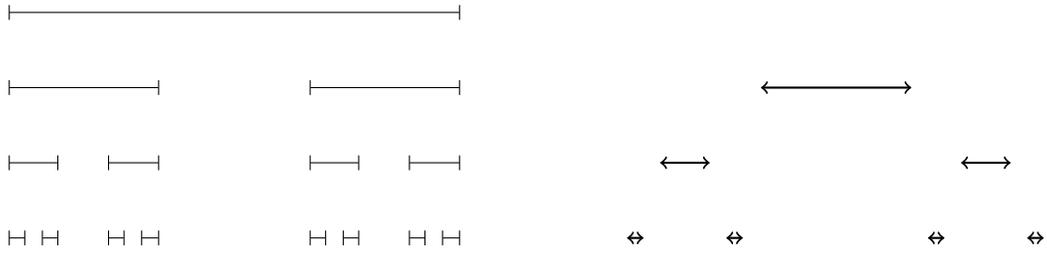
\begin{figure}
\begin{tikzpicture}[scale=2]
\draw[|-|] (0,0) -- (3,0); 
\foreach \x/\y in {0/1 , 2/3}  \draw[|-|] (\x,-.5) -- (\y, -.5); 
\foreach \x/\y in {0/0.33 , 0.66/1, 2/2.33, 2.66/3}  \draw[|-|] (\x,-1) -- (\y, -1); 
\foreach \x/\y in {0/0.11 , 0.22/0.33,   0.66/0.77, 0.88/1,  2/2.11,  2.22/2.33,  2.66/2.77, 2.88/3}  \draw[|-|] (\x,-1.5) -- (\y, -1.5); 

\draw[<->,thick,xshift=4cm] (1,-.5) -- (2,-.5); 
\foreach \x/\y in {0.33/0.66 , 2.33/2.66}  \draw[<->,thick,xshift=4cm] (\x,-1) -- (\y, -1); 
\foreach \x/\y in {0.11/0.22 , 0.77/0.88, 2.11/2.22, 2.77/2.88}  \draw[<->,thick,xshift=4cm] (\x,-1.5) -- (\y, -1.5); 
\end{tikzpicture}
\caption{The approximates to the Cantor set $ C$ on the left, and on the right, the gaps, namely  the components of $ [0,1]-C$ .  
The intervals on the left are in $ \mathcal K$, and those on the right are in $ \mathcal G$.}  
\label{f:cantor}
\end{figure}

The weight $ \sigma $ will be a sum of point masses selected from the intervals in  $ \mathcal G$, taken to be the components of the open set $ [0,1]-C$. ($ G$ is for `gap.') 
Consider the $ H w $ restricted an interval $ G\in \mathcal G$.  This is a smooth, monotone function, hence it has a unique zero $ z_G$.  
Then, the weight $ \sigma $ is 
\begin{equation} \label{e:x:sigma}
\sigma \coloneqq  \sum_{G\in \mathcal G} s_G  \cdot \delta _{z_G}\,, 
\end{equation}
where $ s_G>0$ will be chosen momentarily, consistent with the $ A_2$ condition. 
A second measure is given by 
$
\sigma' \coloneqq  \sum_{G\in \mathcal G} s_G  \cdot \delta _{z_G'}
$,
where $ z'_G$ is the unique point in $ G$ at which $ Hw (z'_G)= \lvert  G\rvert ^{-1+ \frac{\ln 2} {\ln 3}}  $.  
See Figure~\ref{f:z}.  

\begin{figure}
\begin{tikzpicture}[samples=50,domain=-1.1:1.1]
\draw plot  (\x,{tan(\x r)}); 
\draw[|-|]  (1.1,.2) -- (-1.1,.2)  node[above] {$ G$}; 
\draw (.2,.2) -- (.2,0) ; \draw (.2,.4) node[above] {$ z_G$}; 
\draw (.6,.2) -- (.6,0) node[below] {$ z'_G$}; 
\end{tikzpicture}
\caption{The selection of the points $ z_G$ and $ z'_G$ for a gap interval $ G$.  The function $ Hw$, restricted to $ G$ is monotone increasing, 
hence has a unique zero, the point $ z_G$.  The second point $ z'_G$ will be to the right, but a distance to the boundary of $ G$ that is at 
least $ c \lvert  G\rvert $, for absolute constant $ 0<c<\frac 12 $. }
\label{f:z}
\end{figure}
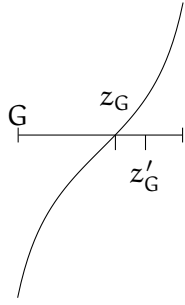

The constants $ s_G$ are be specified by the simple $ A_2$ ratio 
\begin{gather}
\frac {w (3G)} {\lvert  G\rvert } \cdot \frac {\sigma (G)} {\lvert  G\rvert } = 2
\,, \qquad 
\textup{that is} \quad s_G = 2\lvert  G\rvert ^{2- \frac {\ln 2} {\ln 3}} \,. 
\end{gather}
To see this, note that
\begin{equation*}
w (3G)= w (G- \lvert  G\rvert ) + w (G+ \lvert  G\rvert ) = 2 \lvert  G\rvert ^{\frac {\ln 2} {\ln 3}} \,, 
\end{equation*}
since $ G \pm \lvert  G\rvert $ are components of some $ C_n$.   With this definition, the basic facts about the $ w$ and $ \sigma $ 
come from the geometry of the Cantor set and the relations below,  
\begin{gather} \label{e:x:simple}
\begin{split}
w (I)    &\lesssim   \lvert  I\rvert ^{\frac {\ln 2} {\ln 3}}\,, \qquad 
\qquad  \textup{$ I  $ is triadic,}
\\  
 \sigma (I) & \lesssim  \lvert  I\rvert ^{2 - \frac {\ln 2} {\ln 3}} 
\qquad  \textup{$ I  $ is triadic, $ I$ not strictly contained in any $ G\in \mathcal G$.}
\end{split}
\end{gather}
On the other hand, if $ I \in \mathcal G \cup \mathcal K$, the inequalities above can be reversed, namely 
\begin{equation}\label{e:xsimple}
w (3I) \simeq \lvert  I\rvert ^{\frac {\ln 2} {\ln 3}}\,, \quad \sigma (I) \simeq \lvert  I\rvert ^{2 - \frac {\ln 2} {\ln 3}}\,, \qquad 
 I \in \mathcal G \cup \mathcal K \,. 
\end{equation}

The properties of these measures that we are establishing are as follows.

\begin{lemma}\label{l:exampes} For the measures just defined, there holds 
\begin{enumerate}
\item   The Hilbert transform $ H _{\sigma }$ is bounded from $ L ^{2} (\sigma )$ to $ L ^{2} (w)$. 

\item   The Hilbert transform $ H _{\sigma '}$ is \emph{unbounded} from $ L ^{2} (\sigma '  )$ to $ L ^{2} (w)$, 
but the pair of weights satisfy the $ A_2$ condition, and the testing conditions 
\begin{equation*}
 \sup _{ \textup{$ I$ an interval}} \sigma ' (I) ^{-1} \int _{I}  \lvert  H _{\sigma '}  I\rvert ^2 \; d w< \infty \,.  
\end{equation*}
\end{enumerate}
 
\end{lemma}

Concerning point 2, the unboundedness of $ H _{w}$ is direct from the construction of $ \sigma '$. 
\begin{align}
\int (H w) ^2 \; d \sigma '  &= 
\sum_{G\in \mathcal G} H w (z'_G) ^2 \sigma ' ( \{z'_G\}) 
\\ \label{e:x:zI} 
&= \sum_{G\in \mathcal G}  \lvert  G\rvert  ^{2- \frac {\ln 2} {\ln 3} -2(1-\frac{\ln 2} {\ln 3}) } 
= \sum_{G\in \mathcal G}  \lvert  G\rvert  ^{+ \frac {\ln 2} {\ln 3}  }  = \infty \,. 
\end{align}
There are exactly $ 2 ^{n-1}$ elements of $ \mathcal G$ of length $ 3 ^{-n}$, proving the  sum is infinite.

\subsection{The Poisson $ A_2$ Condition}

\begin{lemma}\label{l:xA2} For either weight $ \mu \in \{\sigma , \sigma '\}$, the pair of weights $ w, \mu $ 
satisfy the $ A_2$ condition. 
\end{lemma}

\begin{proof}
It suffices to check the $ A_2$ condition on the triadic intervals in the unit interval.   
Let us begin by showing that for  any triadic interval $ I \in \mathcal K \cup \mathcal G$, 
\begin{equation} \label{e:xa2}
P (\sigma , I) \lesssim \frac {\sigma (I)} {\lvert  I\rvert } \,, \quad 
\textup{and} \quad 
P (w , I) \lesssim \frac {w (3I)} {\lvert  I\rvert } \,. 
\end{equation}
For then, the control of the simple $ A_2$ ratio will imply the control of the full $ A_2$ ratio.  
(For the inequality on $ w$, the triple of the interval appears on the right, since $ w (I)$ can be zero if $ I\in \mathcal G$.)
Now, it will be clear that this argument is insensitive to the location of the points $ z_G $ and $ z'_G$, so 
the same argument for $ \sigma $ will work equally well for $ \sigma '$. 

Let us consider $ \sigma $. Using \eqref{e:xsimple}, there holds 
\begin{align*}
P (\sigma , I)  & \le \frac {\sigma (I)} {\lvert  I\rvert }
+ \sum_{k=1} ^{\infty } \int _{3 ^{k}I \setminus 3 ^{k-1}I} 
\frac {\lvert  I\rvert } { (\lvert  I\rvert ^2 + \textup{dist} (x,I) ) ^2 } \; \sigma (dx) 
\\
& \lesssim 
\frac {\sigma (I)} {\lvert  I\rvert } +  \sum_{k=1} ^{\infty }  \frac {\sigma (3 ^{k} I)} { 3 ^{k}   \lvert  3 ^{k} I\rvert   } 
\\
& \lesssim 
\frac {\sigma (I)} {\lvert  I\rvert } +  \sum_{k=1} ^{\infty }   3 ^{-k} \lvert  3 ^{k} I\rvert ^{1- \frac {\ln 2} {\ln 3}}  
 \lesssim 
\frac {\sigma (I)} {\lvert  I\rvert }  \sum_{k=0} ^{\infty } 3 ^{-k \frac {\ln 2} {\ln 3} } \lesssim \frac {\sigma (I)} {\lvert  I\rvert }   \,. 
\end{align*}

Turning to the weight $ w$, one has 
\begin{align*}
P (w , I)  & \le \frac {w (3I)} {\lvert  I\rvert }
+ \sum_{k=2} ^{\infty } \int _{3 ^{k}I \setminus 3 ^{k-1}I} 
\frac {\lvert  I\rvert } { (\lvert  I\rvert ^2 + \textup{dist} (x,I) ) ^2 } \; w (dx) 
\\
& \lesssim 
\frac {w(3I)} {\lvert  I\rvert } +  \sum_{k=2} ^{\infty }  \frac {w(3 ^{k} I)} { 3 ^{k}   \lvert  3 ^{k} I\rvert   } 
\\
& \lesssim 
\frac {w(3I)} {\lvert  I\rvert } +  \sum_{k=2} ^{\infty }   3 ^{-k} \lvert  3 ^{k} I\rvert ^{-1+ \frac {\ln 2} {\ln 3}}  
 \lesssim 
\frac {w(3I)} {\lvert  I\rvert }  \sum_{k=1} ^{\infty } 3 ^{-k(2- \frac {\ln 2} {\ln 3}) } \lesssim \frac {w(3I)} {\lvert  I\rvert }   \,. 
\end{align*}

\medskip 

The $ A_2$ product $ P (\sigma ,I) \cdot P (w,I)$ has been bounded for $ I\in \mathcal K \cup \mathcal G$.  
Suppose that $ I$ is a triadic interval that is not in these two collections.  
Then, $ I$ must be strictly contained in some gap $ G\in \mathcal G$.  
Writing $ I ^{(k)}=G$, where, $ I ^{(k)}$ denotes the $ k$-fold parent of $ I$ in the triadic grid, we have $ w (G)=0$. 
Hence,  
\begin{equation*}
P (w , I) = \int _{[0,1] \setminus G} 
\frac {\lvert  I\rvert }  { (\lvert  I\rvert + \textup{dist} (x,I) ) ^2 } \; w (dx)
\simeq 3 ^{-k} 
P (w , G). 
\end{equation*}
First, consider $ \sigma $ restricted to the gap $ G$: 
\begin{equation*}
P (w,I)  P (\sigma \cdot G, I )
\lesssim  3 ^{-k}  P (w , G) \frac {\sigma (G)} {\lvert  I\rvert }  \simeq 
 P (w , G) \frac {\sigma (G)} {\lvert  G\rvert } \lesssim 1. 
\end{equation*}
Now, we have to consider the Poisson average of $ \sigma $ off of the gap $ G$, in which case we have 
\begin{equation*}
P (\sigma \cdot ([0,1] \setminus G), I) \simeq 3 ^{-k} P (\sigma , G)\,,   
\end{equation*}
and so the estimate follows. 

\end{proof}

\subsection{The Testing Conditions}

We turn to the testing conditions, using in an essential way the precise definition of the weight $ \sigma $: it gives 
a huge cancellation, which simplifies things considerably. 

\begin{lemma}\label{l:Xtesting} For any interval $ I$, there holds 
\begin{equation}\label{e:Xtesting}
\int _{I} \lvert  H _{w} I\rvert ^2 \; d \sigma  \lesssim w (I) \,.  
\end{equation}
\end{lemma}

\begin{proof}
By construction of $ \sigma $, there are two reductions. 
The first is simple, namely that the two endpoints of the interval $ I$ can be taken to be an endpoint of an interval in $ \mathcal G$. 
The second comes from the construction of $ \sigma $:  $ H w \equiv 0$, relative to $ d \sigma $ measure.   Hence, 
\begin{equation*}
\int _{I} \lvert  H _{w} I\rvert ^2 \; d \sigma = \int _{I}   H _{w} ([0,1] -I)^2 \; d \sigma\,,  
\end{equation*}
namely the \emph{complement of $ I$} is the argument of the Hilbert transform on the right.  

Then, one abandons all further cancellations.  Let us show that for all intervals $ K\in \mathcal K$ (the components of the 
	sets $ C_n$ which generate the Cantor set), 
\begin{equation}\label{e:xK<}
\int _{K}  \lvert  H _{w } K _{\textup{rt}}\rvert ^2 \; d \sigma \lesssim w (K)\,, 
\end{equation}
where $ K _{\textup{rt}}$ is the right component of $ [0,1] \setminus K$. The same estimate holds for the left component, and this 
completes the proof.  For, if we set $ I _{\textup{rt}}$ to be the right component of $ [0,1] \setminus I$, 
and take $ K^1 , K ^{2}, \dotsc, $ to be the maximal intervals in $ K$ contained in $ I$, there holds 
\begin{align*}
\int _{I} (H _{w } I _{\textup{rt}}) ^2 \; d \sigma 
& \le \sum_{n=1} ^{\infty } \int _{K^n}(H _{w } K ^n_{\textup{rt}}) ^2 \; d \sigma 
\\
& \lesssim  \sum_{n=1} ^{\infty } w (K^n) \lesssim w (I)\,. 
\end{align*}

Now, for $ K\in \mathcal K$, let $ K_1 , K_2, \dotsc, $ be the maximal intervals in $ \mathcal K$ that lie to the right of $ K$.  
Arranging them in increasing length, note that the length of $ K_1$ is either $ \lvert  K\rvert $ or $ 3 \lvert  K\rvert $.  
For $ n\ge 2$, the length of $ K _{n}$ increases by a factor of 3, and $ \textup{dist} (K, K_n) \gtrsim \lvert  K_n\rvert $, and 
hence there are  at most $ 1 - \log _{3} \lvert  K\rvert $ such intervals in $ \mathcal K$.  
Here is an illustration:

\begin{center}
\begin{tikzpicture}[scale=1.3]
\draw[thick] (0,0) -- (0.33,0) node[above,midway] {$ K _{\phantom 0}$};
\foreach \x/\y/\n in {0.66/1/1,2/3/2,6/9/3} \draw[thick] (\x,0) -- (\y,0) node[above,midway] {$ K_{\n}$}; 
\end{tikzpicture}
\end{center}

Then, one has the estimate below, where the sum is of a decreasing  geometric series, estimated by its first term.   
\begin{equation*}
\lvert  H _{w } K _{\textup{rt}}\rvert  
\lesssim \sum_{n=1} ^{\infty } \frac {w (K_n)} {\lvert  K_n\rvert } \simeq \frac {w (K)} {\lvert  K\rvert } \,. 
\end{equation*}
Hence, \eqref{e:xK<} follows from the control of the $ A_2$ ratio.  

\end{proof}

An important part of the remaining arguments is that points $ z_G $, and  $ z'_G$ cannot cluster close to the boundary of $ G$. 

\begin{lemma}\label{l:xclose} There is a constant $0 < c < \tfrac 12 $ such that 
\begin{equation*}
\lvert  z_G - z'_G\rvert   \le   c \lvert  G\rvert\,.  
\end{equation*}
\end{lemma}

\begin{proof}

Estimate $ H w$ at the midpoint $ z''_G$ of a component $ G$. 
By symmetry of the Hilbert transform, and the Cantor set, it always holds that 
$ H (w \mathbf 1_{3G}) (z'_G)=0$, so that appealing to \eqref{e:x:simple}, 
\begin{align*}
\lvert  H w (z''_G) \rvert &= \lvert  H (w \mathbf 1_{(3G) ^{c}}) (z''_G) \rvert 
\\
&  \lesssim \sum_{k=2} ^{n }   \frac {w (3 ^{k} G )} {  \lvert  3 ^{k}G\rvert   } 
 \\&\lesssim \sum_{k=2} ^{n }  \lvert  3  ^{k}G\rvert ^{-1+\frac {\ln 2} {\ln 3}} \lesssim  \lvert  G\rvert ^{-1+\frac {\ln 2} {\ln 3}} 
\end{align*}

Next, we turn to a derivative calculation. 
The function $ H w $, restricted to $ G$ is a smooth function, one that diverges at the end points of $ G$ at a 
rate that reflect the fractal dimension of $ G$.  For any $ x\in G$ note that 
\begin{align*}
\frac d {dx} H w (x)  &  =    \int _{C} \frac {w (dy)} { (y-x) ^2 }
\\&
\gtrsim \frac {w (3G)} {\lvert  G\rvert ^2  } \simeq  \lvert  G\rvert ^{-2 + \frac {\ln 2} {\ln 3}} \,.  
\end{align*}
This is a uniform lower bound, and in fact the lower bound is very poor at the boundaries of $ G$.  
Indeed, 
\begin{equation*}
\frac d {dx} H w (x) \gtrsim  \textup{dist} (x, \partial G) ^{-2+\frac {\ln 2} {\ln 3}} \,. 
\end{equation*}
It follows that we have to have $ \lvert  z_G - z'_G\rvert < c \lvert  G\rvert $, for some $ 0< c < \tfrac 12 $.  
That is, one need only move at fixed small multiple of $ \lvert  G\rvert $, passing from the location of the zero $ z_G$ to  the point $ z'_G$.

\end{proof}

The second half of the testing intervals inequalities is as follows. 

\begin{lemma}\label{l:xtesting} For $ \mu \in \{\sigma , \sigma '\}$, and any interval $ I$,  
\begin{equation}\label{e:xtesting}
\int _{I} \lvert  H _{\mu } I  \rvert ^2 \; dw \lesssim \mu (I) \,.  
\end{equation}

\end{lemma}

\begin{proof}
For the sake of specificity, let $ \mu = \sigma $.  Indeed, by Lemma~\ref{l:xclose}, the same argument will work for   $ \sigma' $.  
To fix ideas, let us assume that $ I\in \mathcal K$. 
Write the left, middle and right thirds of $ I$ as $ I _{-1}, I _{0}, I _{1}$, respectively.  Then,  note that 
\begin{align}
\int _{I} H _{\sigma } (I) ^2 \; dw &= \int _{I _{-1} \cup I _{1} } H _{\sigma } (I) ^2 \; dw 
\\  \label{e:xFix}
& \lesssim  \int  _{I _{-1} \cup I _{1} }  H _{\sigma } (I_0) ^2 \; dw  +
\int _{I _{-1}} H _{\sigma } ( I_0+ I _{1}) ^2 \; dw + \int _{I _{1}} H _{\sigma } (I _{-1}+ I _{0}) ^2 \; dw 
\\  \label{e:xxFix}
& \qquad + \int _{I _{-1}} H _{\sigma } (I _{-1}) ^2 \; dw + \int _{I _{1}} H _{\sigma } (I _{1}) ^2 \; dw 
\,. 
\end{align}
The first term on the right is simple. On the  interval $ I_0$, $ \sigma $ is    a point mass, at a point that is 
at distance $ \ge c \lvert  I\rvert $ from $  I _{\pm 1}$.  Thus, by \eqref{e:xsimple}, 
\begin{equation*}
 \int  _{I _{-1} \cup I _{1} }  H _{\sigma } (I_0) ^2 \; dw
\lesssim \frac {\lvert  I\rvert ^{4- 2\frac {\ln 2} {\ln 3}} } {\lvert  I\rvert ^2 } \lvert  I\rvert ^{\frac {\ln 2} {\ln 3}} \simeq \sigma (I) \,. 
 \end{equation*}
That completes the first integral.  The remaining two integrals in \eqref{e:xFix} are handled by a similar argument.

Concerning the two integrals in \eqref{e:xxFix}, one should note that $ I _{\pm1} \in \mathcal K$
and that $ \sigma (I _{\pm1}) \le 3 ^{-2+2 \frac {\ln 2} {\ln 3}} \sigma (I)$. 
This geometric factor is smaller than $ \tfrac 12 $, therefore one can recurse on \eqref{e:xFix} and \eqref{e:xxFix} to see that 
\begin{equation}\label{e:xK}
\int _{K} H _{\sigma } (K) ^2 \; dw  \lesssim \sigma (K) \,, \qquad K\in \mathcal K \,. 
\end{equation}

\smallskip 
For a general interval $ I$, since $ \sigma $ is a sum of Dirac masses,  we can assume that the interval $ I$ is in a canonical form. 
Namely, each endpoint of $ I$ can be assumed to be an endpoint of an interval in $ \mathcal G$.  
The basic inequality is 
\begin{equation}\label{e:xT1}
\sum_{K\in \mathcal K_I} \int _{K}  \lvert  H _{\sigma } (I-K)\rvert ^2 \; d w \lesssim \sigma (I) \,, 
\end{equation}
where $ \mathcal K_I $ is the  maximal elements of $ \mathcal K$ contained in $ I$.
The integration is over $ K$, and  the argument of the Hilbert transform is $ I-K$.  

To see that \eqref{e:xT1} implies the Lemma, note that by \eqref{e:xK}, 
\begin{align*}
\int _{I} H _{\sigma } (I) ^2 \; dw & = \sum_{K\in \mathcal K_I } \int _{K}  H _{\sigma } (I) ^2 \; dw 
\\
& \lesssim \sum_{K\in \mathcal K_I } \int _{K}  H _{\sigma } (I-K) ^2 \; dw + \sum_{K\in \mathcal K_I } \int _{K}   H ^{\sigma } (K) ^2 \; dw \\
& \lesssim \sigma (I) + \sum_{K\in \mathcal K_I } \sigma (K) \lesssim \sigma (I)\,.  
\end{align*}

In fact, \eqref{e:xT1} follows from 
\begin{equation}\label{e:xT}
\int _{K}  \lvert  H _{\sigma } (I-K)\rvert ^2 \; d w \lesssim \frac {\sigma (I) ^2 } {\lvert  I\rvert ^2  } w (K) \,, \qquad K\in \mathcal K_I \,. 
\end{equation}
For this is summed over $ K\in \mathcal K_I$, and then one uses the $ A_2$ property.  

\smallskip 

To prove \eqref{e:xT}, all hope of cancellation is abandoned.  For an interval $ K\in \mathcal K_I$, let us consider component  $ I_{\textup{rt}}$ of $ I-K$ which lies to the right of $ K$.  It has a Whitney like decomposition into a finite sequence of intervals $ J _{1} ,\dotsc, J _{t}$ that we construct now.  
These intervals will have the property that they are (a) pairwise disjoint, (b) their union is $ I_{\textup{rt}}$, (c) and $ \textup{dist} (K , \textup{supp} (\sigma J_s)) \gtrsim \lvert  J_s\rvert \gtrsim 3 ^{\frac s2} \lvert  K\rvert $, for all $1\le  s \le t$.

Now, $J_1 = K+ \lvert  K\rvert \in \mathcal G$.  If this interval is not contained in $ I$, it follows that $ K$ contains the right hand endpoint of $ I$,  and there is nothing to prove.  
Assuming that $ J_1\subset I$, the inductive step is this.  Given $ J_1 ,\dotsc, J_s$, as above, whose union is not $I_{\textup{rt}}$
\begin{enumerate}
\item  If $ J_s \in \mathcal G $, then $ J_s + \lvert  J_s\rvert \in \mathcal K$.  
If this interval is contained in $ I_{\textup{rt}}$, then we take $ J _{s+1}=  J_s + \lvert  J_s\rvert \in \mathcal K$, and repeat the recursion. 
Otherwise, we update $ J_s \coloneqq  I_{\textup{rt}} - \bigcup _{u=1} ^{s-1} J_t$, and the recursion stops. 

\item  If $ J_s \in \mathcal K$, then it follows that $ J _{s-1} \in \mathcal G$, and the element of $   \mathcal G $ immediately to the right of $ J_s$  is $3 (J _{s} + 6 \lvert  J_s\rvert)  $.   If this interval is contained in $ I_{\textup{rt}}$, then 
we take $ J _{s+1}= 3 (J _{s} + 6 \lvert  J_s\rvert)   \in \mathcal G$, and repeat the recursion. 
Otherwise, we update $ J_s \coloneqq  I_{\textup{rt}} - \bigcup _{u=1} ^{s-1} J_t$, and the recursion stops.  
\end{enumerate}

With this construction, it follows that 
\begin{align*}
\lvert  H _{\sigma } (I_{\textup{rt}}) \cdot K\rvert 
&\lesssim \sum_{u=1} ^{t}  \frac {\sigma ( J_s)} {\lvert  J_s\rvert } 
 \lesssim \sum_{n=1} ^{\infty }  \lvert  J_s\rvert ^{1-\frac {\ln 2} {\ln 3}}  \lesssim   \frac {\sigma (I)} {\lvert I\rvert } \,. 
\end{align*}
This proves the `right half' of \eqref{e:xT}, that is, when the argument of the Hilbert transform is $ I _{\textup{rt}}$.
The `left half' is the same, so the proof is complete. 
\end{proof}

At this point, we have proven that the pair of weights $ (w, \sigma' )$ satisfy the full Poisson $ A_2$ condition,  
and the testing condition \eqref{e:xtesting}.  But, $ \lVert H w \rVert_{L ^2 (\sigma ')}$ is infinite, by \eqref{e:x:zI}.  
Hence, points (1) and (2) of Theorem~\ref{t:examples} are shown.   

We have also shown that the pair of weights $ (w, \sigma )$ satisfy the full Poisson $ A_2$ condition, and both sets of testing conditions. 
Hence, by our main theorem, $ H _{w}$ is bounded from $ L ^2 (w)$ to $ L ^2 (\sigma )$.   This pair of weights also fail the pivotal condition \eqref{e:pivotal} of Nazarov-Treil-Volberg \cite{10031596}.  This is verified by observing that the collection $ \mathcal G$ of gaps is a partition of $ [0,1]$, and 
\begin{align*}
\sum_{G\in \mathcal G}  P (w , G) ^2 w (G) 
& \simeq \sum_{G\in \mathcal G}   \frac {w (3G) ^2 } {\lvert  G\rvert ^2  } \sigma  (I ) 
\\
& \simeq \sum_{G\in \mathcal G} w (3G) \simeq \sum_{G\in \mathcal G} \lvert  G\rvert ^{\frac {\ln 2} {\ln 3}} = \infty 
\end{align*}
since $ \mathcal G$ contains $ 2 ^{n}$ intervals of length $ 3 ^{-n}$, for all integers $ n$.  
Here, we have used \eqref{e:xa2}, followed by \eqref{e:x:simple}.   
Since $\inf _{x\in G} M w (x) \gtrsim  P (w , G) $, this also shows that the maximal function $ M $ is not bounded from $ L ^2 (w)$ to $ L ^2 (\sigma ) $.  

Notice in contrast that the energy inequality \eqref{e:energy} for the partition $ \mathcal G$ is trivial, since $ \sigma $ restricted to any interval $ G$ is a point mass, hence $ E (\sigma ,G)=0$, for all $ G\in \mathcal G$.  

 \subsection{Context and Discussion}
  
\subsubsection{}
Counterexamples were an important source of inspiration on these questions.  
The early paper of Muckenhoupt and Wheeden \cite{MR0417671} includes an example of the fact that the 
simple $ A_2$ condition is not sufficient for the two weight inequality.  
For instance, the  boundedness of the simple $ A_2$ ratio is simple to check for the pair  $ w = \delta _0$, and $ \sigma (dx)= x\mathbf 1_{[0, \infty )} dx $.  Then, one sees that for $ f= \tfrac 1x \mathbf 1_{[1, L]}$, 
\begin{equation*}
\sqrt {\log L} \simeq \lVert f \rVert_{\sigma } 
 \ll   \log L \simeq   \lVert H _{\sigma } f  \rVert_{w} \,, \qquad  L > 1 \,. 
\end{equation*}
Thus, the Hilbert transform is unbounded.  And, one can directly see that the half-Poisson $ A_2$ condition fails.

Much harder, is the fact that the Poisson $ A_2$ condition is not sufficient. This was the contribution of Nazarov \cite{N1}.   This example lead to the conjecture of Nazarov-Treil-Volberg \cite{V} proved herein.   
A more delicate example, of a pair of weights which satisfied the Poisson $ A_2$ condition, and one set of testing conditions, say \eqref{e:T}, but not the norm inequality was that of Nazarov-Volberg \cite{NV}.  Also see Nikol{\cprime }ski{\u \i }-Treil \cite{MR1945291}, for a related example  to disprove a conjecture about similarity to a normal operator. Both of these latter examples were based upon Nazarov's indirect example.  

\subsubsection{}
The example given here is directly inspired by a Cantor set type example in Sawyer's two weight maximal function paper \cite{MR676801}. 
It is drawn from  \cite{10014043}, with the purpose to show that the \emph{pivotal condition} of Nazarov-Treil-Volberg \cites{V,10031596} was \emph{not necessary} for the two weight inequality to hold.  
This was an explicit example, and also pointed to the primary role of the notion of energy. 
It is very interesting and delicate, in that the point masses have to be placed on the zeros 
of the Hilbert transform, in order to obtain the boundedness of the transform.   
It is also humbling in that it still does not reveal how delicate the proof of the sufficiency in the main theorem needs to be.

\subsubsection{}
It is subtle example of Maria Carmen Reguera \cite{MR2799801} and Reguera-Thiele \cite{MR2923171} 
that proves this, as is pointed out by Reguera-Scurry \cite{1109.2027}.  

\begin{priorResults}\label{t:rs} 
	There is a pair of weights for which the maximal function $ M _{\sigma } $ is bounded from $ L ^2 (\sigma ) \to L ^2 (w)$ and $ M _{w} $ is bounded from $ L ^2 (w )\to L ^2 (\sigma )$, but   norm inequality for the Hilbert transform  \eqref{e:N} does not hold. 
\end{priorResults}

This is quite a bit more intricate than the examples we have presented.  
It had been suggested, in the early days of the weighted theory, that the boundedness of the maximal functions would be sufficient for the norm boundedness of the Hilbert transform.   On the other hand, if one considers `off-diagonal' estimates, then boundedness of the 
maximal function is sufficient for norm inequalities for singular integrals \cite{12035906}.

\section{Applications of the Main Inequality} 

The interest in the two weight problem stems from a range of potential applications arising in 
sophisticated arenas of complex function and spectral theory.  The motivations for these questions 
are complicated, and based upon subtle theories.
The connections to the two weight Hilbert transform are 
not always immediate, and the properties of interest are frequently more intricate than those of mere 
boundedness of a transform.  Nevertheless, the acknowledged experts Belov-Mengestie-Seip in \cite{MR2812502} write 
``\ldots we have found it both useful and conceptually appealing to transform the subject into a study of the mapping properties of discrete Hilbert transforms. We have learned to appreciate that the essential difficulties thus seem to appear in a more succinct form.'' 
A brief guide to the subjects, and some of the `essential difficulties' follow. 

\subsection{Sarason's Question on Toeplitz Operators}\label{s:sarason}

This question arose from Sarason's work on exposed points of $ H ^{1}$ \cite{MR1038352}. 
Indeed, this was part of an influential body of work that pointed to the distinguished role of de Branges spaces in the subject. 
This paper contains  examples of pairs of functions $ f,g$, for which the individual Toeplitz operators where unbounded, but the composition bounded.

\begin{question}[Sarason \cite{sarasonConj}] 
	Characterize those pairs of outer functions $ g, h \in H ^2 $ for which the composition of Toeplitz operators  $ T _{g} T _{\overline h }$ is bounded on $ H ^2 $. 
\end{question}

Following \cite{sarasonConj}, for a function $ h \in L ^2 (\mathbb T )$, the Toeplitz operator $ T _{h}$ can be thought of 
as taking $ f \in H ^{2}$ to the space of analytic functions by the definition
\begin{equation*}
	T_h f (z) \coloneqq  \frac 1 {2 \pi }\int _{\partial \mathbb D } f (\operatorname e ^{i \theta  }) h (\operatorname e ^{i \theta })
\overline {k _{z} (\operatorname e ^{i \theta })}\; d \theta \,, 
\end{equation*}
where $ k _{w} (z) \coloneqq  \frac { (1- \lvert  w\rvert ^2  ) ^{1/2}} { 1- \overline w z}$ is the reproducing kernel.  

Also in \cite{sarasonConj} is an argument of S.~Treil that  a Poisson $ A_2$ condition is 
necessary condition for the boundedness of the composition: 
\begin{equation} \label{e:PA2}
\sup _{z\in \mathbb D } P \lvert  f\rvert ^2  (z) P \lvert  g\rvert ^2 (z) < \infty \,, 
\end{equation}
where $ P$ denotes the Poisson extension to the unit disk.  
Sarason wrote that \emph{`It is tempting to conjecture that the last condition is also sufficient for the boundedness of $ T _{g} T _{\overline h}$.'} 
This statement, widely referred to as the Sarason Conjecture, is of interest in both the Hardy and Bergman space settings.%
(Aleman-Pott-Reguera \cite{13041750} have resolved the conjecture in the negative in 
a Bergman space setting. A striking argument in which they prove the boundedness of the Bergman projection 
is \emph{equivalent to} the boundedness of the positive part of the Bergman projection. This allows 
a much simpler counterexample to be identified.) 

The connection with the two weight problem for the Hilbert transform is indicated by the diagram from \cite{MR1294717}*{\S5}, 
see Figure~\ref{f:Sarason}.  
In the diagram, 
$ M _{\overline h}$ is multiplication by $ \overline h$ and $ P_+$ is the Riesz projection from $ L ^2 $ to $ H ^2 $.   
The boundedness is equivalent to 
\begin{equation*}
	M _{g} P_+ M _{\overline f} \::\: H ^2 \mapsto H ^2 \,. 
\end{equation*}

The structure of outer functions  leads to these simplifications.  
Since the product of analytic is analytic, the  second $ H ^2 $ above can be replaced by $ L ^2 $, and then, the outside multiplication $ M _{g}$ can then be replaced by $ M _{\lvert  g\rvert }$.  
Thus, we are considering $ M _{\lvert  g\rvert } P_+ M _{\overline f} \::\: H ^2 \mapsto L^2 $.  
Now, $ \overline f$ is anti-analytic, so we can replace $ H ^2 $ above by $ L ^2 $. 
Moreover, the multiplication operator $ M _{f/ \lvert  f\rvert }$ is unitary, 
since an outer function can be equal to zero on $ \mathbb T $ only on a set of measure zero. 
Thus, it is equivalent to consider 
\begin{equation*}
	M _{\lvert  g\rvert } P_+ M _{\lvert  f\rvert } \::\: L ^2 \mapsto L^2 \,. 
\end{equation*}
This is a two weight inequality for $ P_+$. (Sergei Treil helped us with the history of this question.) 

\begin{figure}
\begin{tikzpicture}[node distance=3cm, auto] 
\node (A) {$ H ^2 $};   \node[right of=A] (B) {$ H ^2 $}; 
\node[below of=A] (C)  {$ L ^2 (\lvert  h\rvert ^{-2} )$};    \node[below of=B] (D) {$ H ^2 (\lvert  g\rvert ^2  )$};    
\draw[->]  (A) to node {$T _{g} T _{\overline h}$} (B); 
\draw[->]  (C) to  node[above] {$P_+$} (D); 
\draw[->]  (A) to node[left] {$M _{\overline h}$} (C); 
\draw[->]  (D) to node[right] {$M_g$} (B);  
\end{tikzpicture}
\caption{Sarason's Question concerns the top line of the diagram, which is equivalent to the lower part of the diagram. 
The  operator  $ M _{\overline h}$ on the left  is an isometry onto its range, while   $ M _{g}$, the operator on the right is an isometry between the two spaces.  } 
\label{f:Sarason} 
\end{figure}

The Riesz projection is a linear combination of the identity and the Hilbert transform, and   our main theorem will apply to it. 
 Note that the inequality 
\begin{equation*}
\lVert P_+ (\lvert  f\rvert \phi  )\rVert_{L ^2 (\lvert  g\rvert ^2 dx   )} \lesssim \lVert \phi \rVert_{ L ^2 (dx)}
\end{equation*}
is equivalent to 
\begin{equation*}
\lVert P_+ (\lvert  f\rvert ^2  \psi   )\rVert_{L ^2 (\lvert  g\rvert ^2 dx   )} \lesssim \lVert \psi  \rVert_{ L ^2 ( \lvert  f\rvert ^2 dx)}\,. 
\end{equation*}
Recall that $ P_+ = I  -\frac  \pi i  H $, according to how we defined the Hilbert transform, where $ I$ represents the identity operator.  In the two weight setting, we interpret the norm inequality $ \lVert P _{+} (\sigma f) \rVert_{w} \lesssim \lVert f\rVert_{\sigma }$, as uniform over all truncations $ 0< \tau < 1$ defined by 
\begin{equation*}
P _{+, \tau } (\sigma f) \coloneqq  \sigma f + \frac i \pi \int _{\tau < \lvert  x-y\rvert < \tau ^{-1}  } f (y) \frac {\sigma (dy)} {y-x} \
\end{equation*}

\begin{theorem}\label{t:P} For pairs of weights $ w, \sigma $ that   absolutely continuous with respect to Lebesgue measure,  the norm inequality 
$ \lVert P _{+} (\sigma f) \rVert_{w} \lesssim \lVert f\rVert_{\sigma }$ holds if and only if the pair of weights satisfy the Poisson $ A_2$ condition \eqref{e:A2}, and these testing inequalities hold, uniformly over all intervals $ I$, for a finite positive constant $ \mathscr P$, 
\begin{equation*}
\int _{I} \lvert  P _{+} (\sigma \mathbf 1_{I})\rvert ^2 \; w (dx) \le \mathscr P ^2 \sigma (I) \,, 
\qquad 
\int _{I} \lvert  P _{+} (w \mathbf 1_{I})\rvert ^2 \; \sigma  (dx) \le \mathscr P ^2 w (I) \,. 
\end{equation*}
\end{theorem}

One must be sure that the $ A_2$ inequality is necessary from the norm inequality.  As it suffices to test real-valued functions, the real-variable proof 
given here will suffice.  
This in particular shows that  for the densities of the weights, $ \sigma (x) \cdot w (x) \le \mathscr A_2  $, for a.e.$ x$. Thus, the identity part of 
the norm, and testing, inequalities are trivial.  The remaining parts just concern the Hilbert transform, so one can use the main result. 

If one is interested in the Sarason question for functions $ f,g$ that are not outer, there is no simple reduction to the two weight inequality 
for the Hilbert transform, and the problem is quite subtle, as the role of the multiplier $ P_+ M _{\overline f}$ is more involved than that of just a weight.

\subsection{Model Spaces}

For a  probability measure $ \sigma $ on $ \mathbb T $, define a holomorphic function $ \theta $ on $ \mathbb D $ by the Poisson integral 
\begin{equation*}
\frac 1 {1 - \theta (z)} \coloneqq  \int _{\mathbb T } \frac  1 {1 - z \overline \zeta } \; \sigma (d \zeta )\,. 
\end{equation*}
This is an inner function: A holomorphic map of $ \mathbb D $ to itself which is unimodular a.e.\thinspace on $ \mathbb T $.  Also, $ \theta (0)=1$.  
(The measure $ \sigma $ is a Clark measure for $ \theta $, frequently written as  $ \sigma _1$.)

The shift operator $ S f (z) = z f(z)$ on $ H ^2 $ has invariant subspace $ \theta H ^2 = \{ \theta f \::\: f\in H ^2 \}$, whence 
$ K _{\theta } \coloneqq  H ^2 \ominus \theta H ^2 $ is invariant for $ S ^{\ast} $.  
Beurling's theorem states that every invariant subspace for $ S ^{\ast} $ is of this form.  
The model operator is $ S _{\theta } \coloneqq  P _{\theta } S$, where $ P _{\theta }$ is the orthogonal projection from $ H ^2 $ onto $ K _{\theta }$.  
Remarkably, subject to mild conditions, every contractive operator on a Hilbert space is unitarily equivalent to a properly chosen $ S _{\theta }$.  For this, and other reasons, properties of the $ K _{\theta }$ spaces have broad significance.  

The spaces $ K _{\theta }$ and $ L ^2 (\sigma )$ are unitarily equivalent, with the unitary  map  from $ f \in L ^2 (\sigma )$ to $ F\in K _{\theta }$ given by  
\begin{equation*}
F (z) = (1- \theta (z)) \int _{\mathbb T } \frac {f (\zeta )} {1- z \overline \zeta } \sigma (d \zeta ) \,. 
\end{equation*}
One is interested in those measures $ \mu $ on $ \mathbb T $ for which the natural embedding operator 
is bounded from $ K _{\theta }$ to $ L ^2 (\mu )$, namely, is it the case that $ \lVert F\rVert_{ \mu } \lesssim \lVert F\rVert_{ K _{\theta }}$.  
We see that this bound is equivalent to  
\begin{equation*}
\int _{\mathbb T } \Bigl\lvert \int _{\mathbb T } \frac {f (\zeta )} {1- z \overline \zeta } \sigma (d \zeta )  \Bigr\rvert ^2 
\lvert  1 - \theta (z)\rvert ^2 \mu (dz) \lesssim \lVert f\rVert_{\sigma } ^2 \,.  
\end{equation*}
That is, the question is equivalent to a two weight inequality for the Hilbert transform on $ \mathbb T $. 

From this perspective, one can lift counterexamples concerning the two weight Hilbert transform to those for embedding operators, which is the tactic of \cite{NV}, from which we have taken this condensed presentation.  
A characterization of the embedding question can be read off from our main theorem.   

But note that Clark measure is on $ \mathbb T $, by definition, and the second measure $ \mu $ is constrained to be supported on $ \mathbb T $, whereas the disk would be the natural assumption. 
In the case where $ \mu $ is supported on the disk, and one seeks an \emph{isometric} embedding, the question has 
a remarkable answer, found by Aleksandrov \cite{MR1464420}.  The general question is resolved in \cite{13104820}, 
which gives a characterization of a two weight inequality for the Cauchy transform, under these restrictions on the supports of the weights.  The method of the proof is similar to that of the Hilbert transform, with some additional complications.  

The model spaces are also important to spectral theory, and the subject of rank one perturbations of a unitary operator.  
In spectral theory, it is important to understand the structure of the unitary operator that sends the Hilbert space to into 
$ L ^2 $ of the spectral measure. Weighted Hilbert transforms arise therein.  See for instance \cite{MR1945291}, which 
uses the example of Nazarov showing that the $ A_2$ condition is not sufficient for the boundedness of the Hilbert transform. 
Also see \cite{MR2540995}. 

We point the interested readers to \cites{MR827223,MR2198367}, and the many citations therein for more information about these subjects. 

\subsection{de Branges Spaces}

We recall the setting of \cites{MR2812502,MR2763007}. 
For a sequence of distinct points $ \Gamma =\{ \gamma _n\} \subset \mathbb C $ and a sequence of positive numbers $ v = \{v_n\}$ 
consider the Cauchy transform 
\begin{equation*}
H _{(\Gamma ,v)} \::\: 
 a=\{a_n\} \mapsto \sum_{n \::\: z\neq \gamma _n} \frac {a_n v_n} {z- \gamma _n}
\end{equation*}
This is well defined for $ a \in \ell ^{2} _{v}$ and $ z \in \Omega $, defined by 
\begin{equation*}
\Omega \coloneqq  \Bigl\{ z\in \mathbb C \::\:  \sum_{n \::\: z\neq \gamma _n} \frac {v_n} {\lvert  z- \gamma _n\rvert ^2  } < \infty  \Bigr\}\,. 
\end{equation*} 

Call $ \mathcal H (\Gamma ,v)$ the space of functions analytic on $ \Omega $ given by the image of $ \ell ^2 _v$ under $ H _{(\Gamma ,v)}$. 
For appropriate choices of $ (\Gamma ,v)$, these Hilbert spaces have deep connections to analytic function spaces.  
For instance,  the reproducing kernels of $ \mathcal H (\Gamma ,v)$ are 
\begin{equation*}
k _{ z} (\zeta ) \coloneqq  \sum_{n} \frac {v_n} {(\overline  {z - \gamma _n}) (\zeta - \gamma _n)}\,, \qquad z\in \Omega \,. 
\end{equation*}
And, many natural questions, such as the structure of frames of reproducing kernels for $\mathcal H (\Gamma ,v) $, 
require knowledge about the two weight inequality for the Cauchy transform. 
For instance, the main real-variable result in \cite{MR2812502} is a characterization of a two weight inequality,
but under the requirement that both measures be a sum of point masses on sparse collections of points.   
This yields interesting results in the setting of de Branges spaces.  

The definition of $ \mathcal H (\Gamma ,v)$ provides just one possible representation of a de Branges space, a class of Hilbert spaces with remarkable properties.  
The standard reference for them is \cite{MR0229011}.  Beginning from the works of Sarason \cites{MR1038352}, they have 
become an essential part of subject of analytic function spaces.

\end{document}